\documentclass[a4paper;11pt]{amsart}
\usepackage{mathrsfs}
 \usepackage[arrow,matrix]{xy}
\usepackage{amsmath,amssymb,amscd,bbm,amsthm,mathrsfs}
\newtheorem{thm}{Theorem}[section]
\newtheorem{lem}{Lemma}[section]
\newtheorem{cor}{Corollary}[section]
\newtheorem{prop}{Proposition}[section]
\newtheorem{rem}{Remark}[section]

\theoremstyle{definition}

\begin{document}
\numberwithin{equation}{section}

 \title[On quaternionic  complexes over unimodular   quaternionic   manifolds ]{
On quaternionic  complexes over  unimodular  quaternionic   manifolds    }
\author {Wei Wang}
\begin{abstract}Penrose's two-spinor notation for $4$-dimensional Lorentzian manifolds is extended to two-component   notation for
quaternionic   manifolds,
which is a    useful tool for calculation. We can construct
 a family of quaternionic  complexes over unimodular   quaternionic   manifolds only by elementary calculation. On complex
 quaternionic manifolds  as
 complexification  of
quaternionic  K\"ahler manifolds, the existence of these complexes was established by Baston by using twistor transformations  and
spectral sequences. Unimodular   quaternionic
manifolds constitute a  large nice class of quaternionic  manifolds: there exists a very special curvature
decomposition; the   conformal change of
a unimodular   quaternionic structure is still unimodular   quaternionic; the complexes over such manifolds are conformally invariant.
This class of manifolds is the real version of torsion-free QCFs  introduced by Bailey   and   Eastwood. These
complexes  are elliptic. We also obtain a Weitzenb\"ock  formula to establish   vanishing of the cohomology groups of these
complexes
for quaternionic  K\"ahler manifolds with negative scalar curvatures.
\end{abstract}

\thanks{Supported by National Nature Science
Foundation
  in China (No. 11571305)\\   Department of Mathematics,
Zhejiang University, Zhejiang 310027,
 P. R. China, Email:   wwang@zju.edu.cn}
 \maketitle

\section{Introduction}
Quaternionic   manifolds  are important in supersymmetric theory in physics, in particular in nonlinear   sigma models. It is well known
that the supersymmetric  sigma models are deeply
related to geometries with complex structures: the target   manifold must be a K\"ahler manifold      in $ N=1$
theories; it must be a hyperK\"ahler manifold    in
   rigid $ N=2$ supersymmetric theories; while in local supersymmetric $ N=2$   theories, it must be a quaternionic K\"ahler manifold of
   negative curvature, etc.. The quaternionic complexes over quaternionic K\"ahler manifolds were used to investigate  $N=2$
   supersymmetric black holes recently   \cite{CLW}. The
  Baston operator in the these complexes appears in
quantization of $ N=2$ supergravity
black holes \cite{NPV}. Physicists are also  interested in
supersymmetric and superconformal theory over more general   quaternionic manifolds  \cite{BdHCGV}.  The geometry
of quaternionic   manifolds   is an active direction of research in last four  decades  (cf. e.g.   \cite{Ba}  \cite{CS2} \cite{HKLR}
\cite{IM} \cite{Sa} \cite{Sa86} \cite{Sw} \cite{Ver} and references therein).

Recall that an {\it  almost quaternionic structure}  on a manifold $M$ is a rank-$3$ subbundle of End$ TM$
which is locally spanned by  three almost complex structures  on $TM$ satisfying the commutating relation of quaternions,  i.e. the
frame
bundle of $M$   reduces   to a principal bundle $P$ with structure group ${\rm {GL}}(n,\mathbb{H}){\rm {Sp}}(1)\cong {\rm
{GL}}(n,\mathbb{H})\times_{\mathbb{Z}_2}{\rm {Sp}}(1)$, where ${\rm {Sp}}(1) $ is the Lie group  of right multiplying unit
quaternions.  It is a  {\it    quaternionic  manifold}   if there exists a torsion-free connection
on $P$. It is called   {\it quaternionic K\"ahler} if the Levi-Civita connection for the metric  preserves
the   quaternionic structure, i.e.    the frame
bundle of $M$   reduces   to    a  principal  ${\rm {Sp}}(n ){\rm {Sp}}(1)$-bundle with a torsion-free connection.
   A     quaternionic  manifold    $M$ is
 called   {\it unimodular}    if the quaternionic  connection   preserves  a volume form  on $M$, i.e.  the frame
bundle of $M$ reduces  to    a  principal  ${\rm {SL}}(n,\mathbb{H}){\rm {Sp}}(1)$-bundle with a torsion-free connection.

Given a representation $W$ of $ {\rm {GL}}(n,\mathbb{H})\times {\rm {Sp}}(1)$ (a double covering of ${\rm {GL}}(n,\mathbb{H}){\rm
{Sp}}(1)$), choose a lift of $P$ to a  principal  ${\rm {GL}}(n,\mathbb{H})\times {\rm {Sp}}(1)$-bundle $\widetilde{P}$. Then we can
define the associated bundle $ \widetilde{P} \times_{{\rm {GL}}(n,\mathbb{H})\times{\rm {Sp}}(1)} W$. Such a lifting       always exists
locally, and exists globally   when the obstruction to the lifting in  $H^2(M,\mathbb{Z}_2)$ vanishes, e.g. when it is a
$8n$-dimensional quaternionic K\"ahler manifold, (cf. section 2 in \cite{Sa} and section 2 in \cite{Sa86}). In the sequel, we assume
that  such a lifting       always exists.
  Taking the standard ${\rm {GL}}(n,\mathbb{H})$-module $\mathbb{C}^{2n}$ and ${\rm {Sp}}(1)$-module $\mathbb{C}^{2 }$, we have
  associated vector bundles
 \begin{equation}\label{eq:associated}
  E:= \widetilde{P} \times_{{\rm {GL}}(n,\mathbb{H})\times {\rm {Sp}}(1)}\mathbb{C}^{2n},\qquad H  := \widetilde{P} \times_{{\rm
  {GL}}(n,\mathbb{H})\times{\rm {Sp}}(1)}\mathbb{C}^{2 },\qquad  T  M   \cong  \widetilde{P} \times_{{\rm {GL}}(n,\mathbb{H})\times {\rm
  {Sp}}(1)}\mathbb{H}^{ n } .\end{equation}
  A connection on the principal bundle $ {P}$ is trivially lifted to the principal bundle $\widetilde{P}$, and so induces    connections
  on associated vector bundle
$E $,  $H$ and $T  M$, respectively.
It is well known \cite{Sa} \cite{Sa86} that
the complexified tangent  bundle of an
almost quaternionic manifold $M$   decomposes as the tensor product
\begin{equation}\label{eq:identification0}
  \mathbb{C} TM\cong E\otimes H.
 \end{equation}
Denote by $\Gamma( V)$   the space of smooth sections of a vector bundle $V$ over $M$.  In this paper we will discuss the construction
of  a family of elliptic differential
 complexes
over  a unimodular   quaternionic   manifold $M$ ($M$ is right conformally flat when  $\dim_{\mathbb{R}}M=4$):
\begin{equation}\label{eq:quaternionic-complex-diff}\begin{split}
0\longrightarrow\Gamma\left( \odot^{k }{H}^* \right) &
\xrightarrow{D_0^{(k)}}\Gamma \left(  \Lambda^1 {E}^*\otimes\odot^{k-1
}{H}^*\right)
\xrightarrow{D_1^{(k)}}\cdots\cdots\longrightarrow \Gamma\left (  \Lambda^k
{E}^*\right) \xrightarrow{D_k^{(k)}} \Gamma (
\Lambda^{k+2} {E}^*)\\& \xrightarrow{D_{k+1}^{(k)}}\Gamma \left(
\Lambda^{k+3} {E}^*\otimes{H}\right) \longrightarrow\cdots\cdots
\xrightarrow{D_{2n-2}^{(k)}} \Gamma\left(  \Lambda^{2n}
{E}^*\otimes
\odot^{2n-k-2}{H}\right )\longrightarrow
0,\end{split}
\end{equation}$k=0,1,\ldots$,
  where $\Lambda^{q} {E}^*$ is the $q$-th exterior product of $ {E}^*$, and $\odot^{p}{H}^*$ is the $p$-th symmetric product of $ H^*$.
  The first operator  $D_0^{(k)}$ is  called the
{\it $k$-Cauchy-Fueter operator}.  $D_j^{(k)}$'s are   differential operators of the first order  except   $D_k^{(k)}$, which is of the
second order (cf. Theorem \ref{thm:k-Cauchy-Fueter}).
By using the twistor transformation  and spectral sequences,
Baston \cite{Ba} proved the existence of these   complexes over complex quaternionic manifolds. He generalized the result and the
complex geometric method of   Eastwood, Penrose  and Wells
\cite{EPW} for $n=1$. A {\it complex quaternionic manifold} $\widetilde{M} $ with $\dim_{\mathbb{C}}\widetilde{M} =4n$ is a complex
manifold whose  holomorphic tangential
bundle  decompose as
$
    T\widetilde{M} \cong \widetilde{E} \otimes\widetilde{ H },
$
 where $\widetilde{E} $  and $ \widetilde{H} $ are holomorphic  vector bundles of dimension $2n$ and  $2$, respectively, and there
 exists a torsion-free holomorphic   connection preserving  symplectic forms $\epsilon$ on  $\widetilde{E} $  and $\varepsilon$ on $
 \widetilde{H} $. Then
 $\epsilon\otimes\varepsilon$ is a complex Riemannian metric on $\widetilde{M} $. Baston used Levi-Civita connection on $\widetilde{M}$
 to construct these complexes  after section 2.6 in his paper \cite{Ba}. So he constructed   complexes over
 quaternionic  K\"ahler manifolds.
 Because the twistor transformation is a   complicated technique in complex geometry, it is  interesting to construct   complexes
 (\ref{eq:quaternionic-complex-diff})  by only using elementary method of differential geometry.

Another motivation to consider these complexes comes from the function theory of several quaternionic variables. We write a vector  in
the quaternionic space $\mathbb H^{n}$ as
$\mathbf{q}=(\mathbf{q}_{0},\ldots,\mathbf{q}_{n-1})$
 with $\mathbf{q}_{l}=x_{4l+1}+x_{4l+2}\mathbf{i}+x_{4l+3}\mathbf{j}+x_{4l+4}\mathbf{k}\in \mathbb H$,
 $l=0,1,\ldots,n-1.$
 The  usual {\it Cauchy-Fueter
operator} is
defined as
\begin{equation*} \label{eq:CF}  \mathscr  D  :C^1( \mathbb H^{n}, \mathbb H )\rightarrow C (\mathbb H^{n}, \mathbb H^{n}),\qquad \qquad
\mathscr D f=\left(\begin{array}{c} \overline\partial_{
{\mathbf{q}}_0}f\\\vdots\\  \overline\partial_{ {\mathbf{q}}_{n-1  }}f\end{array}\right),
\end{equation*}
for $f\in C^1(\mathbb H^{n}, \mathbb H )$,  where $
 \overline{\partial}_{\mathbf{q}_{l}}
 =\partial_{x_{4l+1}}+\mathbf{i}\partial_{x_{4l+2}}+\mathbf{j}\partial_{x_{4l+3}}+\mathbf{k}\partial_{x_{4l+4}},
$
$l=0,1,\ldots,n-1$.
 A function $f:\mathbb H^{n}\rightarrow \mathbb H$ is called {\it (left) regular} if
$
  \mathscr  D  f\equiv 0  $
  on $   \mathbb H^{n}$. As in the theory
of several complex variables,  to investigate   regular
quaternionic  functions, it is important to solve the {\it
non-homogeneous Cauchy-Fueter equation}
\begin{equation}\label{eq:Cauchy-Fueter-equations} \mathscr D f=  h
\end{equation}
for prescribed $  h\in C^2(\Omega, \mathbb{H}^n)$ over a domain $\Omega$. This system is overdetermined, i.e. the number of equations
is larger than the number of unknown functions for $n>1$. So for (\ref{eq:Cauchy-Fueter-equations}) to be solvable, $  h$ must
satisfy  some compatible condition. This condition was found by
Adams,
 Loustaunau,   Palamodov  and   Struppa    \cite{adams2} with the help of computer algebra method, namely,
\begin{equation*}\label{eq:compatible}\mathscr D_1  h=0,
\end{equation*}
for some differential operator   of second
order $\mathscr D_1:C^2(\Omega,\mathbb{H}^{ n})\rightarrow C(\Omega,
\Lambda^3\mathbb{H}^{ n})$. In fact, there exists a differential complex
corresponding to the Dolbeault complex in    the theory of several
complex variables:
\begin{equation}\label{eq:Q-complex}
0\rightarrow \Gamma(\Omega, \mathbb{H} )\xrightarrow {\mathscr D
}\Gamma(\Omega,  \mathbb{H}^{ n}) \xrightarrow {\mathscr D_1
}\Gamma(\Omega, \Lambda^3\mathbb{H}^{ n})\rightarrow\cdots
 \cdots \rightarrow 0,
\end{equation}
  called the  {\it Cauchy-Fueter complex}
 (cf. \cite{CSS} \cite{CSSS} and references therein). It was realized later   that the Cauchy-Fueter operator
 is exactly the $1$-Cauchy-Fueter operator  and  the Cauchy-Fueter complex \cite{bures} \cite{CSS} is equivalent   to a sequence
 obtained by Baston in \cite{Ba}, although
 Baston's result
 is a complexified version. In \cite{Wa10}, the author introduced   notions of the $k$-Cauchy-Fueter operator  on the quaternionic space
 $\mathbb{H}^n$  and {\it $k$-regular functions} annihilated by this operator. In the function theory,
 (\ref{eq:quaternionic-complex-diff}) is  called the {\it $k$-Cauchy-Fueter  complex}.
 The $k$-Cauchy-Fueter  complex  over   $\mathbb{H}^n$ was explicitly written down by using the twistor transformation  and spectral
 sequences \cite{Wa10} (see also \cite{bS}).
 By solving the non-homogeneous $k$-Cauchy-Fueter equations, we showed the Hartogs'  phenomenon  for
  $k$-regular functions \cite{Wa10}. To develop the function theory over curved    manifolds, we need to write down these complexes  on
  manifolds
  explicitly.

  The $k$-Cauchy-Fueter operator over the  $1$-dimensional quaternionic  space $\mathbb{H}$ also has the origin in physics: it is
the elliptic version of {\it spin $  k/ 2$  massless field} operator \cite{CMW} \cite{EPW} \cite{PR1} \cite{PR2} over the
Minkowski space:
    $D_0^{(1)}\phi=0$ corresponds to the Dirac-Weyl equation   whose solutions correspond to neutrinos;
  $D_0^{(2)}\phi=0$ corresponds to  the  Maxwell equation   whose  solutions correspond to
photons;
  $D_0^{(3)}\phi=0$ corresponds to   the Rarita-Schwinger
equation;
   $D_0^{(4)}\phi=0$ corresponds to   linearized Einstein's equation
 whose  solutions correspond to  weak gravitational fields, etc..

Salamon  \cite{Sa86}   constructed another family of  quaternionic  complexes over   quaternionic manifolds:
 \begin{equation}\label{eq:Salamon}\begin{split}
0\longrightarrow \Gamma\left( \odot^{k }{H}^* \right) &
\longrightarrow\Gamma \left(  \Lambda^1 {E}^*\otimes\odot^{k+1
}{H}^*\right)
\longrightarrow\cdots\cdots\longrightarrow\Gamma\left (  \Lambda^p
{E}^*\otimes\odot^{k+p
}{H}^*\right) \\&\longrightarrow \Gamma (
\Lambda^{p+1} {E}^*\otimes\odot^{k+p+1
}{H}^*)  \longrightarrow\cdots\cdots
.\end{split}
\end{equation} The half sequence of  the $k$-Cauchy-Fueter  complex  (\ref{eq:quaternionic-complex-diff}),  beginning with the operator
$D_{k+1}^{(k)}$, is similar to Salamon's complexes. In last two decades, quaternionic manifolds were also studied from the point of view
of parabolic geometry (cf. \cite{CS09} \cite{CS} \cite{CS2} \cite{SS} \cite{CSS} and reference therein). Several interesting
differential complexes over curved quaternionic manifolds have been constructed from   BGG-sequences \cite{CSS01} \cite{CS2}. Recall
that for a parabolic subalgebra $\mathfrak p$ (resp. subgroup $P$)  of a complex semisimple Lie algebra $\mathfrak g$ (resp. group $G$),
let $E(\lambda)$ be the irreducible $\mathfrak p$-module with the lowest weight $-\lambda$. Denote by $\mathscr O_{\mathfrak
p}(\lambda)$ the sheaf of holomorphic sections of vector bundle associated  to $E(\lambda)$ over $G/P$.  A    {\it general
BGG-sequence}
 is an exact sequence
\begin{equation}\label{eq:BGG}0 \longrightarrow E_{\mathfrak g}(\lambda)
 \longrightarrow \mathscr O_{\mathfrak p}(\lambda)   \xrightarrow {\hskip 2mm d_0\hskip 2mm } \bigoplus_{w\in W^{\mathfrak p},
 l(w)=1}\mathscr O_{\mathfrak p}(w.\lambda) \xrightarrow {\hskip 2mm d_1\hskip 2mm } \bigoplus_{w\in W^{\mathfrak p}, l(w)=2}\mathscr
 O_{\mathfrak p}(w.\lambda) \longrightarrow \cdots,
\end{equation}for a dominant weight $\lambda$ of $\mathfrak g$, where   $W^{\mathfrak p}$  is the Hasse diagram associated to $\mathfrak
p$ (cf.  theorem 8.4.1 in  \cite{BE}). $E_{\mathfrak g}(\lambda)$ is a   finite dimensional irreducible representation of $\mathfrak g$. But on the
flat space $\mathbb{H}^n$, the $k$-Cauchy-Fueter complex after complexification is
  a sequence (\ref{eq:BGG}) with the weight $\lambda $ singular for $\mathfrak g$, but dominant  for $\mathfrak p$ (cf.  theorem 11 in
  \cite{Ba}). In this case, $E_{\mathfrak g}(\lambda)$ is an infinite dimensional irreducible representation of $\mathfrak g$. So it is not a
  BGG-sequence. Moreover, $D_k^{(k)}$  is an example of    non-standard  invariant operators (cf. Remark 12 in \cite{BE}). In general,
  it is not easy to construct an exact sequence with singular weights. People usually   construct such a sequence  from  a relative  BGG
  sequence,   case by case, by using the twistor method (cf. e.g. \cite{BE} \cite{PS} \cite{G2} and references therein)  or the method
  of cohomology parabolic induction in the representation theory (cf. section 11.3 of \cite{BE}).
From the point of view of    function theory, we are especially interested in differential complexes   (\ref{eq:BGG}) with singular
weights,  because
 only  in this case   ``regular functions" as elements of  $\ker d_0\cong  E_{\mathfrak g}(\lambda)$  are abundant. On the flat space $\mathbb{H}^n$,
 a generalized Penrose integral formula provides all solutions to
 the   $k$-Cauchy-Fueter equation, which is of infinite dimensional (cf.  \cite{KW}).

The $0$-Cauchy-Fueter operator $D_0^{(0)}:\Gamma (M, \mathbb{C}) \rightarrow  \Gamma (M,
\Lambda^{ 2} {E}^*)$ is   called the {\it Baston operator}. Certain exterior product of this operator gives us the quaternionic
Monge-Amp\`ere
operator by Alesker
\cite{alesker2}. This interpretation together with the second operator $D_1^{(0)}$ allows us to develop pluripotential theory over
$\mathbb{H}^n$ \cite{Wan} \cite{WaK} \cite{WaZ} \cite{WaWang}. To develop pluripotential theory over curved
quaternionic   manifolds, in particular to study the quaternionic  Calabi-Yau problem on quaternionic   manifolds \cite{alesker7}
   \cite{alesker6}, we need to know $0$-Cauchy-Fueter complex on    manifolds explicitly.

Penrose's two-spinor notation is   useful for studying $4$-dimensional manifolds \cite{PR1}  \cite{PR2}. It is generalized to  complex
quaternionic manifolds by Baston \cite{Ba} and  to more general complex paraconformal manifolds by   Bailey   and   Eastwood
\cite{BaiE}. As a real version, we   extend
this
notation to quaternionic   manifolds simply by  realizing the isomorphism
$
 \mathbb{ C}TM\cong E\otimes H
$ in (\ref{eq:identification0}): for local frames $\{e_{A }\}$ and $\{e_{A'}\}$   of $E$¡¡ and $H$, respectively,
 we identify $e_A\otimes e_{A'}$  with a complex tangential vector
$
Z_{AA'}
$ (see section \ref{sub:Realization}).
 The  quaternionic   connection on $M$ induces a $\mathfrak {gl} (2n,\mathbb{ C } )$-connection  on $E$ and a $\mathfrak {su} (2
 )$-connection  on $H$, respectively, and so the curvature of
the quaternionic   connection has two components
\begin{equation*}
  R_{abA}^{\phantom{abA}B}\qquad {\rm and}\qquad  R_{abA'}^{\phantom{abA'}B'},
\end{equation*}
corresponding to  curvatures of the bundles $E$ and $H$, respectively. Here we use indices $A,B$ and $A',B'$ for components of sections
of bundles  $E$  and $
H$, respectively, and indices $a, b$ for components of the local quaternionic  frame of the tangent bundle $TM$. Furthermore,
curvatures of a unimodular quaternionic   connection  have a very special
decomposition (cf. Proposition \ref{prop:curvature}), with the help of which we can check that the sequence
(\ref{eq:quaternionic-complex-diff})
is a complex, i.e.
$
   D_{j+1}^{(k)}\circ   D_j^{(k)}=0,
$ by direct calculation in Section 3.1.  Two-component   notation is a   useful tool for calculation over a    quaternionic   manifold
and everything in this paper is
based on elementary
 calculation by this   notation.
 Unimodular   quaternionic  manifolds constitute a nice class of quaternionic  manifolds, because the
conformal change of
a unimodular   quaternionic structure is still unimodular   quaternionic, while  the conformal change of
a  quaternionic K\"ahler structure is usually not quaternionic K\"ahler (cf. \cite{OP}).  We also give the conformal transformation
formula of these operators
$D_j^{(k)}$  in Section 3.1.

In Section 3.2, we show that the $k$-Cauchy-Fueter complex is elliptic, i.e. its symbol complex is a exact sequence of complex vector
spaces. Write the $k$-Cauchy-Fueter complex  as
\begin{equation}\begin{split}
0\longrightarrow & \Gamma \left( \mathscr V_0^{(k)} \right)
\xrightarrow{D_0^{(k)}} \Gamma \left( \mathscr V_1^{(k)}\right)
\xrightarrow{D_1^{(k)}}\cdots\longrightarrow \Gamma\left ( \mathscr V_{2n-1}^{(k)}\right)
 \longrightarrow
0,\end{split}\label{eq:quaternionic-complex-diff-V}
\end{equation}  where $\mathscr V_j^{(k)}$ is the  $j$-th vector space in the sequence (\ref{eq:quaternionic-complex-diff}). By  the
theory of elliptic operators, we   know the     Hodge-type
decomposition
and that the   $j$-th cohomology group
\begin{equation*}
    H^j_{ (k) }(M)=  \ker D_j^{(k)} / {\rm Im}\, D_{j-1}^{(k)}
\end{equation*}
 of the $k$-Cauchy-Fueter complex  over a compact unimodular   quaternionic   manifold (right conformally flat if
 $\dim_{\mathbb{R}}M=4$) is finite dimensional, and  can be represented by Hodge-type elements.

 In Section 4, we prove a
 Weitzenb\"ock formula  for these complexes over a quaternionic  K\"ahler manifold $M$,  and show a vanishing theorem for the
 cohomologies $H^j_{ (k) }(M)$, $j=1,\ldots,k-1$, if its  scalar curvature
 is negative. The Weizenb\"ock formula and vanishing theorem  for Salamon's complexes over such manifolds with negative scalar
 curvatures were already given by Horan  \cite{Ho}  (see also Homma \cite{Hom} and Nagatomo-Nitta \cite{NN}). The latter one essentially
 gives us  the result for the $k$-Cauchy-Fueter  complex for $j\geq k+3$.

I would like to thank the referee for many valuable suggestions.

 \section{ Unimodular  quaternionic manifolds  and their curvatures}
 \subsection{Realization of
the isomorphism
$
 \mathbb{ C}TM\cong E\otimes H
$}\label{sub:Realization}Denote by $\text{GL} (n,\mathbb{H})$ the group of all invertible quaternionic $(n\times
n)$-matrices. ${\rm {Sp}}(n) :=\{A\in\text{GL}(n,\mathbb{H});\overline{A}^tA=A\overline{A}^t=I_{n\times n}\}$.
$\mathfrak {sl}(n,\mathbb{H}):=\{A\in\mathfrak{gl}(n,\mathbb{H}); {\rm Re }\, (tr(A))=0\}$. We denote by ${\rm {SL}}(n,\mathbb{H})$ the
connected component containing the identity of the Lie group with Lie algebra   $\mathfrak {sl}(n,\mathbb{H})$.

Let $A=(A_{jk})_{p\times m}$ be a quaternionic $(l\times m)$-matrix and write
$
    A_{jk}=a_{jk}^1+\textbf{i}a_{jk}^2+\textbf{j}a_{jk}^3+\textbf{k}a_{jk}^4\in\mathbb{H}.
$
We define $\tau(A)$ to be the complex $(2p\times 2m)$-matrix \begin{equation}\label{tau}\tau(A)=\left(
                                                              \begin{array}{cccc}
                                                                \tau(A_{00}) & \tau(A_{01}) & \cdots & \tau(A_{0(m-1)}) \\
                                                                \tau(A_{10}) & \tau(A_{11}) & \cdots & \tau(A_{1(m-1)}) \\
                                                                \cdots & \cdots & \cdots & \cdots \\
                                                                \tau(A_{(p-1)0}) & \tau(A_{(p-1)1}) & \cdots & \tau(A_{(p-1)(m-1)}) \\
                                                              \end{array}
                                                            \right),
\end{equation}where $\tau(A_{jk})$ is the complex $(2\times2)$-matrix
\begin{equation}\label{2.301}\left(
                                       \begin{array}{cc}
                                         a_{jk}^1+\textbf{i}a_{jk}^2 & -a_{jk}^3-\textbf{i}a_{jk}^4 \\
                                         a_{jk}^3-\textbf{i}a_{jk}^4 & \quad a_{jk}^1-\textbf{i}a_{jk}^2 \\
                                       \end{array}
                                     \right).\end{equation}
                                      This is motivated by the embedding of quaternionic  numbers into $2\times 2$-matrices. The
                                      definition of $\tau$ above and the following proposition
are the conjugate version of those in \cite{WaWang}.

\begin{prop}\label{p2.1} {\rm (proposition 2.1 in \cite{WaWang})} $(1)$ $\tau(AB)=\tau(A)\tau(B)$ for a quaternionic $(p\times
m)$-matrix $A$ and a quaternionic $(m\times l)$-matrix $B$. In particular, for $q'=Aq,$ $q,q' \in\mathbb{H}^n$, $A\in$
$\text{GL}(n,\mathbb{H})$, we have
\begin{equation}\label{eq:Aq}
 \tau(q')=\tau(A)\tau(q)
\end{equation}
 as complex $(2n\times2)$-matrices.\\
$(2)$ For $A\in\text{GL}(n, \mathbb{H})$, we have \begin{equation}\label{2.231}J\overline{\tau(A)}=\tau(A)J, \qquad {\rm where}\qquad
J=\left(
                                       \begin{array}{ccccc}
                                       0 & 1 &   &   &   \\
                                         -1 & 0 &   &   &   \\
                                           &   & 0 & 1 &   \\
                                           &   & -1 & 0 &   \\
                                           &   &   &   & \ddots \\
                                            \end{array}
                                     \right).
\end{equation}
$(3)$  $\tau\left(\overline{A}^t\right)=\overline{\tau(A)}^t$ for a quaternionic $(n\times n)$-matrix $A$. If $A\in{\rm {Sp}}(n)$,
$\tau(A)$ is symplectic, i.e., $\tau(A)J\tau(A)^t=J$.
\end{prop}

  Proposition \ref{p2.1} (3) implies $\tau({\rm {Sp}}(n))\subset{\rm {SU}}(2n )$. (\ref{tau})-(\ref{2.301}) implies $\tau( \mathfrak
  {sl}(n,\mathbb{H}))\subset\mathfrak {sl}(2n,\mathbb{C})$, and so
 $\tau({\rm {SL}}(n,\mathbb{H}))\subset{\rm {SL}}(2n,\mathbb{C})$. Given the standard  volume form on $\mathbb{R}^{4n}$, ${\rm
 {SL}}(n,\mathbb{H})$ is the group consisting of  elements of $
{\rm {GL}}(n,\mathbb{H})$ which induce  transformations of $\mathbb{R}^{4n}$ preserving  this volume form.
Let  $I_{1},I_{2},I_{3}$  be the induced action of $\mathbf{i},\mathbf{j},\mathbf{k}$ on the frame bundle. Then
  we can choose a    frame  of the tangent bundle
\begin{equation}\label{eq:admissible}
   (X_1,X_1 I_{1},X_1 I_{2},X_1 I_{3}, \ldots,  X_{4l+1},X_{4l+1}I_{1},X_{4l+1}I_{2},X_{4l+1}I_{3}, \ldots)
\end{equation}
  called a {\it local quaternionic  frame}. Label this frame  as $(X_1,\ldots, X_{4n})$ .

  $ H \otimes E$ is isomorphic to the tangent bundle $\mathbb{C}T  M $ as follows.
It follows from  Proposition \ref{p2.1} that
$\mathbb{C}^{2n }$ is a $\text{GL}(n,\mathbb{H})$-module with $A\in \text{GL}(n,\mathbb{H})$ acting on $\mathbb{C}^{2n }$ by $\tau(A)$,
and $\mathbb{C}^{2  }$ is a ${\rm {Sp}}(1)$-module with $q\in {\rm {Sp}}(1)$ acting on $(z_1,z_2)\in \mathbb{C}^{2  }$ by right
multiplying   the $2\times2$-matrix $\tau(q)$.
Let $\{v_A\}_{A=0}^{2n-1}$, $\{v_{A'}\}_{A'=0',1'}$ and $\{w_a\}_{a=1}^{4n}$ be the standard bases of $\mathbb{C}^{2n }$, $\mathbb{C}^{2
}$ and $\mathbb{R}^{4n }$, respectively.  Write $v_{AA'}:=v_A \otimes v_{A'}$ in $\mathbb{C}^{2n }\otimes\mathbb{C}^{2  }$. The map
$\tau$ provides an isomorphism from $\mathbb{C}\otimes\mathbb{H}^{ n }$ to $\mathbb{C}^{2n }\otimes\mathbb{C}^{2  }$ as ${\rm
{GL}}(n,\mathbb{H})\times{\rm {Sp}}(1)$-module. Under this identification of $\tau$, we have
\begin{equation*}\begin{split}
  2 w_{4l+1 }&=\hskip 3mm v_{(2l)0'}+ v_{(2l+1)1'},\qquad  2 w_{4l+2}=-\mathbf{i}v_{(2l)0'}+\mathbf{i} v_{(2l+1)1'},\\2
  w_{4l+3}&=-v_{(2l)1'}+ v_{(2l+1)0'},\qquad\, 2w_{4l+4}=\hskip 3mm \mathbf{i}v_{(2l)1'}+ \mathbf{i} v_{(2l+1)0'} ,
\end{split}\end{equation*}by using definition (\ref{tau})-(\ref{2.301}) of $\tau$ for $m=1$. Thus
\begin{equation}\label{eq:identification}
 \left(\begin{array}{ll} v_{(2l)0'}& v_{(2l)1'} \\v_{(2l+1)0'}&v_{(2l+1)1'}
                                         \end{array}
                                    \right)=\left(\begin{array}{rr}  w_{4l+1 }+  \mathbf{i} w_{4l+2} & -  w_{4l+3}-\mathbf{i}
                                    w_{4l+4}\\
                                      w_{4l+3}-\mathbf{i}  w_{4l+4} &  w_{4l+1 }- \mathbf{i} w_{4l+2}  \end{array}
                                    \right).
\end{equation}
Now for a local quaternionic  frame $\mathfrak e=(X_1,\ldots, X_{4n})$, define local sections
\begin{equation}\label{eq:frme}e_A:=(\mathfrak e, v_A),\qquad e_{A'}:=(\mathfrak e, v_{A'}),\qquad X_a:=(\mathfrak e, w_a)
\end{equation} of $E$, $H$ and $TM$ in (\ref{eq:associated}), respectively.
 Then (\ref{eq:identification}) implies that $Z_{AA'}=\frac 1{\sqrt 2} e_A\otimes e_{A'}$ are given by
\begin{equation}\label{eq:k-CF} \left(\begin{array}{cc}Z_{00'}&Z_{01'} \\
                                         \vdots& \vdots\\Z_{(2n-1)0'}&Z_{(2n-1)1'}\end{array}
                                    \right):=  \frac 1{\sqrt 2}\left(
                                      \begin{array}{ll}
                                     X_{1} +\textbf{i}X_{2}  & - X_{3} -\textbf{i}X_{4}  \\
                                       X_{3}-\textbf{i}X_{4}  &\hskip 3mmX_{1}-\textbf{i}X_{2}  \\
                                         \qquad\vdots&\qquad\vdots\\
                                      X_{4n-3} +\textbf{i}X_{4n-2}  & -X_{4n-1}-\textbf{i}X_{4n }  \\
                                       X_{4n-1} -\textbf{i}X_{4n }  &\hskip 3mm X_{4n-3} -\textbf{i}X_{4n-2}  \\
                                      \end{array}
                                    \right),\end{equation} See (\ref{eq:gab}) for the reason to choose factor $\sqrt 2$ here. This frame
                                    over the flat quaternionic space $\mathbb{H}^n$ plays an important role in the investigation  of
                                    quaternionic analysis  \cite{KW} \cite{Wan}-\cite{WR} \cite{Wa10}.

                                    Let $\{\omega^i\}$ be the coframe dual to $\{X_j\}$ and let $\{e^{AA'}\}$ be complex $1$-forms dual
                                    to the two-component   local quaternionic  frame $\{ Z_{A  A'}\}$ in (\ref{eq:k-CF}),
i.e. $ e^{AA'}(Z_{BB'})=\delta_{ B}^A\delta^{A'}_{B'}$.  It is obvious that
\begin{equation*}\begin{split}
  & \sqrt 2 e^{00'}= \omega^1-\mathbf{i}\omega^2 ,\qquad\sqrt 2 e^{01'}= -\omega^3+\mathbf{i}\omega^4 , \qquad
 \sqrt 2  e^{10'}=  \omega^3+\mathbf{i}\omega^4 ,\qquad \sqrt 2 e^{11'}= \omega^1+\mathbf{i}\omega^2 ,\cdots,
\end{split}\end{equation*} by the expression of $\{ Z_{A  A'}\}$ in (\ref{eq:k-CF}), and so $(e^{00'} \wedge  e^{11'})\wedge
  ( e^{10'} \wedge e^{01'} ) = -\omega^1\wedge\omega^2\wedge\omega^3\wedge\omega^4$. Consequently,
   \begin{equation}\label{eq:vol}vol:= (-1)^n\omega^1\wedge\cdots\wedge\omega^{4n}=
    \left ( \wedge_{A=0}^{2n-1} e^{A0'}\right)\wedge\left ( \wedge_{B=0}^{2n-1} e^{B1'}\right).
   \end{equation}
A local quaternionic  frame $\{X_1,\ldots, X_{4n}\}$ is called a {\it local unimodular quaternionic  frame} if the volume form of the
manifold is locally given by $vol$ in (\ref{eq:vol}). Note that a local quaternionic  frame becomes   unimodular simply by multiplying a
suitable factor.

  \subsection{The two-component   notation}
 Denote by $e^{A'}:=(\mathfrak e, v^{A'})$ a local section of the dual bundle
 $H^*$, where  $v^{A'}$ is the dual of $v_{A'}$ in $\mathbb{C}^{2  }$. It is similar to define $E^*$ and $e^{A  }$. Consider
  \begin{equation}\label{eq:varepsilon}
     \varepsilon_{A'B'}e^{A'}\otimes e^{B'}
  \end{equation}
  where $\varepsilon_{A'B'}  $ is antisymmetric with $ \varepsilon_{0'1'}=1$. Here and in the following, we use the Einstein's
  convention of summation over repeated indices. It  is a section of the line bundle $\Lambda^2
  H^* $ (a symplectic form on $H$ pointwisely).  This is because (\ref{eq:varepsilon}) is invariant under the action of ${\rm {Sp}}(1 )$
  by Proposition \ref{p2.1} (3). So they can be glued to be a global section.
When the manifold is  unimodular    quaternionic,  consider
 \begin{equation}\label{eq:epsilon-unprimed}
    \epsilon_{A_1\ldots A_{2n}  }e^{A_1}\otimes\ldots \otimes e^{A_{2n}}
  \end{equation}
  where $\epsilon_{A_1\ldots A_{2n}  }$ is the sign of the permutation from $ A_1,\ldots, A_{2n} $ to $1,\ldots, 2n$. It is a global
  section  of the line bundle
  $\Lambda^{2n}  E^*$,  because (\ref{eq:epsilon-unprimed}) is invariant under the action of ${\rm {SL}}(2n, \mathbb{C} )$.

A section $  f$ of $\mathfrak T^{ l}_{q,p}:=  (\otimes^l H)\otimes ( \otimes^q E^*)\otimes(\otimes^p H^*)
$
can be written as
 \begin{equation}\label{eq:section}
   f=    f^{ A_1'\ldots A_l'}_{B_1\ldots B_qB_1'\ldots B_p'}  e_{A_1'}\otimes \cdots \otimes e_{A_l'}\otimes e^{B_1}\otimes\cdots
   \otimes
 e^{B_q}\otimes  e^{B_1'}\otimes \cdots \otimes  e^{B_p'} .
 \end{equation}
 We can identify this section  with the tuple of functions
 \begin{equation}\label{eq:section-f}
 \left   (\cdots,   f^{ A_1'\ldots A_l'}_{B_1\ldots B_qB_1'\ldots B_p'},\cdots\right) .
 \end{equation}

A  {\it contraction}  is a map
$
  C:\mathfrak T^{ l+1}_{q,p+1}\longrightarrow \mathfrak T^{ l}_{q,p}
$
given by
$
  (C  f)^{ A_{ 1}'\ldots A_{l }'}_{B_1\ldots B_qB_{1 }'\ldots B_{p }'}:=
   f^{ A_{ 1}'\ldots D' \ldots A_l'}_{B_1\ldots B_qB_{ 1}'\ldots D' \ldots
  B_p'},
$
where the superscript and subscript  $D'$ appear in $j$-th and $\widehat{j} $-th places, respectively. It is a well defined element of
$\mathfrak T^{ l}_{q,p}$ since $e^{D'}\otimes e_{D'}=\widetilde{e}^{E'}\otimes \widetilde{e}_{E'}$ under the transformation $
\widetilde{e}^{E'}=(M^{-1})_{D'}^{{\phantom {A'} } E' }e^{D'} $, $\widetilde{e}_{E'}=M_{E'}^{{\phantom {A'} }D' }e_{D'} $ for
$(M_{E'}^{{\phantom {A'} }D' })\in {\rm {Sp}}(1)$.
We use $\varepsilon_{A'B'}$ to raise or lower primed indices. For example,
\begin{equation*}
    f_{\ldots\phantom{A'}\ldots}^{\phantom{\ldots}A'}=  f_{\ldots{B'}\ldots}\varepsilon^{B'A'},\qquad
    f_{\ldots\phantom{A'}\ldots}^{\phantom{\ldots}A'}\varepsilon_{A'C'}=  f_{\ldots{C'}\ldots},
\end{equation*}
where $(\varepsilon^{A'B'})$ is the inverse of $(\varepsilon_{A'B'})$, i.e.,
$
   \varepsilon_{A'B'}\varepsilon^{B'C'}=\delta_{A'}^{ C'}=\varepsilon^{C'B'}\varepsilon_{B'A'}.
$ So it is the same after raising and    lowering primed indices.
  $\varepsilon$ has the standard form locally:
 \begin{equation}  (\varepsilon_{A'B'})=\left( \begin{array}{cc} 0&
 1\\-1& 0\end{array}\right),\qquad  (\varepsilon^{A'B'}) =\left( \begin{array}{cc} 0&
- 1\\1& 0\end{array}\right). \label{eq:epsilon}
 \end{equation}
On a unimodular
  quaternionic manifold, we can not use $\epsilon$ to   raise or lower unprimed indices. This is why we only consider tensors as
  sections of  $\mathfrak T^{ l}_{q,p}$. But on quaternionic K\"ahler  manifold, we can use $\epsilon_{AB}$ to   raise or lower unprimed
  indices (cf. Section 4).

Recall that
 a {\it covariant derivative} of a vector bundle $V$¡¡ is  a mapping $\nabla: \Gamma(V)\longrightarrow \Gamma((TM)^*\otimes V)$
 satisfying
$
    \nabla (fv) =df\otimes v+f \nabla  v,
    \nabla(v_1+v_2) =\nabla v_1+ \nabla v_2,
$
for any $v,v_1,v_2\in \Gamma(V)$ and scalar function $f$. $\nabla$ acts on $V^*$ naturally by duality:
$
   X(v,v^*)=(\nabla_X v,v^*)+(v,\nabla_X v^*)
$ for any vector field $X\in TM$, $v \in \Gamma(V)$ and $v^* \in \Gamma(V^*)$.
A covariant derivative can be naturally extended to a map
$
 \nabla:  \Gamma((\otimes^k V)\otimes (\otimes^l V^*))\longrightarrow \Gamma((TM)^*\otimes(\otimes^k V)\otimes(\otimes^l V^*))
$.

The quaternionic  connection induces an $\mathfrak {gl} (2n,\mathbb{C} ) $-connection $\omega'$ on $E$ and an  $\mathfrak {su} (2)
$-connection $\omega''$ on $H$.
When the manifold is  unimodular  or quaternionic K\"ahler,
$\omega' $     is $\mathfrak {sl} (2n,\mathbb{C} ) $- or $\mathfrak {sp} (n) $-valued.
  $\nabla$ is naturally extended to   well defined mappings   $E\rightarrow (\mathbb{C}TM)^*\otimes E$  and $H\rightarrow
  (\mathbb{C}TM)^*\otimes H$
    by
$
       \nabla_{X+\mathbf{i}Y}:=\nabla_{X }+\mathbf{i}\nabla_{Y}
$,
  which induce well defined mappings   $E^*\rightarrow (\mathbb{C}TM)^*\otimes E^*$  and $H^*\rightarrow (\mathbb{C}TM)^*\otimes H^*$
  by duality, and so we get a well
  defined mapping  $\mathfrak T^{ l}_{q,p}\rightarrow(\mathbb{C}TM)^*\otimes\mathfrak T^{ l}_{q,p}$.

Choose a  local quaternionic frame $\mathfrak e=\{e_{ a}:=X_a  \}_{a=1}^{4n}$ of $TM$ and its dual $\{e^a\}_{a=1}^{4n}$. Write
\begin{equation*}\label{eq:connection-E-H}
  \nabla  e_{A}=\Gamma_{aA}^{\phantom{aA}B}e^a\otimes e_B, \qquad \nabla  e_{A'}=\Gamma_{aA'}^{\phantom{aA'}B'}e^a\otimes e_{B'},
 \end{equation*}
 where $\Gamma_{aA}^{\phantom{aA}B}=\omega' (X_a)_A^{\phantom{ A}B}$ and $\Gamma_{aA'}^{\phantom{aA'}B'}=\omega''(X_a)_{A'}^{\phantom{
 A'}B'}$ are connection coefficients.
  Then by duality, we have
 $
   \nabla  e^A=-\Gamma_{aB}^{\phantom{aB}A}e^a\otimes e^B,$ $ \nabla  e^{A'}=-\Gamma_{aB'}^{\phantom{aB'}A'}e^a\otimes e^{B'},
$ which are equivalent to
 \begin{equation*}
    \nabla_a   f_A=X_a  f_A-\Gamma_{aA  }^{\phantom{aB }D}   f_D, \qquad  \nabla_a   f_{A'}=X_a  f_{A'}-\Gamma_{aA ' }^{\phantom{aB'
    }D'}   f_{D'}.
 \end{equation*}
In general,  $\nabla   f$ for $  f$ given by (\ref{eq:section})-(\ref{eq:section-f}) is    the tuple
 \begin{equation*}
 \left (\cdots,\nabla_a   f^{ A_1'\ldots A_l'}_{B_1\ldots B_qB_1'\ldots B_p'} ,\cdots\right),
 \end{equation*}as a section of
$
   (\mathbb{C}TM)^*\otimes  \mathfrak T^{ l}_{q,p}\cong   \mathfrak T^{ l}_{q+1,p+1},
$ by the identification (\ref{eq:identification0}),
 where
  \begin{equation}\label{eq:covariant-derivatives}\begin{split}
\nabla_a   f^{ A_1'\ldots A_l'}_{B_1\ldots B_qB_1'\ldots B_p'}:=X_a  f^{ A_1'\ldots A_l'}_{B_1\ldots B_qB_1'\ldots B_p'}&
+\Gamma_{aD'}^{\phantom{aD'}A_j'}  f^{  \ldots D'\ldots  }_{B_1\ldots B_qB_1'\ldots B_p'}\\&
-\Gamma_{aB_j }^{\phantom{aB_j}D}  f^{ A_1'\ldots A_l'}_{ \ldots D\ldots  B_1'\ldots B_p'} -\Gamma_{aB_j' }^{\phantom{aB_j'}D'}  f^{
A_1'\ldots
A_l'}_{B_1\ldots B_q \ldots D'\ldots  }.
\end{split} \end{equation}

 The  covariant derivative is invariant after  contraction:  $\nabla (C  f)=C(\nabla    f)$,
  because by (\ref{eq:covariant-derivatives}), we have
\begin{equation*}
   [C( \nabla_a    f)-\nabla_a (C  f)]^{ A_{ 1}'\ldots A_{l }'}_{B_1\ldots B_qB_{1 }'\ldots B_{p }'}=\Gamma_{aE'}^{\phantom{aD'}D'}  f^{
   A_{ 1}'\ldots E' \ldots A_l'}_{B_1\ldots B_qB_{ 1}'\ldots D' \ldots
  B_p'}-\Gamma_{a D'}^{\phantom{aD'}E'}  f^{ A_{ 1}'\ldots D' \ldots A_l'}_{B_1\ldots B_qB_{ 1}'\ldots E'\ldots
  B_p'}=0.
\end{equation*}
We will use the notation
\begin{equation*}
   \nabla_{AA'}:= \nabla_{Z_{AA'}} = \frac 1{\sqrt 2}(\nabla_{X_a}+\mathbf{i}\nabla_{X_b})
\end{equation*}
if we write $Z_{AA'} =  \frac 1{\sqrt 2}({X}_a+\mathbf{i} X_b)$ for some $a,b$ (cf. (\ref{eq:k-CF})). We also write $\nabla_{AA'}$ as
$\nabla_{A'A}$ when it is more convenient. Then \begin{equation}\label{eq:connection-E-H-AA'}
    \nabla_{AA'}e_B=\Gamma_{AA' B}^{\phantom{AA'C} C}e_C,\qquad \nabla_{AA'}e_{ B'}=\Gamma_{AA'B'}^{\phantom{AA'C'} C'}e_{ C'},
 \end{equation}
where $\Gamma_{AA' B}^{\phantom{AA'C}C}= \frac 1{\sqrt 2}(\Gamma_{aB}^{\phantom{aC}C }+\mathbf{i}\Gamma_{b B}^{\phantom{aC}C})$ and
$\Gamma_{AA' B '}^{\phantom{AA'C'} C'}= \frac 1{\sqrt 2}(\Gamma_{aB'}^{\phantom{aC'}C '}+\mathbf{i}\Gamma_{bB'}^{\phantom{aC'}C '})$.
The  formula   (\ref{eq:covariant-derivatives}) holds for $a= {AA'}$.
  Denote
\begin{equation*}2\nabla_{[a} \nabla_{b]}:=
   \nabla_a \nabla_b- \nabla_b\nabla_a.
\end{equation*}
The {\it torsion} is defined as
$
   2\nabla_{[a} \nabla_{b]}\phi=T_{ab}^{\phantom{ab}c} \nabla_c \phi
$
for any scalar function $\phi$. Then by definition
$
   T_{ab}^{\phantom{ab}c} =\Gamma_{ a b}^{\phantom{ab}c}-  \Gamma_{b a}^{\phantom{ab}c} +C_{ab}^{\phantom{ab}c}
$ where the numbers $C_{ab}^{\phantom{ab}c}$ are given  by $[X_a,X_b]=C_{ab}^{\phantom{ab}c}X_c$.
It is direct to check that
\begin{equation*}
   (\nabla_a \nabla_b- \nabla_b\nabla_a)(\phi f_A)=T_{ab}^{\phantom{ab}c} \nabla_c \phi  f_A+\phi(\nabla_a \nabla_b- \nabla_b\nabla_a)
   f_A
\end{equation*} for a scalar function $\phi$, by the formula (\ref{eq:covariant-derivatives}) for covariant derivatives. So when the
connection is torsion-free,
$ 2\nabla_{[a}
\nabla_{b]}  $ is an endomorphism of $ \Gamma(  E^*)$ as a $C^\infty(M)$-module for fixed $a,b$ (similarly for $\Gamma(  H^*)$). So we
can write
\begin{equation*}\begin{split}
   2\nabla_{[a} \nabla_{b]}  f_A&:=-R_{ab A }^{\phantom{ab D  }   D}  f_D,
\qquad
  2\nabla_{[a} \nabla_{b]}  f_{A'}: =-R_{ab A'  }^{\phantom{ab D'  }   D'}  f_{D'}.
\end{split} \end{equation*}
By (\ref{eq:covariant-derivatives}), we see that
\begin{equation*}\begin{split}
   \nabla_{[a} \nabla_{b]}(  f_{\mathscr A}  h_{\mathscr B})=\nabla_{[a} \nabla_{b]}   f_{\mathscr A}\cdot   h_{\mathscr B} +
   f_{\mathscr A}\cdot
   \nabla_{[a} \nabla_{b]}  h_{\mathscr B}  .
\end{split} \end{equation*}
In general, we have the {\it generalized Ricci identity}:
 \begin{equation}\label{eq:Ricci}\begin{split}
  2\nabla_{[a} \nabla_{b]}  f^{ A_1'\ldots A_l'}_{B_1\ldots B_qB_1'\ldots B_p'}:=  R_{abD'}^{\phantom{abD'}A_j'}  f^{  \ldots D'\ldots
  }_{B_1\ldots B_qB_1'\ldots B_p'}&
-R_{abB_j }^{\phantom{abB_j}D}  f^{ A_1'\ldots A_l'}_{ \ldots D\ldots  B_1'\ldots B_p'}  -R_{abB_j' }^{\phantom{abB_j'}D'}  f^{
A_1'\ldots
A_l'}_{B_1\ldots B_q \ldots D'\ldots  }.
\end{split}
 \end{equation}
See Penrose-Rindler  \cite{PR1} \cite{PR2} or  Bailey-Eastwood  \cite{BaiE}.

  If the manifold is unimodular
  quaternionic, the connection on $E$ preserves the $2n$-form $\epsilon$   (\ref{eq:epsilon-unprimed}), i.e.
\begin{equation}\label{eq:volume-preserve-unprimed}
   \nabla_a \epsilon_{A_1\ldots A_{2n}}=
   \Gamma_{aA_j}^{\phantom{aA_j }D}\epsilon_{A_1\ldots D \ldots A_{2n}  } = \Gamma_{aD}^{\phantom{aA }D}\epsilon_{A_1\ldots   A_{2n}  }
   =0
\end{equation} where $ \epsilon_{A_1 \ldots    A_{2n}  }$ is nonzero only if $(A_1\ldots   A_{2n}) $ is a permutation of
$(0,\ldots , 2n-1)$ and  $ (\Gamma_{aA}^{\phantom{aA }D}) $ is   $\mathfrak {sl} (2n, \mathbb{C} )  $-valued.
Similarly, the symplectic form  $\varepsilon_{A'B'}$   on $H$  is also preserved by the connection on $H$, i.e.
\begin{equation}\label{eq:preserve}
   \nabla_a    \varepsilon_{A'B'}=0.
\end{equation}Thus when the manifold is  quaternionic K\"ahler,
  $\nabla$ is a connection of  $\mathbb{C} TM=H\otimes E$ compatible with metric
\begin{equation}\label{eq:gab}
   g_{ab}=g(Z_{AA'},Z_{BB'})= \varepsilon_{A B }\varepsilon_{A'B'} ,\qquad\qquad  a=AA',\quad b=BB',
\end{equation} for   $Z_{AA'}$ in (\ref{eq:k-CF}) if we choose  local quaternionic  frame $\{X_a\}$ in (\ref{eq:k-CF}) orthonormal.

  The notion of unimodular
  quaternionic structure is a real version of the   notion of  torsion-free QCF-structure on a complex    manifold  introduced  by
  Bailey   and    Eastwood  \cite{BaiE}. A {\it quaternionic conformal structure} (briefly {\it QCF}) on a $4n$-dimensional
complex quaternionic manifold $\widetilde{M}$ is given by an isomorphism between  $T\widetilde{M}$ and $ \widetilde{E }\otimes
\widetilde{H}$ and a fixed isomorphism between $\Lambda^{2n}\widetilde{E}^*$ and
$\Lambda^{2 }\widetilde{H}^*$.
Given a  symplectic form $\widetilde{\varepsilon}$ in $\Lambda^{2 }\widetilde{H}^*$, the isomorphism induces
a $2n$-form $\widetilde{\epsilon}$ in $\Lambda^{2n}\widetilde{E}^*$,  and there exists a unique
connection $\nabla$ preserving
$
    \widetilde{\epsilon}   $ and $ \widetilde{\varepsilon}
$ (cf. theorem 2.4 in  Bailey-Eastwood  \cite{BaiE}).
The QCF-structure is called {\it torsion-free} if the induced connection on the holomorphic tangent  bundle is torsion-free.

  The curvature of the complexified tangent bundle is
 \begin{equation*}
    R_{ AA'B B' C C' }^{\phantom{AA'B B' C C'}    DD'}f^{CC'} =(\nabla_{AA'} \nabla_{BB'}-\nabla_{BB'}\nabla_{AA'})f^{D  D'}
 \end{equation*} if we identify $(f^{D  D'})$ with a local section of $\mathbb{C}TM$.
 By the   generalized Ricci identity, we see that
the curvature has the decomposition:
 \begin{equation}\label{eq:decomposition}
    R_{ AA'B B' C C' }^{\phantom {AA'B B' C C'}    DD'}= R_{ AA'B B' C   }^{\phantom { AA'B B' C }  D }\delta_{C'}^{\phantom {C'}D'}+R_{
    AA'B B' C'
    }^{\phantom {AA'B B' C'  }    D' }\delta_{C }^{\phantom {C}D }.
 \end{equation} We will use the following notations: for $a=AA', b=BB'$,
 \begin{equation*}\begin{split}
     R_{A'B'AB C   }^{\phantom{AA'B B' C  }     D } :=R_{ AA'B B' C   }^{\phantom{AA'B B' C  }     D }=R_{ ab C   }^{\phantom{ab C  }
     D }, \qquad R_{ A B
     A' B'   C' }^{\phantom{A B A' B'  C'}     D'}   :
  =    R_{ AA'B B'   C' }^{\phantom{AA'B B'  C'}     D'}=  R_{ ab   C' }^{\phantom{ab  C'}     D'}.\end{split}
 \end{equation*}
\begin{cor}  On a unimodular
  quaternionic manifold, we have
   \begin{equation}\label{eq:vanishing-trace}
  R_{A'B' AB C    }^{\phantom { AA'B B' C}     C  }=0,\qquad   R_{ A B A'B' C'   }^{\phantom {  AA'B B' C' }   C' }=0.
 \end{equation}
\end{cor}
\begin{proof} Note that
$(\nabla_{AA'} \nabla_{BB'}-\nabla_{BB'}\nabla_{AA'})\epsilon_{A_1\ldots  A_{2n}  }=0$ by (\ref{eq:volume-preserve-unprimed}), which
implies that
\begin{equation*}
    R_{A'B'AB A_j  }^{\phantom{AA'B B' C_j  }   D}\epsilon_{A_1\ldots D \ldots A_{2n}  }=0
\end{equation*}
by the   generalized Ricci identity again. Noting that $ \epsilon_{A_1,\ldots ,  A_{2n}  }$ is nonzero only if $(A_1\ldots   A_{2n}) $
is a permutation of
$(0,\ldots , 2n-1)$,  we find the first trace
    vanishing. It is similar
for the second one.
\end{proof}

We will use  symmetrisation and antisymmetrisation of indices
\begin{equation*}\begin{split}
     f_{\cdots (A_1\ldots A_k)\cdots}: =\frac 1 {k!}\sum_{\sigma\in S_k}  f_{\cdots  A_{\sigma(1)}\ldots A_{\sigma(k)} \cdots},\qquad
     f_{\cdots [A_1\ldots A_k]\cdots}: =\frac 1 {k!}\sum_{\sigma\in S_k}{\rm sgn} (\sigma)  f_{\cdots  A_{\sigma(1)}\ldots A_{\sigma(k)}
     \cdots},
\end{split}\end{equation*}
where ${\rm sgn} (\sigma)$ is the sign of the permutation $\sigma$. It is obvious that
\begin{equation}\label{eq:re-antisym}
    f_{\cdots [A_1\ldots A_k]\cdots}=  f_{\cdots [[A_1\ldots A_p]\ldots A_k]\cdots},
\end{equation} by definition of antisymmetrisation.
We will also us the notation
\begin{equation*}\begin{split}
       f_{\cdots [A_1\ldots|\mathscr{A}|\ldots A_k]\cdots} ,
\end{split}\end{equation*}
which means antisymmetrisation of indices $A_1\ldots  A_k$ except for that in   $\mathscr{A}$.   We will use similar  notations for
symmetrisation of primed indices.

  \subsection{The curvature   decomposition on unimodular   quaternionic   manifolds }

 \begin{prop}\label{prop:curvature} For a  unimodular   quaternionic   manifold with dimension $>4$, the curvatures decompose as
 \begin{equation}\label{eq:curvature}
    \begin{split}  R_{[A'B']AB C    }^{\phantom {(A'B')AB C}    D  }  & =\varepsilon_{A'B'}\left( \Psi_{ AB C   }^{\phantom { AB C    }
    D }+2 \delta_{
    (A
    }^{\phantom{(A}D }\Lambda_{B)C }\right),\\
  R_{(A'B')AB C    }^{\phantom {(A'B')AB C}     D  }  & =2 \delta_{ [A  }^{\phantom {[A}D }\Phi_{ B] C A'    B' }, \\
 R_{(AB) A'B' C '   }^{\phantom {(AB) A'B' C ' }     D'  } &= 2\delta_{[A' }^{\phantom { A'] }D' }\Phi_{B']C'A B   }, \\
   R_{[A B ]A'B' C '   }^{\phantom {[A B ]A'B' C '}D'  } & = 2 \Lambda_{  A B } \delta_{(A'}^{\phantom{(A'}D' } \varepsilon_{B')C'} ,
    \end{split}
 \end{equation}
 where  the first identity above can viewed as the definition of $\Psi_{ AB C   }^{\phantom{AB C }     D } $, and
  \begin{equation}\label{eq:curvature-0}
    \begin{split} \varepsilon_{A'B'}\Lambda_{   A B    }  :=\frac 13 R_{[ AB] C' [A'B']   }^{\phantom{[ A C] C' [A'B'] }  C' } ,\qquad
    \Phi_{AB A'B' } :=    R_{(A B ) C' (A'B')
    }^{\phantom{(A C ) C' (A'B')}    C' },
    \end{split}
 \end{equation}
and $\Phi_{A'B'AB }:=\Phi_{AB A'B' }$. Moreover,
   \begin{equation}\label{eq:curvature-sym}
  \Lambda_{   A B    }=\Lambda_{  [ A B ]   },\qquad \Phi_{AB A'B' }=\Phi_{( AB ) (A'B') },\qquad  \Psi_{ AB C   }^{\phantom{AB C }
  D }= \Psi_{ (AB
  C )  }^{\phantom{(AB C )}     D }
  \end{equation}
  and $\Psi_{ AB C   }^{\phantom{AB C }     D }$ are totally trace free: $ \Psi_{ AB C   }^{\phantom {AB C}    A }=\Psi_{ AB C
  }^{\phantom { AB C}    B
  }=\Psi_{ AB C   }^{\phantom { AB C}   C }=0$.

  When the   manifold is $ 4$-dimensional, (\ref{eq:curvature}) holds except for  the last identity   replaced by
  \begin{equation}\label{eq:curvature-4d}
        R_{[A B ]A'B' C '   }^{\phantom {[A B ]A'B' C '}D'  }   = \epsilon_{AB}{\Psi}_{ A'B' C'   }^{'\phantom { A'B' C     }  D '} + 2
        \Lambda_{  A B } \delta_{(A'}^{\phantom{(A'}D' } \varepsilon_{B')C'} ,
  \end{equation}  with ${\Psi}_{ A'B' C'   }^{'\phantom { A'B' C    }  D '} =\Psi_{  ( A'B' C' )  }^{'\phantom { (A'B' C')    }  D '}$
  also totally trace free. (\ref{eq:curvature-4d}) can be viewed as the definition of $\Psi'$.
     \end{prop}

   In the $4$-dimensional case, we will only consider {\it right conformally flat manifolds} later, i.e. ${\Psi}_{ A'B' C'
   }^{'\phantom { A'B' C     }  D '}=0$
(cf. section 6.9 of \cite{PR2} for this concept and its necessity for defining massless field equations).
  See Penrose and Rindler's book (section 4.6 of \cite{PR1}) for this curvature decomposition  for
$4$-dimensional manifolds. It is generalized to torsion-free QCFs by Bailey   and   Eastwood   (cf. p.83 in \cite{BaiE}) with a sketched
proof.  See the Appendix   for a detailed proof   by only using the   first  Bianchi identity.

It is well known that a   quaternionic K\"ahler  manifold is Einstein.
See lemma A.1 and theorem 7.8 in \cite{BaiE}  for the proofs of the following proposition for QCF manifolds. See also
 the Appendix  for a detailed proof.

\begin{prop} \label{prop:Einstein}
 (1)  If the manifold is unimodular
  quaternionic, then we have
$
   \nabla_{[A }^{A' } \Lambda_{ BC]}=0.
$

(2)
For a   quaternionic K\"ahler  manifold, we have
 \begin{equation}\label{eq:Einstein}
   \Phi_{ABA'B'}=0, \qquad \Lambda_{AB}=\Lambda \epsilon_{AB}, \qquad\Lambda=\frac {s_g}{8n(n+2)},
 \end{equation}where $s_g$ is the scalar curvature. Namely, it is Einstein.
\end{prop}

\section{The quaternionic complexes }

\subsection{The   $k$-Cauchy-Fueter complexes } Recall that an element of the
{\it symmetric power} $\odot^{p} {H}^* $ is given by a tuple $
(  f_{A_1' \ldots A_p'})$,  which as an element of $\otimes^{p} {H}^*$  is invariant under the permutation of
subscripts   $A_1',\ldots, A_p'   =0',1' $. An element of the exterior power $\Lambda^{q}
E^* $ is given by a tuple $ (  f_{A_1\ldots A_q})$,  which as an element of $\otimes^{q} {E}^*$
 is antisymmetric  under the permutation of
subscripts $A_1,\ldots, A_q=0,\ldots, 2n-1 $.

The covariant derivative defines a differential operator $\nabla:\Gamma(\mathfrak T_{q,p} )\rightarrow\Gamma(\mathfrak T_{q+1,p+1} )$
given by
\begin{equation}\label{eq:connection-operator}  \begin{split}
   (\nabla   f)_{A_0\ldots A_q A_0'\ldots A_p'} =& \nabla_{A_0A_0'}  f_{ A_1\ldots A_q A_1'\ldots A_p'}
    ,
 \end{split}\end{equation} Note  that for $f\in \Gamma(\Lambda^{q } E^*\otimes \odot^{p }H^*)$,   $(\nabla   f)_{A_0\ldots A_q
 A_0'A_1'\ldots A_p'} $ is still symmetric in $A_1',\ldots, A_p'$ and antisymmetric in $A_1 ,\ldots, A_p $ by using the formula
 (\ref{eq:covariant-derivatives}) for covariant derivatives.  We need its antisymmetrisation $\widehat{\nabla}:\Gamma(\mathfrak T_{q,p}
 )\rightarrow\Gamma(  \Lambda^{q+1}
  {E}^*\otimes  (\otimes^{p+1}  {H}^* )  )$   given by
\begin{equation}\label{eq:connection-antisymmetisation}   \begin{split}
   (\widehat{\nabla}   f)_{A_0\ldots A_q A_0'\ldots A_p'} =& \nabla_{A_0'[A_0}  f_{ A_1\ldots A_q] A_1'\ldots A_p'}
     .
 \end{split}\end{equation}
Let us consider  operators a little bit more general than those appearing  in the  $k$-Cauchy-Fueter complexes
(\ref{eq:quaternionic-complex-diff}):
$
  \mathscr D_{q,p}:\Gamma(\Lambda^{q } E^*\otimes \odot^{p }H^*)\longrightarrow\Gamma(\Lambda^{q+1} E^*\otimes\odot^{p -1 }H^*)
$
given by
\begin{equation*}
 (\mathscr D_{q,p}  f)_{A_1\ldots A_{q+1}A_2'\ldots A_{p }'}=\nabla_{[A_1}^{A_1'}  f_{A_2\ldots A_{q+1}]A_1'A_2'\ldots A_{p }'},
\end{equation*} for a local section $  f$ in $\Gamma(\Lambda^q {E}^*\otimes \odot^{p
}{H}^*)$ (it is  well defined since contraction over $A_1'$ is well defined), and
$
   \mathscr D^{p}_{q}:\Gamma(\Lambda^{q } E^* \otimes\odot^{p }H)\longrightarrow\Gamma(\Lambda^{q+1} E^* \otimes \odot^{p +1}H)
$
given by
\begin{equation*}
  ( \mathscr D^{p}_{q}  f)_{A_1\ldots A_{q+1}}^{ A_1'\ldots A_{p+1 }'}= \nabla_{[A_1}^{(A_1'}   f_{A_2\ldots A_{q+1}]}^{ A_2'\ldots
  A_{p+1}')},
\end{equation*}for a local section $  f$ in $\Gamma(\Lambda^q {E}^*\otimes\odot^{p
}{H}^*)$. The {\it  Baston operator}
 $ \bigtriangleup :\Gamma(\Lambda^{k } E^*)\longrightarrow \Gamma(\Lambda^{k+2} E^*)$ is given by
\begin{equation}\left(\triangle  f\right)_{A_1\cdots A_{k+2}}:=
 \nabla_{A '[A_1} \nabla_{ A_2}^{ A'}  f_{ A_3\cdots A_{k+2}]}+2\Lambda_{[A_1A_2}  f_{ A_3\cdots A_{k+2}]}.
\label{eq:Dj=k}
\end{equation}

\begin{thm}\label{thm:k-Cauchy-Fueter} Suppose that the manifold $M$ is unimodular   quaternionic and is right conformally flat if
$\dim_{\mathbb{R}}M=4$. For $k=0,1,2,\ldots$, the sequences
(\ref{eq:quaternionic-complex-diff}) are  elliptic    differential complexes, where the operators $D_j^{(k)}= \mathscr D_{ j,k-j }$ for
$j=0,\ldots,
k-1$,
 the operators $D_{j}^{(k)}= \mathscr D_{j+1}^{ j-k-1} $ for $j= k+1,\ldots, 2n -2$, and  the operator
 $D_{k }^{(k)}$ is the    Baston operator $\bigtriangleup  $.
\end{thm}
Note that
by using $\varepsilon^{A'B'}$ (\ref{eq:preserve})
 to raise   primed indices,  we have the following commutators
\begin{equation}\begin{split}\label{eq:raise-bracket}
   (\nabla_{ A }^{A '}\nabla_{ B}^{B'}-\nabla_{ B}^{B'}\nabla_{ A }^{A '})  f_{C'}&=(\nabla_{ A \widetilde{A}'} \nabla_{
   B\widetilde{B}'} -\nabla_{ B\widetilde{B}'} \nabla_{ A \widetilde{A}'} )  f_{C'}\varepsilon^{\widetilde{A}'A'}
   \varepsilon^{\widetilde{B}'B'}\\&=- R_{AB   \widetilde{A}'\widetilde{B}' C' }^{\phantom{AB\widetilde{A}'\widetilde{B}'C' }
 D'}  f_{D'}\varepsilon^{\widetilde{A}'A'} \varepsilon^{\widetilde{B}'B'}  =-R_{AB\phantom{  A'B'}C' }^{\phantom{AB}
   A'B'\phantom{C'}D'}  f_{D'},\\
    (\nabla_{ A }^{A '}\nabla_{ B}^{B'}-\nabla_{ B}^{B'}\nabla_{ A }^{A '})  f_{C}&=-R_{\phantom{  A'B'}ABC  }^{A'B'\phantom{ABC} D }
    f_D,
 \end{split}\end{equation}
where
 $
   R_{\phantom{  A'B'}ABC  }^{A'B'\phantom{ABC} D }:=R_{   \widetilde{A}'\widetilde{B}' ABC  }^{
   \phantom{\widetilde{A}'\widetilde{B}'ABC} D
   }\varepsilon^{\widetilde{A}'A'}\varepsilon^{\widetilde{B}'B'}
$ by raising indices. We can move $\varepsilon^{\widetilde{A}'A'}$ to the left since the connection $\nabla$ preserves it.
The  following formulae for commutators $\nabla_{ [A }^{(A '}\nabla_{ B]}^{B')}$ as curvatures are important in the proof of Theorem
\ref{thm:k-Cauchy-Fueter}:
\begin{equation}\label{eq:bracket-curvature}\begin{split}
    \nabla_{ [A_1}^{(A_1'}\nabla_{ A_2]}^{A_2')}  h_{C'}&=\frac14\left ( \nabla_{  A_1}^{ A_1'}\nabla_{ A_2 }^{A_2' }-\nabla_{  A_2}^{
    A_2'}\nabla_{ A_1 }^{A_1' }-\nabla_{  A_2}^{ A_1'}\nabla_{ A_1
   }^{A_2' }
   + \nabla_{  A_1}^{ A_2'}\nabla_{ A_2 }^{A_1' }\right)  h_{C'} =-\frac12R_{ [ A_1 A_2] \phantom{   A_1'
   A_2'  }C '
   }^{\phantom{ [A_1 A_2] }   A_1' A_2'  \phantom{ D' } D' }  h_{D'},
 \\
   \nabla_{ [A_1}^{(A_1'}\nabla_{ A_2]}^{A_2')}  h_{C }& =-\frac12R_{ \phantom{(A_1' A_2')}  A_1  A_2 C
   }^{  (A_1' A_2') \phantom{ A_1 A_2   D  } D  }  h_{D }.
 \end{split}\end{equation} by using (\ref{eq:raise-bracket}).
We will also frequently use the following corollary of Proposition \ref{prop:curvature} .

\begin{cor} \label{lem:R-upper} On a  unimodular   quaternionic   manifold (right conformally flat if it is $4$-dimensional), we have
\begin{equation}\label{eq:R-upper}
   R_{[ A_1 A_2]\phantom{ A'B'} C'  }^{\phantom{ [A_1 A_2]} A'B'\phantom{C'}D '} = 2\Lambda_{ A_1 A_2}\delta_{ C'}^{\phantom{
 B'}(A'}\varepsilon^{B')D'},
\end{equation}
in particular,
\begin{equation}\label{eq:R-upper-0}
   R_{ [A_1 A_2]\phantom{( A'B'} C'  }^{\phantom{[ A_1 A_2]}( A'B'\phantom{C'}D ')}=0,
\end{equation}
and \begin{equation}\label{eq:R-A_1' A_2'-0}
   R_{ A_1' A_2' [ A B C  ] }^{\phantom{ A_1' A_2'[A B C] }D }=0,  \qquad {\rm and} \qquad R_{  \phantom{ (A' B'  ) }  AB  C
   }^{ ( A' B ')\phantom{ AB  C} D  }
 = 2 \delta_{ [A  }^{\phantom {[A}D }\Phi_{ B] C}^{\phantom {  B] C} ( {A}'     {B}') }.
\end{equation}
\end{cor}
\begin{proof} (\ref{eq:R-upper})  follows from
\begin{equation*}\label{eq:R-upper'}
   R_{[ A_1 A_2]\phantom{ A'B'} C'  }^{\phantom{ [A_1 A_2]} A'B'\phantom{C'}D '}=R_{[ A_1
   A_2]\widetilde{A}'\widetilde{B}'C'}^{\phantom{[A_1
   A_2]\widetilde{A}'\widetilde{B}'C'} D'}\varepsilon^{\widetilde{A}'A'}\varepsilon^{\widetilde{B}'B'}= 2\Lambda_{ A_1 A_2}\delta_{(
   \widetilde{A}'}^{\phantom{(
   A'}D'}\varepsilon_{\widetilde{B}')C'}\varepsilon^{\widetilde{A}'A'}\varepsilon^{\widetilde{B}'B'}= 2\Lambda_{ A_1 A_2}\delta_{
   C'}^{\phantom{
   A'}(A'}\varepsilon^{B')D'} ,
\end{equation*}by the last identity in (\ref{eq:curvature}) and $\varepsilon^{D'A'}$ antisymmetric.

  The first identity in (\ref{eq:R-A_1' A_2'-0}) follows from antisymmetrising $[ABC]$ in
\begin{equation}\label{eq:R-A_1' A_2'}
   R_{ A_1' A_2'  A B C   }^{\phantom{ A_1' A_2'A B C }D }=\varepsilon_{A_1' A_2'}\left(\Psi_{ABC}^{\phantom{  A B C }D}+2\delta_{(A
   }^{\phantom{(A
   }D}\Lambda_{B)C}\right)+2\delta_{ [A  }^{\phantom {[A}D }\Phi_{ B] C A_1' A_2' },
\end{equation}by using (\ref{eq:curvature}) and symmetry (\ref{eq:curvature-sym}) of $\Phi$ and $\Psi$ in subscripts. For the second
identity, we have
\begin{equation*} \label{eq:curvature-1} \begin{split}
       R_{  \phantom{ (A' B'  ) }  AB  C
   }^{ ( A' B ')\phantom{  AB  C} D  }&= -R_{     A_1' A_2'   AB  C
   }^{ \phantom{ ( A_1' A_2' )  AB C} D  }\varepsilon^{ A_1'( {A }'}\varepsilon^{{B}')A_2'  }
 = -2 \delta_{ [A  }^{\phantom {[A}D }\Phi_{ B] C  A_1' A_2' }\varepsilon^{ A_1'( {A
 }'}\varepsilon^{{B}')  A_2'}
 = 2 \delta_{ [A  }^{\phantom {[A}D }\Phi_{ B] C}^{\phantom {  B] C}  ({A}'     {B}' )}\end{split}\end{equation*}
by
using  (\ref{eq:R-A_1' A_2'}) and $\varepsilon_{A_1' A_2'}\varepsilon^{ A_1'( {A }'}\varepsilon^{{B}')A_2'  }=-\varepsilon^{({B}'A')
}=0 $.
\end{proof}

 {\it Proof of Theorem \ref{thm:k-Cauchy-Fueter}}.
 {\it Case 1: $j=0,\ldots, k-2$}.  To check $D_{j+1}^{(k)}\circ   D_j^{(k)}=0$,  it is sufficient to show $\mathscr D_{ q+1,p-1}
 \mathscr
 D_{ q,p}=0$. For a local section $  f\in \Gamma(\Lambda^q {E}^*\otimes \odot^{p
}{H}^*)$,
we have
\begin{equation*} \begin{split}
 \left(\mathscr D_{ q+1,p-1} \mathscr
 D_{ q,p}  f\right)_{A_1\ldots A_{q+2}A_3'\ldots A_{p }'}&=\nabla_{[A_1}^{A_1'}\nabla_{[A_2}^{A_2'}  f_{A_3\ldots
 A_{q+2}]]A_2'A_1'A_3'\ldots A_{p }'}
 =\nabla_{[[A_1}^{(A_1'}\nabla_{ A_2]}^{A_2')}  f_{A_3\ldots A_{q+2}]A_1'A_2'\ldots A_{p }'} .
 \end{split}\end{equation*}by (\ref{eq:re-antisym}). We can symmetrise superscripts $(A_1'A_2')$ since $ \nabla_b f_{\ldots
 A_1'A_2'\ldots A_{p }'}$ is symmetric in  $ A_1', \ldots, A_{p }'$ as we have mentioned under (\ref{eq:connection-operator}), and so is
 $\nabla_a\nabla_b f_{\ldots A_1'A_2'\ldots A_{p }'}$.
Apply   formula (\ref{eq:bracket-curvature}) for commutators to the above identity and antisymmetrise unprimed indices   to get $
\left(\mathscr D_{ q+1,p-1} \mathscr
 D_{ q,p}  f\right)_{A_1\ldots A_{q+2}A_3'\ldots A_{p}'}
  $ equal to
\begin{equation}\label{eq:D_{k-1,m+1}-D_{k,m}0}\begin{split}
-\frac12R_{[[A_1 A_2]\phantom{  A_1' A_2' }|A_j'| }^{\phantom{[[A_1 A_2]}  A_1' A_2' \phantom{|A_j'|}D'}  f_{A_3\ldots A_{q+2}] \ldots D
'\ldots
 }
 -\frac12R_{\phantom{( A_1' A_2')}[[A_1 A_2 A_j ]  }^{  (A_1' A_2') \phantom{[[A_1 A_2 A_j]  }D}  f_{A_3\ldots |D |\ldots
 A_{q+2}]A_1'A_2'\ldots A_{p }'},
 \end{split}\end{equation}by using  (\ref{eq:re-antisym}). The second term in   (\ref{eq:D_{k-1,m+1}-D_{k,m}0}) vanishes by
 (\ref{eq:R-A_1' A_2'-0}), while   the first term in   (\ref{eq:D_{k-1,m+1}-D_{k,m}0}) is equal to
\begin{equation*}\begin{split}&
   -\frac32\left \{ \varepsilon^{A_2'D'} \Lambda_{[A_1A_2}f_{A_3\ldots A_{q+2}] D 'A_2'\ldots
 A_{p }'}+\varepsilon^{A_1'D'} \Lambda_{[A_1A_2}f_{A_3\ldots A_{q+2}]A_1'D 'A_3'\ldots
 }\right\}\\&-\frac12R_{[A_1 A_2\phantom{  A_1' A_2' }|A_j'| }^{\phantom{[A_1 A_2}  A_1' A_2' \phantom{|A_j'|}D'}  f_{A_3\ldots
 A_{q+2}]A_1' A_2' \ldots D '\ldots
 A_{p }'}=0
 \end{split}\end{equation*}
  by (\ref{eq:R-upper})-(\ref{eq:R-upper-0}), since $  f$ is symmetric in primed indices.
 Consequently, we get $\mathscr D_{ q+1,p-1} \mathscr
 D_{ q,p}   f=0$.
\vskip 3mm
{\it Case 2: $j=k+1,\ldots, 2n-2$}.
To check $D_{j+1}^{(k)}\circ   D_j^{(k)}=0$,  it is sufficient to show $\mathscr D^{p+1}_{q+1}\mathscr
D^{p}_{q}=0$.
Similarly,  we have
\begin{equation}\label{eq: D_{k-1}^{m+1}-D_{k}^{m}}\begin{split}
  (\mathscr D^{p+1}_{q+1}\mathscr D^{p}_{q}  f )_{A_1\ldots A_{q+2}}^{ A_1'\ldots A_{p+2 }'}&=\nabla_{[A_1}^{(
  A_1'}\nabla_{[A_2}^{(A_2'}  f_{A_3\ldots
 A_{q+2}]]}^{ A_3'\ldots A_{p+2}'))}  =\nabla_{[[A_1}^{(( A_1'}\nabla_{ A_2]}^{A_2' )}  f_{A_3\ldots A_{q+2}]}^{ A_3'\ldots A_{p+2}')}
  \\
 =&\frac12R_{[A_1 A_2 \phantom{( A_1' A_1'}|D'|  }^{\phantom{[A_1 A_2  }( A_1' A_2'\phantom{|D'|} A_j'}  f_{A_3\ldots   A_{q+2}]}^{
 A_3'\ldots |D'| \ldots
 A_{p+2}')}-\frac12R_{\phantom{ (A_1' A_2'}[A_1 A_2 A_j  }^{ (A_1' A_2'\phantom{A_1 A_2A_j } |D|}  f_{A_3\ldots |D| \ldots
 A_{q+2}]}^{A_3'\ldots
 A_{p+2}')}
 =0,\end{split}\end{equation}
by using (\ref{eq:bracket-curvature}) and (\ref{eq:R-upper-0})-(\ref{eq:R-A_1' A_2'-0}) again for vanishing of curvatures.
\vskip 3mm

{\it Case 3: $j=k-1 $}. Recall that $
 \Gamma\left (   \Lambda^{k-1} E^*\otimes H^*\right) \xrightarrow{\mathscr D_{k-1 ,1 }}  \Gamma \left( \Lambda^k E^* \right)
\xrightarrow{\triangle} \Gamma\left( \Lambda^{k+2} E^*\right)
\xrightarrow{\mathscr D_{k+2 }^0} \Gamma (   \Lambda^{k+3} E^*$ $\otimes H  )
 .$
Let us show   that for a local section $f\in  \Gamma\left (   \Lambda^{k-1} E^*\otimes H^*\right)$,
\begin{equation}\label{eq:Delta-D-0'}\begin{split}
 (\Delta \mathscr D_{k-1 ,1 } f&)_{A_1\ldots A_{k+2} }= \nabla_{A'[A_1}\nabla_{A_2}^{A'} \nabla_{A_3}^{B'}   f_{A_4\ldots A_{k+2}]B' }+2
 \Lambda_{[A_1A_2} \nabla_{A_3}^{B'}   f_{A_4\ldots A_{k+2}]B' }=0.
 \end{split}\end{equation}
Locally we choose a coordinate chart  $U_\alpha$ with trivialization $E^*|_{U_\alpha }
= U_\alpha\times \mathbb{C}^{2n}$, $H^*|_{U_\alpha }= U \times \mathbb{C}^{2 } $, and  a two-component  local quaternionic  frame
$\{Z_{AA'}\}$  such that $  \varepsilon$ and $\epsilon$ are  standard.
In particular, $\varepsilon^{1'0'}=-\varepsilon^{0'1'}=1$. Note that
\begin{equation}\label{eq:Delta-D}\begin{split}
 \nabla_{A'[A_1} \nabla_{A_2}^{A'}  \nabla_{A_3}^{B'}  f_{A_4\ldots A_{k+2}]B' }= &\nabla_{0'[A_1}\nabla_{A_2}^{0'}  \nabla_{A_3}^{0'}
 f_{A_4\ldots
 A_{k+2}]0' }+\nabla_{1'[A_1}\nabla_{A_2}^{1'}  \nabla_{A_3}^{0'}   f_{A_4\ldots A_{k+2}]0' }\\&+\nabla_{0'[A_1}\nabla_{A_2}^{0'}
 \nabla_{A_3}^{1'}
   f_{A_4\ldots A_{k+2}]1' }+\nabla_{1'[A_1}\nabla_{A_2}^{1'}  \nabla_{A_3}^{1'}   f_{A_4\ldots A_{k+2}]1'
   }\\=&:\Sigma_1+\Sigma_2+\Sigma_3+\Sigma_4,
 \end{split}\end{equation}
where both sides are sections of $ \mathfrak T^{ 2}_{k+2,2}$, while the left hand side can be viewed as  a section of $ \mathfrak T^{
0}_{k+2,0}$ after contraction of
two
primed indices.
 Since
 \begin{equation*}
    \nabla_{[A }^{A'}  \nabla_{B]}^{A'}  f_{\ldots}=\frac12 \left(\nabla_{ A }^{A'}   \nabla_{B }^{A'}- \nabla_{B }^{A'}\nabla_{ A
    }^{A'}\right)  f_{\ldots}  ,
 \end{equation*}
we have
 \begin{equation}\label{eq:Sigma-1}\begin{split}
    \Sigma_1&= \nabla_{0'[A_1}\nabla_{[A_2}^{0'}  \nabla_{A_3]}^{0'}   f_{A_4\ldots A_{k+2}]0' }\\&=\frac 12\nabla_{0'[A_1}\left(- R_{
    [A_2 A_3]\phantom{
    0'0'} |
    0' | }^{\phantom{ [A_2 A_3]} 0'0'\phantom{|0'|}A '}  f_{A_4\ldots A_{k+2}]A' }- R_{ \phantom{ 0'0'} [ A_2 A_3 A_j ] }^{0'0'\phantom{
    [A_2 A_3 A_j
    }D}  f_{A_4\ldots |D|\ldots A_{k+2}]0' }\right)\\&
    =  \nabla_{0'[A_1} \left(\Lambda_{ A_2 A_3}   f_{A_4\ldots A_{k+2}]1' }
    \right)= - \Lambda_{[ A_2 A_3} \nabla_{A_1}^{1' }  f_{A_4\ldots A_{k+2}]1' }
  ,
\end{split} \end{equation} by  using  (\ref{eq:bracket-curvature})-(\ref{eq:R-upper}),  $R_{ \phantom{ 0'0'} [ A_2 A_3 A_j ]
}^{0'0'\phantom{ [A_2 A_3 A_j] }D}=0$
by (\ref{eq:R-A_1' A_2'-0}) and Proposition \ref{prop:Einstein} (1).
  Similarly, we have
\begin{equation*}\begin{split}
     \Sigma_4&=\frac 12\nabla_{ 1' [A_1}\left(-R_{ [A_2 A_3]\phantom{ 1'1'} |1' | }^{\phantom{ [A_2 A_3]}1'1'\phantom{||1'}A '}
     f_{A_4\ldots A_{m+2}]A' }-R_{
     \phantom{
   1'1'} [ A_2 A_3 A_j  }^{1'1'\phantom{ [A_2 A_3 A_j }D}  f_{A_4\ldots |D|\ldots A_{k+2}]1' }\right)\\&
    = \nabla_{1'[A_1} \left(-\Lambda_{ A_2 A_3}   f_{A_4\ldots A_{k+2}]0' }
    \right)=-\Lambda_{ [A_2 A_3}\nabla_{A_1}^{0'}     f_{A_4\ldots A_{k+2}]0' }
    .
\end{split} \end{equation*}

To calculate the term $\Sigma_2$,   applying the trivial identity
\begin{equation}\label{eq:change-order}
   \nabla_{ A  }^{0'}\nabla_{B}^{1'}  f_{\ldots} =\left[\nabla_{B}^{1'} \nabla_{ A  }^{0'}+ (\nabla_{A  }^{0'}
   \nabla_{B}^{1'}-\nabla_{B}^{1'}\nabla_{A  }^{0'} )\right]  f_{\ldots}
\end{equation} and (\ref{eq:raise-bracket}) for commutators,  then   antisymmetrising  $[A_1 \ldots A_{k+2}]$, we get
 \begin{equation}\label{eq:Sigma-2}\begin{split}
     \Sigma_2=& \nabla_{[A_1 }^{0'}\nabla_{A_2}^{1'}  \nabla_{A_3}^{0'}   f_{A_4\ldots A_{k+2}]0' }  \\=& \nabla_{[ A_2}^{1'}
     \nabla_{A_1}^{0'}
    \nabla_{A_3}^{0'}   f_{ A_4\ldots A_{k+2}]0' }
    -R_{\phantom{ 0'1' } [A_1 A_2 A_j   }^{0'1'\phantom{ [A_1 A_2A_j }D  } \nabla_{A_3}^{0'}  f_{\ldots |D|\ldots A_{k+2}]0' }  \\&+ R_{
    [ A_1 A_2
    \phantom{ 1'0'}|A' |  }^{\phantom{[ A_1 A_2  }0' 1'\phantom{ |A'| }0' } \nabla_{A_3}^{A'}  f_{A_4\ldots   A_{k+2}]0' }- R_{[  A_1
    A_2 \phantom{
    1'0'}|0' |  }^{\phantom{ A_1 A_2  } 0'1'\phantom{| 0'| }A' } \nabla_{A_3}^{0'}  f_{A_4\ldots   A_{k+2}]A' }
   \\=&- \nabla_{[A_1 }^{1'}   \nabla_{A_2}^{0'}
    \nabla_{A_3}^{0'}   f_{A_4\ldots A_{k+2}]0' }= \Sigma_1,
\end{split} \end{equation} (here $D$ is in the $j$-th place  of $A_3,\ldots A_{k+2}$), since the   last two curvature terms cancel  by
\begin{equation}\label{eq:curvature-id3}
   R_{  [A_1 A_2] \phantom{ 1'0'}A'   }^{\phantom{[ A_1 A_2]  } 0'1'\phantom{ A' }0' }=\Lambda_{ A_1 A_2} \delta_{A'}^{\phantom{ A'}
   0'}\varepsilon^{1' 0'} =
   R_{ [ A_1 A_2] \phantom{ 1'0'}0'   }^{\phantom{ [A_1 A_2]  } 0'1'\phantom{ 0' }A' },
\end{equation}   and
  using   (\ref{eq:R-upper}) again.
 Similarly, we have
 \begin{equation*}\begin{split}
      \Sigma_3=& -\nabla_{[A_1 }^{1'}\nabla_{A_2}^{0'}  \nabla_{A_3}^{1'}   f_{A_4\ldots A_{k+2}]1' }
     \\=& -\nabla_{[ A_2}^{0'}   \nabla_{A_1 }^{1'} \nabla_{A_3}^{1'}   f_{A_4\ldots A_{k+2}]1' }
    +R_{\phantom{ 1'0' } [A_1 A_2 A_j   }^{1'0'\phantom{ [A_1 A_2A_j }D  } \nabla_{A_3}^{1'}  f_{A_4\ldots |D|\ldots A_{k+2}]1' } \\&-
    R_{ [ A_1 A_2
    \phantom{ 1'0'}|A'|   }^{\phantom{ [A_1 A_2  } 1'0'\phantom{ |A' |}1' } \nabla_{A_3}^{A'}  f_{A_4\ldots   A_{k+2}]1' }+ R_{[  A_1
    A_2 \phantom{
    1'0'}|1'|   }^{\phantom{[ A_1 A_2  } 1'0'\phantom{| 0'| }A' } \nabla_{A_3}^{1'}  f_{A_4\ldots   A_{k+2}]A' }
  \\=& \nabla_{[ A_1}^{0'}   \nabla_{A_2 }^{1'} \nabla_{A_3}^{1'}   f_{A_4\ldots A_{m+2}]1' }=  \Sigma_4.
\end{split} \end{equation*}
Thus,   we find that
$
   \Sigma_1+\Sigma_2+\Sigma_3+\Sigma_4=-2\Lambda_{[A_1 A_2 }\nabla_{  A_3}^{A'}    f_{A_4\ldots A_{k+2}]A'}.
$
So (\ref{eq:Delta-D-0'}) follows.

\vskip 3mm
{\it Case 4: $j= k$}. For $k\geq 1$ and $B'=0'$, we have
\begin{equation}\label{eq:Delta-D-0}\begin{split}
 ( \mathscr D_{k+2 }^0& \Delta  f)_{A_1\ldots A_{k+3} }^{B'}= \nabla_{[A_1 }^{B'} \nabla_{|A'| A_2}\nabla^{ A'  }_{A_3  }   f_{A_4\ldots
 A_{k+3}]  }+2
\nabla_{[A_1 }^{B'} (\Lambda_{ A_2A_3}     f_{A_4\ldots A_{k+3}]  })\\&
=\nabla_{[A_1 }^{0'}\nabla_{ A_2}^{0'}  \nabla_{A_3 }^{1'}   f_{A_4\ldots A_{m+2}]  } - \nabla_{[A_1 }^{0'}\nabla_{ A_2}^{1'}
\nabla_{A_3 }^{0'}
f_{A_4\ldots A_{k+3}]  }+2
 \Lambda_{[A_2A_3} \nabla_{ A_1 }^{0'}    f_{A_4\ldots A_{k+3}]  } \\&
=2\nabla_{[[A_1 }^{0'}\nabla_{ A_2]}^{0'}  \nabla_{A_3 }^{1'}   f_{A_4\ldots A_{m+2}]  } + \nabla_{[A_1 }^{0'} R_{\phantom{ 1'0' }  A_2
A_3 A_j   }^{1'0'\phantom{  A_1 A_2A_j }D  }
f_{A_4\ldots |D| \ldots A_{k+3}]  }+2
 \Lambda_{[A_2A_3} \nabla_{ A_1 }^{0'}    f_{A_4\ldots A_{k+3}]  }
  \end{split}\end{equation}
by raising the index $A'$ and  applying (\ref{eq:change-order}) to $\nabla_{ A_2}^{1'}  \nabla_{A_3 }^{0'}$. The second term in the
right hand side vanishes by (\ref{eq:R-A_1' A_2'-0}), while the first term equals to
 \begin{equation*}\begin{split}     R_{  [A_1 A_2 \phantom{ 1'0'}|D'|   }^{\phantom{[ A_2 A_3  } 0'0'\phantom{ |A'| }1' }
 \nabla_{A_3}^{D'} f_{A_4\ldots A_{k+2}]  }- R_{\phantom{ 1'0'}  [A_1
A_2A_j    }^{ 0'0'\phantom{ A_2 A_3 }\phantom{ A_j }D }  \nabla_{A_3}^{1'}f_{A_4\ldots |D|\ldots  A_{k+2}]  } =-2 \Lambda_{ [A_1
A_2}\nabla_{A_3}^{0'}
f_{A_4\ldots A_{k+2}]  },
 \end{split}
 \end{equation*} by using Corollary \ref{lem:R-upper}  again.
 So $( \mathscr D_{k+2 }^0  \Delta  f)_{ \ldots   }^{B'} $ vanishes for $B'=0'$. It is similar for $B'=1'$. If $k=0$, we can obtain the
 result by using vanishing of torsions.

 The  ellipiticity will be proved in Subsection 3.2.   \hskip 80mm $\Box$

Consider conformal transformation
\begin{equation}\label{eq:conformal-transformation}
  \widetilde{\epsilon}_{A_1\ldots A_{2n}}:=\Omega\epsilon_{A_1\ldots A_{2n}},\qquad \widetilde{\varepsilon}_{A '
  B'}:=\Omega{\varepsilon}_{A '
 B '}.
\end{equation}Fix a two-component  local quaternionic  frame $\{Z_{AA'}\}$ with respect to a volume element $vol$ in (\ref{eq:vol}) and
denote
$
   \Upsilon_{AA'}:=\Omega^{-1}Z_{AA'}\Omega.
$
  Under the conformal transformation
(\ref{eq:conformal-transformation}),
define a new connection $\widetilde{\nabla}$
on the bundles $E^*$ and $H^*$ by
\begin{equation}\label{eq:connection-change}\begin{split}
   \widetilde{\nabla}_{AA'}  f_B&= {\nabla}_{AA'}   f_B-\Theta_{A'AB}^{\phantom{A'AB}D}  f_D,\qquad
   \quad\Theta_{A'AB}^{\phantom{A'AB}D}=\delta_{ A }^{\phantom{ A }D} \Upsilon_{ B A' } ,\\ \widetilde{\nabla}_{AA'}  f_{B'}& =
   {\nabla}_{AA'}
     f_{B'}-\Theta_{AA'B'}^{\phantom{AA'B'}D'}  f_{D'},\qquad \Theta_{AA'B'}^{\phantom{AA'B'}D'} =\delta_{ A' }^{\phantom{ A'
  }D' } \Upsilon_{ A B ' } . \end{split}
\end{equation}
 Then it is a quaternionic connection for the unimodular  quaternionic   structure with respect to the volume $ \Omega^{2n+2}vol $.
 The curvatures of the unimodular   quaternionic connection $\widetilde{\nabla}_{AA'} $ satisfy
  \begin{equation}\label{eq:curvature-conformal}\begin{split}&
     \Omega\widetilde{\Lambda}_{AB }=\Lambda_{AB }+ \frac 12\left(  \nabla_{A'[A}   \Upsilon_{ B] }^{A'}+\Upsilon_{A'[A}  \Upsilon_{B]
     }^{A'} \right),  \qquad
     \Omega \widetilde\Psi_{ AB C   }^{\phantom{AB C }     D }= \Psi_{ AB C   }^{\phantom{AB C }     D },\\&
   \widetilde{\Phi}_{AB A'B' }= \Phi_{AB A'B' }- \nabla_{(A|(A'}    \Upsilon_{B')|B ) } +\Upsilon_{(A|(A'}   \Upsilon_{B')|B) } .
 \end{split} \end{equation}

\begin{prop} \label{prop:D-conformal} The operators associated to the unimodular   quaternionic connection $\widetilde{\nabla}_{AA'} $
are conformal covariant
in the following sense:
\begin{equation}\label{eq:D-conformal} \begin{split}\widetilde{\mathscr D}_{q}^{p} (\Omega^{-q}  f)&=
  \Omega^{-q-1}{\mathscr D}_{q}^{p}  f,\hskip 2mm\qquad {\rm for}\quad f\in \Gamma \left(  \Lambda^q {E}^*\otimes\odot^{p
}{H} \right),\\
  \widetilde{\mathscr D}_{q,p } (\Omega^{-q-1}  f)&=  \Omega^{-q-2}{\mathscr D}_{q, p}  f, \qquad {\rm for}\quad f\in \Gamma \left(
  \Lambda^q
  {E}^*\otimes\odot^{p
}{H}^*\right), \\
     \widetilde{\Delta}  (\Omega^{-q-1}  f)&=\Omega^{-q-2} {\Delta}    f ,\qquad\quad {\rm for}\quad f\in \Gamma \left(  \Lambda^q {E}^*
     \right).
     \end{split}\end{equation}
\end{prop}

(\ref{eq:D-conformal}) holds with respect to fixed frame and coframe.  The weight factors in (\ref{eq:D-conformal}) coincide with that
obtained by the representation theory in Proposition 10 in  Baston  \cite{Ba}.
  See Penrose and   Rindler
   \cite{PR1} section 5.6-5.7 for conformal transformation of spin $  k/ 2$  massless field operator over $4$-dimensional manifolds.
   The conformal
   change of connections  for QCF's  was  introduced by Bailey   and    Eastwood  (cf. section 2.2 of \cite{BaiE}). See
    the Appendix  for a detailed proof.

\subsection{The   ellipticity of the $k$-Cauchy-Fueter complex}

Locally  for a matrix-valued differential operator $L=\sum_{j=0}^m\sum_{|\nu|=j}  A_\nu \partial_{x_{\nu_1}}\cdots\partial_{x_{\nu_j}}$
of order $m$ on a $N$-dimensional manifold $ M$, its {\it symbol} at point
$p$ is
\begin{equation*}
   \sigma(L)(\xi)=\sum_{|\nu|= m} A_\nu(p) \mathbf{i}\xi_{ {\nu_1}}\cdots \mathbf{i}\xi_{ {\nu_j}},\qquad \xi\in \mathbb{R}^N.
\end{equation*}
$L$ is called {\it elliptic} if $ \sigma(L)(\xi)$ invertible for any $0\neq\xi\in \mathbb{R}^N$.
A differential complex is called {\it elliptic} if the associated symbol sequence is exact at each point $p$. It is well known that the
ellipticity of a differential operator or a differential complex  is independent of the choice of the local coordinate charts \cite{We}.
For a fixed point $p\in M$, let $\sigma(X_j)=\mathbf{i}\xi_j$ and \begin{equation}\label{2.3}  (\xi_{AA'}):=\left(
                                      \begin{array}{ll}
                                      \xi_{1}+\textbf{i}\xi_{2} & -\xi_{3}-\textbf{i}\xi_{4} \\
                                        \xi_{3}-\textbf{i}\xi_{4} & \hskip 3mm\xi_{1}-\textbf{i}\xi_{2} \\
                                       \hskip 3mm \vdots&\hskip 3mm\hskip 3mm\vdots\\
                                                                                  \xi_{4n- 3}+\textbf{i}\xi_{4n-2} &
                                                                                  -\xi_{4n-1}-\textbf{i}\xi_{4n } \\
                                        \xi_{4n-1}-\textbf{i}\xi_{4n } & \hskip 3mm\xi_{4n-3}-\textbf{i}\xi_{4n-2} \\
                                      \end{array}
                                    \right).
\end{equation}In particular, $\sigma (Z_{AA'} )(\xi)= \frac {\mathbf{i}}{\sqrt 2} \xi_{AA'} $ by $ Z_{AA'}$ in (\ref{eq:k-CF}).
For  fixed $k$, we use notations
\begin{equation*}
  \sigma_j(\xi ):=\frac {\sqrt 2}{\mathbf{i}}\sigma\left(D_j^{(k)}\right)(\xi), \mbox{ for } j\neq k,\qquad \mbox{ and } \quad
  \sigma_k(\xi ):=\frac 2{\mathbf{i}^2}\sigma\left(D_k^{(k)}\right)(\xi).
\end{equation*}
The ellipticity of the complex (\ref{eq:quaternionic-complex-diff}) is given by the following exact sequence of the associated symbols,
which can be easily proved by using elementary linear algebra.
\begin{prop}\label{prop:elliptic}
For any $0\neq \xi\in \mathbb{R}^{4n}$,
\begin{equation}\begin{split}
0\longrightarrow & \odot^{k }\mathbb{C}^2
\xrightarrow{\sigma_0(\xi )}\Lambda^1 \mathbb{C}^{2n} \otimes \odot^{k-1
}\mathbb{C}^2
\xrightarrow{\sigma_1(\xi )}\cdots\longrightarrow   \Lambda^k
\mathbb{C}^{2n}\xrightarrow{\sigma_k(\xi )}
\Lambda^{k+2} \mathbb{C}^{2n} \\& \longrightarrow\cdots
\xrightarrow{\sigma_{2n-2} (\xi ) }\Lambda^{2n}
\mathbb{C}^{2n}\otimes
\odot^{2n-k-2}\mathbb{C}^2\longrightarrow
0\end{split}\label{eq:symbol}
\end{equation} is exact. Namely, $\ker \sigma_0(\xi )=\{0\}$,
\begin{equation} \ker
\sigma_j(\xi )={\rm range}\, \sigma_{j-1} (\xi )
\label{eq:elliptic-symbol}
\end{equation} for  $
j=1,\ldots,2n-2$, and $\sigma_{2n-2} (\xi ) $ is surjective.
\end{prop}
\begin{proof}{\it Case 1. $j<k$}. We need to show that \begin{equation}\label{eq:short-exact}
  \Lambda^{j-1} \mathbb{C}^{2n} \otimes \odot^{p+1
}\mathbb{C}^2
\xrightarrow{\sigma_{j-1}(\xi )}\Lambda^{j } \mathbb{C}^{2n} \otimes \odot^{p
}\mathbb{C}^2
\xrightarrow{\sigma_{j }(\xi )}\Lambda^{j+1} \mathbb{C}^{2n} \otimes \odot^{p-1
}\mathbb{C}^2
\end{equation}is exact
  for $p=k-j$, where
the linear mapping
$
    \sigma_j(\xi)
$
is given by
\begin{equation}\label{eq:sigma-xi}\begin{split}
[ \sigma_j(\xi)  \vartheta ]_{A_0\ldots    A_{j  }  A'_2 \ldots A'_{p  }  }&= \xi_{[A_0}^{A'}\vartheta_{ A_1\ldots
A_{j  } ]A'A'_2 \ldots A'_{p  }   } =\frac 1{j+1}\sum_{s=0}^{j }(-1)^{s }\xi_{A_s}^{A'}\vartheta_{  \ldots   { A}_0\ldots
  A'A'_2 \ldots A'_{p  }   } ,
\end{split}\end{equation} for $\vartheta \in \Lambda^j \mathbb{C}^{2n}\otimes  \odot^{p
}\mathbb{C}^2$, by definitions of symbols and antisymmetrisation. Then
\begin{equation*}
   [ \sigma_{j }(\xi)\circ \sigma_{j-1}(\xi)  \vartheta ]_{A_0\ldots    A_{j  }  A'_2 \ldots A'_{p  }  }=\xi_{[[A_0}^{A '}\xi_{
   A_1]}^{B'}\vartheta_{  \ldots
A_{j  } ]B'A 'A'_2 \ldots A'_{p  }   }=0
\end{equation*}
by $\xi_{ [A_0}^{(A '}\xi_{ A_1]}^{B')}=0$  and $\vartheta$  symmetric in the primed indices.

  We can choose
a  transformation $M\in  {\rm GL}( n,\mathbb{H})$ such that its complexification  matrix $\tau(M) \in {\rm GL}( 2n,\mathbb{C})$
transforming $\xi:=\left(\begin{array}{c}\xi_{00'}\\\vdots\\
\xi_{(2n-1)0'}  \end{array}                                    \right)\neq 0$ to  $\widetilde{\xi}:=\left(\begin{array}{c}1\\0\\\vdots
\end{array}
\right)$. Note that $\left(\begin{array}{c}\xi_{01'}\\\vdots\\
\xi_{(2n-1)1'}  \end{array}                                    \right) = - {J}\left(\begin{array}{c}\overline{\xi_{00'}}\\\vdots\\
\overline{\xi_{(2n-1)0'} } \end{array}                                    \right)$ by (\ref{2.3}).  It follows from $ \tau(M) {J}= {J}
\overline{\tau(M)}$  in (\ref{2.231}) that
\begin{equation}\label{eq:M-xi}
\left  (\widetilde{\xi}_{AA'}\right)= \tau(M) (\xi_{AA'})=\left(\begin{array}{cc}
                                      1 &0 \\
                                        0 & 1 \\
                                                                               0 &0 \\
                                                                                \vdots&\vdots
                                      \end{array}
                                    \right).
\end{equation}
Then $\widetilde{\xi}_{A}^{A'}=\tau(M) _{A
   }^{\phantom{A }  B}{\xi}_{B}^{A'}$   by raising indices. Set
$
   \widetilde{\vartheta}_{A_1\ldots  A_{j  }  A_1'\ldots}:=\prod_{l=1}^j \tau(M) _{A_l
   }^{\phantom{A_1} B_l }\cdot \vartheta_{ B_1\ldots   B_{j  }A_1'\ldots} .
$  It is direct to see that
 \begin{equation}\label{eq:sigma-j} \begin{split}
  [ \sigma_j(\widetilde{\xi})  \widetilde{\vartheta}]_{A_0\ldots    A_{j  }  A'_2 \ldots A'_{p  }}=&
 \frac 1{j+1} \sum_{s= 0}^{j }(-1)^{s }\widetilde{\xi}_{A_s}^{A'}\widetilde{\vartheta}_{ A_1\ldots    { A}_0\ldots A_{j  } A'A'_2 \ldots
 A'_{p  }   }\\=&
 \frac 1{j+1} \sum_{s=0}^{j } (-1)^{s } \tau(M) _{ A_s}^{\phantom{A_1}B_s }  \xi_{B_s}^{A'}\prod_{l\neq s}\tau(M) _{
 A_l}^{\phantom{A_1}B_l }  \vartheta_{
 B_1\ldots
  B_0\ldots B_{j  }A'A'_2 \ldots A'_{p  }   }
    \\=& {\prod_{l=0 }^j\tau(M) _{ A_l}^{\phantom{A_1}B_l}} \cdot [ \sigma_j( {\xi})   {\vartheta}]_{
    B_0\ldots B_{j  }A'A'_2 \ldots A'_{p  }   }
   =\widetilde{ [\sigma_j( {\xi})  {\vartheta}]}_{A_0\ldots    A_{j  }  A'_2 \ldots A'_{p  }}
\end{split}\end{equation}
So $\sigma_j( {\xi})  {\vartheta}=0$  if and only if  $\sigma_j(\widetilde{\xi})  \widetilde{\vartheta}=0$. Suppose that the
exactness is proved for $\widetilde{\xi}$ in (\ref{eq:M-xi}). Then there exists $\widetilde{\kappa}\in \Lambda^{j-1} \mathbb{C}^{2n}
\otimes \odot^{p+1
}\mathbb{C}^2$ such that $\sigma_{j-1}(\widetilde{\xi}) \widetilde{\kappa}=\widetilde{\vartheta}$. It follows that
$
   \sigma_{j-1}( {\xi})  {\kappa}= {\vartheta}
$ for $\kappa=\tau(M) ^{-1}\widetilde{\kappa}$ by
(\ref{eq:sigma-j}). So we only need to check the exactness of (\ref{eq:short-exact}) for $\xi $ given by   (\ref{eq:M-xi}).
By raising indices, we
have\begin{equation}\label{eq:xi-lift}
  \left (\xi_{A}^{A'}\right)=\left(\begin{array}{cc}
                                      0 & -1 \\
                                       1 & 0 \\
                                                                               0 &0 \\
                                                                                \vdots&\vdots
                                      \end{array}
                                    \right).
\end{equation}

For $\vartheta\in \Lambda^{j } \mathbb{C}^{2n}\otimes \odot^{p
}\mathbb{C}^2 $,
 by definition (\ref{eq:sigma-xi}) of $\sigma_j( {\xi})  {\vartheta}$ and $\xi_{A}^{A'} $ in (\ref{eq:xi-lift}), we have (i)
\begin{equation}\label{eq:ker-i}
 [ \sigma_j(\xi)  \vartheta ]_{A_0\ldots   A_{j   } A'_2 \ldots A'_{p  } }=0,\qquad {\rm for } \quad   2\leq A_0,\cdots  , A_{j  } ,
\end{equation}
since ${\xi}_{A_j}^{A'}=0$ for such $A_j$'s; (ii) for $2\leq A_1,\cdots  , A_{j  }$,
\begin{equation}\label{eq:ker-ii}\begin{split}&
 [ \sigma_j(\xi)  \vartheta ]_{0A_1\ldots   A_{j  }    A'_2 \ldots A'_{p  } } =-\frac 1{j+1}
 {\vartheta}_{ A_1\ldots   A_{j  }  1'  A'_2\ldots A'_{p  } },
\quad
 [ \sigma_j(\xi)  \vartheta ]_{1A_1\ldots   A_{j  }    A'_2 \ldots A'_{p  } } = \frac 1{j+1}
 {\vartheta}_{ A_1\ldots   A_{j  }  0 ' A'_2 \ldots A'_{p  } };
\end{split} \end{equation}
(iii) for $2\leq A_2,\cdots  , A_{j  }$,
\begin{equation}\label{eq:ker-iii}
 [ \sigma_j(\xi)  \vartheta ]_{01A_2\ldots   A_{j  }    A'_2 \ldots A'_{p  } } =-\frac 1{j+1}(
 {\vartheta}_{1 A_2\ldots   A_{j  } 1' A'_2 \ldots A'_{p  } }+{\vartheta}_{0 A_2\ldots   A_{j  }  0' A'_2 \ldots A'_{p  } }).
 \end{equation}
Therefore   $\vartheta\in \ker \sigma_j( {\xi})  $ for $j\geq 1$ if and only if
 \begin{equation}\label{eq:ker-iff}\begin{split}&
    {\vartheta}_{ A_1\ldots   A_{j  }  A_1'  A'_2 \ldots A'_{p  } } =0, \qquad
     {\vartheta}_{1 A_2\ldots   A_{j  } 1' A'_2 \ldots A'_{p  } }=-{\vartheta}_{0 A_2\ldots   A_{j  } 0' A'_2 \ldots A'_{p  } },
\end{split} \end{equation} for any $A_1',A'_2, \ldots, A'_{p  }=0',1' $ and $2\leq  A_1<\cdots  < A_{j  } $.  While for $j=0$,
$\vartheta\in\ker \sigma_0( {\xi}) $ if and only if ${\vartheta}_{    A_1'  A'_2 \ldots A'_{p  } } =0$. So $\vartheta\in\ker \sigma_0(
{\xi})=\{0\}$.

 To show  the surjectivity    of $\sigma_{j-1}( {\xi})$, we need to find an element     $\Theta\in \Lambda^{j-1} \mathbb{C}^{2n}
 \otimes
 \odot^{p+1
}\mathbb{C}^2$ such that
$
    \sigma_{j-1}( {\xi}) \Theta=\vartheta
$
for $\vartheta$ satisfying (\ref{eq:ker-iff}). We define such an  element $\Theta $ by
\begin{equation}\label{eq:Theta1}\begin{split}
 & \Theta_{A_2\ldots   A_{j  } 0' 1'A'_2 \ldots A'_{p  }} := j{\vartheta}_{1 A_2\ldots   A_{j  } 1' A'_2 \ldots A'_{p  } }
 =-j{\vartheta}_{0 A_2\ldots
   A_{j  } 0' A'_2 \ldots A'_{p  } },\\&
   \Theta_{  A_2\ldots   A_{j  }1'1'1'\ldots 1' }=- j\theta_{0  A_2\ldots   A_{j  } 1'1'\ldots 1' },\qquad \Theta_{  A_2\ldots   A_{j
   } 0'0'0' \ldots 0' }=j\theta_{0  A_2\ldots   A_{j   }  0'0' \ldots 0' },
\end{split} \end{equation} for $2\leq A_2,\cdots  , A_{j   }$, and
\begin{equation}\label{eq:Theta2}\begin{split}
\Theta_{1 A_3\ldots   A_{j   } 1' A'_1 \ldots A'_{p  } }+\Theta_{0 A_3\ldots   A_{j   }  0' A'_1 \ldots A'_{p   } }=-j\vartheta_{
01A_3\ldots   A_{j
  }  A'_1  A'_2 \ldots A'_{p    } },
\end{split} \end{equation}
and all other kind of  entries vanish. Obviously, there exists an element $\Theta$ satisfying conditions
(\ref{eq:Theta1})-(\ref{eq:Theta2}),  but such $\Theta$ is not
unique.
Then by definition (\ref{eq:sigma-xi}) for the  symbol  $\sigma_{j-1}( {\xi}) $ again, as in (\ref{eq:ker-i})-(\ref{eq:ker-iii}),  it is
easy to see that for any
$2\leq A_1,\cdots , A_{j   }$, we
have
$
[\sigma_{j-1}( {\xi})  \Theta]_{ A_1\ldots   A_{j   } A'_1  A'_2 \ldots A'_{p   } } = 0,
$
and
\begin{equation*}\begin{split}&
[\sigma_{j-1}( {\xi}) \Theta]_{ 0A_2\ldots   A_{j  } 0'   A'_2 \ldots A'_{p   } } = - \frac 1{j }\Theta_{  A_2\ldots   A_{j } 1'0'
A'_2 \ldots
A'_{p
} }= \vartheta_{ 0A_2\ldots   A_{j  } 0'   A'_2 \ldots A'_{p   } },\\&
[\sigma_{j-1}( {\xi}) \Theta]_{ 1A_2\ldots   A_{j  } 1'   A'_2 \ldots A'_{p   } } =\frac 1{j   }\Theta_{  A_2\ldots   A_{j  } 0'1'
A'_2 \ldots
A'_{p
} }=\vartheta_{ 1A_2\ldots   A_{j   }1'    A'_2 \ldots A'_{p  } },
\end{split} \end{equation*} by (\ref{eq:Theta1}),
and
\begin{equation*}\begin{split}&
[\sigma_{j-1}( {\xi}) \Theta]_{ 0A_2\ldots   A_{j   }1'1'\ldots 1' } = -\frac 1{j   }\Theta_{  A_2\ldots   A_{j  }1'1'1'\ldots 1' }=
\theta_{0  A_2\ldots   A_{j  } 1'1'\ldots 1' },\\&
[\sigma_{j-1}( {\xi}) \Theta]_{ 1A_2\ldots   A_{j  } 0'0' \ldots 0' } =\frac 1{j   }\Theta_{  A_2\ldots   A_{j   } 0'0'0' \ldots 0'
}=\theta_{0  A_2\ldots   A_{j   }  0'0' \ldots 0' },
\end{split} \end{equation*}by (\ref{eq:Theta1}) again,
and
\begin{equation*}\begin{split}
[\sigma_{j-1}( {\xi}) \Theta]_{ 01A_3\ldots   A_{j   }    A'_1 \ldots A'_{p    } }&= -\frac 1{j }\left
(\Theta_{1 A_3\ldots   A_{j   } 1' A'_1 \ldots A'_{p  } }+\Theta_{0 A_3\ldots   A_{j  }  0' A'_1 \ldots A'_{p   } }\right) =\vartheta_{
01A_3\ldots
A_{j   }  A'_1  A'_2 \ldots A'_{p   } },
\end{split} \end{equation*}by (\ref{eq:Theta2}). Thus $\sigma_{j-1}( {\xi}) \Theta=\vartheta$, and so the sequence
(\ref{eq:short-exact})  is exact.
\vskip 3mm

{\it Case 2. $j=k $}. We need to show that
\begin{equation}\label{eq:short-exact2}
  \Lambda^{k-1} \mathbb{C}^{2n}  \otimes  \mathbb{C}^2
\xrightarrow{\sigma_{k-1}(\xi )}\Lambda^{k } \mathbb{C}^{2n}
\xrightarrow{\sigma_{k }(\xi )}\Lambda^{k+2} \mathbb{C}^{2n}
\end{equation}
 is exact.
For $\vartheta\in  \Lambda^{k} \mathbb{C}^{2n}$, we have
\begin{equation}\label{eq:sigma-k}
  [\sigma_k( {\xi})  {\vartheta}]_{ A_1\ldots   A_{k+2  }   } =\xi_{ A'[A_1}\xi_{A_2 }^{A'}\vartheta_{ A_3\ldots   A_{k+2 }]}.
\end{equation}$\sigma_k( {\xi}) \circ \sigma_{k-1}( {\xi})=0$ holds as before.  Without loss of generality, we can assume $\xi$ is given
by (\ref{eq:M-xi}) as in case $1$.
  We only need to consider unprimed indices with  $ A_1<\cdots  < A_{k+2  }$. Then for this $\xi $, we have
\begin{equation*}
   [\sigma_k( {\xi})  {\vartheta} ]_{ A_1\ldots   A_{k+2  }   } =\left\{
   \begin{array}{l} \frac 2{(k+2)(k+1)}{\vartheta}_{ A_3\ldots   A_{k+2  }   },\qquad {\rm if}\,   A_1= 0  ,  A_2= 1 ,\\
      0,\qquad \qquad \qquad\qquad \qquad{\rm otherwise},\\
   \end{array}\right.
\end{equation*}by $\xi_{ A'[A_1}\xi_{A_2] }^{A'}=0$ except for $\xi_{ A'[0}\xi_{1] }^{A'}=1$.
Therefore  $\ker \sigma_k( {\xi})$ consists of $\vartheta\in  \Lambda^{k} \mathbb{C}^{2n}$ with $ {\vartheta}_{  B_1\ldots    B_{k  }
}=0$ for $2\leq
  B_1,\cdots  ,  B_{ k    }$. On the other hand, for $\Theta\in  \Lambda^{k-1} \mathbb{C}^{2n}\otimes\mathbb{C}^{2 }$ and  $ 2\leq
  A_1,\cdots  , A_{k  } $, we have
 \begin{equation*}\begin{split} &[\sigma_{k-1}( {\xi})  \Theta]_{ A_1A_2\ldots   A_{k  }   }=0,\\
    &
    [\sigma_{k-1}( {\xi})  \Theta]_{ 0A_2\ldots   A_{k  } }=-\frac 1{k }{\Theta}_{  A_2\ldots   A_{k  } 1'  }  ,\qquad  [\sigma_{k-1}(
    {\xi})  \Theta]_{ 1A_2\ldots
    A_{k  } }= \frac 1{k }{\Theta}_{  A_2\ldots   A_{k  } 0'  } , \\&[\sigma_{k-1}( {\xi})  \Theta]_{ 01A_3\ldots   A_{k  } }=-\frac 1{k
    }\left({\Theta}_{ 1 A_3\ldots   A_{k  } 1'  }+{\Theta}_{0  A_3\ldots   A_{k  } 0'  }\right).
 \end{split} \end{equation*}
 Hence as in the case $1$, if we set
$
    \Theta_{  A_2\ldots   A_{k  } 1'  }:=-k\vartheta_{ 0A_2\ldots   A_{k  }   },$ $ \Theta_{  A_2\ldots   A_{k  } 0'
 }:=k\vartheta_{ 1A_2\ldots   A_{k  }   }
 $  and
\begin{equation*}\begin{split}
\Theta_{1 A_3\ldots   A_{k } 1'   }+\Theta_{0 A_3\ldots   A_{k  }  0'   }=-k\vartheta_{ 01A_3\ldots   A_{k }    },
\end{split} \end{equation*}  for $
 A_2, \cdots  , A_{ k    }\geq 2$,
and all other kind of  entries vanish,  then we get $ \sigma_{k-1}( {\xi})  \Theta=\vartheta$. So (\ref{eq:short-exact2})  is exact.

\vskip 3mm
{\it Case 3. $j\geq k+1$}. It is similar. We omit details.
 \end{proof}

 We  can consider the
{\it associated
Laplacian operators} as in the flat case \cite{Wa10}, some of which are of order $4$. This is because $D_k^{(k)}$'s are  of order $2$.
They are all self adjoint elliptic operators. Then by applying the standard theory of elliptic operators \cite{We}, we can obtain  the
following theorem.

 \begin{thm} \label{thm:BVP} Suppose that $M$ is a compact unimodular   quaternionic  manifold and is right conformally flat if
 $\dim_{\mathbb{R}}M=4$. Then we have the Hodge-type orthogonal  decomposition:
\begin{equation}\label{eq:Hodge-decomposition}
    \Gamma\left  (  \mathscr V_j^{(k)}\right)={\rm Image}\, D_{j-1}^{(k) }\oplus {\rm Image}\, D_j^{(k) *}\oplus  \mathscr H^j_{ (k)
    }(M),
\end{equation}where $\mathscr
H^j_{ (k) }(M)\cong  \{f \in   \Gamma    (  \mathscr V_j^{(k)} ); D_j^{(k) }f=0, D_{j-1}^{(k)* }f=0  \}\cong H^j_{ (k) }(M)
$
are  finite dimensional.
\end{thm}
See   the Appendix  for a detailed proof.
For domains in $\mathbb{H}^n$,   the cohomology groups of the $k$-Cauchy-Fueter complex are much more difficult to study. In \cite{CMW}
\cite{Wa08}, we got results for the associated Neumann problem over domains in   $\mathbb{H}$.
\section{  Vanishing Theorem over quaternionic K\"ahler manifolds}
We only consider     quaternionic K\"ahler manifolds in this section.

\begin{thm}\label{thm:Vanishing} Suppose that $M$ is  a compact quaternionic K\"ahler  manifold (right conformally flat if
$\dim_{\mathbb{R}}M=4$) with negative scalar curvature.
    Then we have $  H^j_{ (k) }(M)=\{0\}$ for $j=1,\ldots, k-1 $.
\end{thm}

 Horan
 \cite{Ho} \cite{Ho2} proved the Weitzenb\"ock formula and the vanishing theorem for the first cohomology group of  Salamon's complex.
 See also theorem 4.3  of Nagatomo-Nitta \cite{NN} for vanishing theorem for Salamon's complexes when the manifolds have negative scalar
 curvatures, which essentially  implies the   vanishing of $ H^j_{ (k) }(M)$ of the $k$-Cauchy-Fueter  complex for $j\geq k+3$.  See
 also \cite{Sp} for
indices of Salamon's  complexes.

It is sufficient to prove the associated  Weitzenb\"ock formula. In this section, on a quaternionic K\"ahler  manifold,   we will choose
coordinate charts  $U_\alpha$ with trivialization $E^*|_{U_\alpha }
= U_\alpha\times \mathbb{C}^{2n}$, $H^*|_{U_\alpha }= U \times \mathbb{C}^{2 } $, two-component  local orthonormal quaternionic frame
$\{Z_{AA'}\}$  (\ref{eq:k-CF}) such that $g, \varepsilon$ and $\epsilon$ are  standard:
 \begin{equation}(\epsilon_{AB})= \left ( \begin{array}{rrrrr} 0&
 1&&&\\-1& 0&&&\\
 &&\ddots&&\\&&&0&
 1\\&&& -1& 0\end{array}\right).  \qquad \label{eq:epsilon2}
 \end{equation}$ (\epsilon^{A B }) $ is the inverse of $(\epsilon_{AB})$. They are used to raise or lower unprimed indices, e.g. $
  \nabla_{A'}^{A }=\nabla_{A'B }\epsilon^{B A }$.
Define inner product locally as
\begin{equation*}
  \langle v,w\rangle:=\sum_{B_1,\ldots , B_p'}{v}_{B_1\ldots B_qB_1'\ldots B_p'} \overline{{w}_{B_1\ldots B_qB_1'\ldots B_p'} }
\end{equation*}for two $\mathfrak T_{q,p}$-tensor $v$ and $w$, and $|v|^2:= \langle v,v\rangle$.   It is obviously  well defined
globally since it
 is invariant under the transformation $v\rightarrow \widetilde{v}$  ($w\rightarrow \widetilde{w}$) with
 \begin{equation*}
    \widetilde{v}_{A_1\ldots A_qA_1'\ldots A_p'}= \prod_{j=1}^q \Phi_{A_j}^{\phantom{A_j} B_j} \prod_{l=1}^p \Psi_{A_l'}^{\phantom{A_j'}
    B_l'} v_{B_1\ldots B_qB_1'\ldots B_p'},
 \end{equation*} for $\Phi\in{\rm SU}(2n)$, $\Psi\in{\rm SU}(2 )$.
  It is an Hermitian inner product.  The covariant derivatives can be extended naturally to  $(\otimes^r E)\otimes (\otimes^s H)\otimes
  ( \otimes^q E^*)\otimes(\otimes^p H^*)$.
Define $L^2$-inner product $(f,h):=\int_M\langle f,h\rangle dV_g$, where $dV_g $ is the volume form (\ref{eq:vol}) of the metric $g$.
Denote $\|f\|:=(f,f)^{\frac 12}$.

Theorem \ref{thm:Vanishing} is a consequence of the Hodge-type    decomposition in Theorem \ref{thm:BVP} and  the following
Weitzenb\"ock formula,
because the right hand side of (\ref{eq:Weitzenbock}) is negative if the scalar curvature is negative.

\begin{prop}\label{prop:Weitzenbock}  Suppose that $f\in \Gamma(\Lambda^j E^* \otimes\odot^{k-j} H^* )$ satisfying  $ D_j^{(k)}  f=0$, $
D_{j-1}^{(k)*}  f=0$
on   a compact quaternionic K\"ahler  manifold  $M$ of dimension $4n$ (it is right conformally flat if  $\dim_{\mathbb{R}}M=4$),
$j=1,\ldots, k-1 $.  Then we have
 \begin{equation}\label{eq:Weitzenbock}
  \begin{split}
  \left\|\widehat{\nabla }{f} \right\|^2+c_{j } \left\|\nabla^*  f \right\|^2
 = \frac { 2 n-1+(-1)^j}{ (j+1)}( p+2)\int_M\Lambda \left|f  \right |^2dV_g,\end{split}
\end{equation}
where  $p=k-j$,  $c_{j}= \frac { 1-(-1)^j  } {  2(j+1)}\frac {p+2}{ p+1 }\geq 0$,   $\Lambda=\frac {s_g}{8n(n+2)}$, and
 $s_g$ is the scalar curvature.
 \end{prop}

  \begin{rem}On quaternionic K\"ahler  manifold, Semmelmann and Weingart \cite{SW}
derived   Weitzenb\"ock formulae only involving twisted Dirac and twistor operators, but not for our operators $D_j^{(k)} $'s. They
\cite{Se} obtained universal Weitzenb\"ock formula for all irreducible  non-symmetric holonomy groups, and  an recursive procedure for
obtaining   coefficients of Weitzenb\"ock formulas for the holonomy groups ${\rm SO}(n)$, $ G_2$ and ${\rm Spin}(7)$.
But the   recursive formula for   coefficients of Weitzenb\"ock formulae for the holonomy group ${\rm {Sp}}(n ){\rm {Sp}}(1)$ was not
given there. Even if we obtain the recursive formula and know the concrete Weitzenb\"ock formulae,  we also need combine several
Weitzenb\"ock formulae to obtain  the identity (\ref{eq:Weitzenbock}).  It will be a tedious algebraic calculation to derive
(\ref{eq:Weitzenbock}) from universal Weitzenb\"ock formula in \cite{Se}. See also Homma \cite{Hom} for Weitzenb\"ock formulae over
quaternionic K\"ahler  manifolds.
  \end{rem}
\subsection{The formal adjoint operators}
By  two-component   notation, we can derive the formal adjoint operators of
$\nabla$ and $\mathscr D_{q,p  }$ explicitly, while there is no proof for them in Horan
 \cite{Ho} \cite{Ho2}.

\begin{prop}\label{prop:unitary-gamma} On a     quaternionic K\"ahler  manifold, we have    \begin{equation}\label{eq:Z-raise}
      Z^{AA'}=\overline{Z_{AA'}};
    \end{equation}
    and
\begin{equation}\label{eq:unitary-gamma}
  \overline{\Gamma^{\phantom {AA'D} C }_{AA'B}}
    = -\Gamma^{AA'\phantom { B }B
   }_{\phantom { BB' } C },\qquad  \overline{\Gamma^{\phantom {AA'D'} C' }_{AA'B'}}
    = -\Gamma^{AA'\phantom { B' }B'
   }_{\phantom { BB' } C '}.
\end{equation}
Moreover, the formal adjoint
 of $ Z_{AA'}$ is
   \begin{equation}\label{eq:Z_AA'}
   Z_{AA'}^*=- Z^{ AA'}+ {\Gamma}_{\phantom{DB'}D}^{DA'\phantom{ E}
A} + {\Gamma}_{\phantom{BB'}D'}^{A D'\phantom{ E'}
A'}.
\end{equation}
\end{prop}
\begin{proof} (1) By definition,  $Z^{AA'}:=Z_{BB'}\epsilon^{BA}\varepsilon^{B'A'}$ and $\epsilon^{1 0}=-\epsilon^{01}=1$,
$\varepsilon^{1' 0'}=-\varepsilon^{0'1'}=1$. It is direct from definition of $ Z_{AA'}$'s in (\ref{eq:k-CF}) to see that
 \begin{equation*}\begin{split}
   & \overline{Z_{00'}}=Z_{11'}=Z^{00'},\qquad  \overline{Z_{10'}}=-Z_{01'}=Z^{10'},\\&
   \overline{Z_{01'}}=-Z_{10'}=Z^{01'},\qquad  \overline{Z_{11'}}= Z_{00'}=Z^{11'}, \cdots.
    \end{split}
 \end{equation*}

(2) On a     quaternionic K\"ahler  manifold $M$, we know that for any $X\in TM$, the connection $\omega'(X)$ on the bundle $E$  is
  $  \mathfrak {su}(2 n )$-valued. If  write $Z_{AA'}=X+\mathbf{i}Y$ for some  $X, Y\in TM$,  we have
\begin{equation*}\begin{split}
     \overline{\Gamma^{\phantom {AA'D} C }_{AA'B}}&=\overline{\omega'(X +\mathbf{i}Y)_B^{\phantom {A}C}}=\overline{\omega'(X
     )_B^{\phantom
   {A}C}} -\mathbf{i}\overline{\omega'( Y)_B^{\phantom {A}C}} =-
   {\omega'(\overline{Z_{AA'} } )_C^{\phantom {A}B}}  = -\Gamma^{AA'\phantom { B }B
   }_{\phantom { BB' } C }.
\end{split}\end{equation*}
Here we have used $\overline{Z_{AA'} }= {Z^{AA'} }$ in (\ref{eq:Z-raise}).  It is similar for $\Gamma^{\phantom {AA'D'} C' }_{AA'B'}$.

(3) Recall that the formal adjoint $Z_{AA'}^*$
 of $ Z_{AA'}$ satisfies
 \begin{equation}\label{eq:formal adjoint-Z}
    \int_M Z_{AA'} f\cdot \overline{h}dV_g=\int_Mf \cdot\overline{Z_{AA'}^*h}dV_g
 \end{equation}for any compactly supported scalar functions $f$ and $h$.
 Note that $Xf=i_{X }df$, where $i_X$ is the interior operator. By
Stokes' formula, we have
\begin{equation}\begin{split}\int_M Z_{AA'} (f\overline{h})dV_g&=  \int_M d (f\overline{h}) \wedge i_{Z_{AA'} }dV_g  =-\int_M
f\overline{h}\cdot d(i_{Z_{AA'} }dV_g),\end{split}\label{eq:stokes}\end{equation}
where the volume form $dV_g$ is given by (\ref{eq:vol}),  $\{e^{BB'}\}$ is dual to $\{ Z_{BB'}\}$ and
\begin{equation}\label{eq:d-dV}
   d\left(i_{Z_{AA'} }dV_g\right)=\sum_{(B,B')\neq (A,A')}d e^{BB'}\wedge i_{Z_{BB'} } i_{Z_{AA'} }dV_g.
\end{equation}
By the standard
exterior differentiation formula,
$  d  \varphi(X,Y)=\frac 12 [(\nabla_X   \varphi)(Y)-(\nabla_Y  \varphi)(X)+  \varphi (\tau_{X,Y})]
$
for $1$-form $  \varphi\in {\Omega}^1(M)$  and  the torsion
 $\tau_{X,Y}=0$ since the connection is torsion-free, we find that
\begin{equation*}\begin{split}
 2  d e^{BB'}( Z_{CC'}, Z_{DD'})&=\left(-{\Gamma}_{CC'E}^{\phantom{CC'E}
B}e^{EB'}-{\Gamma}_{CC'E'}^{\phantom{CC'E'}
B'}e^{BE'}\right)(Z_{DD'})-CC'\leftrightarrow DD'\\&
= (-{\Gamma}_{CC'D}^{\phantom{CC'E}
B} \delta_{D'}^{B'}-{\Gamma}_{CC'D'}^{\phantom{CC'E'}
B'}\delta_{D }^{B })-CC'\leftrightarrow DD'.
\end{split}\end{equation*} Using $\omega_1\wedge \omega_2(X,Y):= \frac 12[\omega_1(X )\omega_2( Y)-\omega_1( Y)\omega_2(X )]$,
  we get the Cartan formula
\begin{equation} \label{eq:dth}d e^{BB'}=  - \left( {\Gamma}_{CC'D}^{\phantom{CC'E}
B} \delta_{D'}^{B'}+{\Gamma}_{CC'D'}^{\phantom{CC'E'}
B'}\delta_{D }^{B }  \right)e^{CC'}\wedge e^{DD'}  .
\end{equation}Note that for fixed $AA',BB'$, we can write $dV_g=e^{AA'}\wedge e^{BB'}\wedge \widetilde{\omega}$ with the $(4n-2)$-form
$\widetilde{\omega}=i_{Z_{BB'} } i_{Z_{AA'} }dV_g$.
Substitute it and (\ref{eq:dth}) into (\ref{eq:d-dV})  to get
\begin{equation}\label{eq:d-dV2}\begin{split}
d\left(i_{Z_{AA'} }dV_g\right)=- \sum_{(B,B')\neq (A,A')}\left(  {\Gamma}_{AA'B}^{\phantom{CC'E}
B}  +{\Gamma}_{AA'B'}^{\phantom{CC'E'}
B'} -{\Gamma}_{BB'A}^{\phantom{CC'E}
B} \delta_{A'}^{B'}-{\Gamma}_{BB'A'}^{\phantom{CC'E'}
B'}\delta_{A }^{B }\right)dV_g.
\end{split}\end{equation} Now substituting   the conjugate of (\ref{eq:d-dV2}) into the right hand side of (\ref{eq:stokes}),
we find that for $A$ and $A'$ fixed,
\begin{equation}\label{eq:formal adjoint-Z'}\begin{split}
  Z_{AA'}^*=- Z^{ AA'}- \sum_{(B,B')\neq (A,A')}  \left({\Gamma}_{\phantom{BB'}B}^{AA'\phantom{ E}
B} + {\Gamma}_{\phantom{BB'}B'}^{AA'\phantom{ E'}
B'}\right)+\sum_{B\neq A}{\Gamma}_{\phantom{BB'}B}^{BA'\phantom{ E}
A} +\sum_{B'\neq A'}{\Gamma}_{\phantom{BB'}B'}^{A B'\phantom{ E'}
A'},
\end{split}\end{equation}by using  definition (\ref{eq:formal adjoint-Z}) and $\overline{Z_{AA'} }= {Z^{AA'} }$ in (\ref{eq:Z-raise}).
But for $A$ and $A'$ fixed, we have
\begin{equation}\label{eq:sum-gamma1}
   \sum_{(B,B')\neq (A,A')}{\Gamma}_{\phantom{BB'}B}^{AA'\phantom{ E}
B}={\Gamma}_{\phantom{BB'}A}^{AA'\phantom{ E}
A}+2\sum_{B \neq A }{\Gamma}_{\phantom{BB'}B}^{AA'\phantom{ E}
B}=-{\Gamma}_{\phantom{BB'}A}^{AA'\phantom{ E}
A}
\end{equation}
by tr ${\Gamma}_{a*}^{ \phantom{a*}
*}=0$ since it is also  $  \mathfrak {sl}(2 n , \mathbb{C})$-valued, and  similarly
\begin{equation}\label{eq:sum-gamma2}
   \sum_{(B,B')\neq (A,A')}{\Gamma}_{\phantom{BB'}B'}^{AA'\phantom{ E'}
B'}= \sum_{ B' \neq  A' }{\Gamma}_{\phantom{BB'}B'}^{AA'\phantom{ E'}
B'}+(2n-1)\sum_{ B' }{\Gamma}_{\phantom{BB'}B'}^{AA'\phantom{ E'}
B'}=-{\Gamma}_{\phantom{BB'}A'}^{AA'\phantom{ E'}
A'},
\end{equation}by tr ${\Gamma}_{a*'}^{ \phantom{ a*'}
*'}=0$.
  Now (\ref{eq:Z_AA'})  follows from substituting   (\ref{eq:sum-gamma1})-(\ref{eq:sum-gamma2}) into (\ref{eq:formal adjoint-Z'}).
\end{proof}
\begin{prop} \label{prop:adjoint} On a     quaternionic K\"ahler  manifold, we have

 (1) The formal adjoint $\nabla^* :\Gamma(\mathfrak T_{q+1,p+1})\rightarrow\Gamma(\mathfrak T_{q,p})$ of  $\nabla$ is given by
\begin{equation*}  \begin{split}
   (\nabla^*  f)_{A_1\ldots A_q A_1'\ldots A_p'} =&-\nabla^{A A '}  f_{A A_1\ldots A_qA 'A_1'\ldots A_p'}
    ;
 \end{split}\end{equation*}
(2)    the formal adjoint $\mathscr D_{q,p  }^*: \Gamma(\Lambda^{q+1 } E^*\otimes \odot^{p-1}H^*)\longrightarrow\Gamma(\Lambda^{q}
E^*\otimes\odot^{p
}H^*)$ of  $\mathscr D_{q,p  }$ is given by
\begin{equation*}  \begin{split}
   (\mathscr D_{q,p  }^*  f)_{A_1\ldots A_q A_1'\ldots A_p'}=& \nabla^{A}_{(A_1 '}  f_{| A A_1\ldots A_q | A_2'\ldots A_p')}
    .
 \end{split}\end{equation*}
\end{prop}
\begin{proof} (1) For any local section $h\in \Gamma(\mathfrak T_{q,p})$, we have
\begin{equation*}
  \begin{split}(\nabla  h,  f)&=
   \int{\nabla_{BB'}   h_{A_1\ldots A_q A_1'\ldots A_p'}} \overline{  f_{BA_1\ldots A_qB'A_1'\ldots A_p'}}dV_g\\&
  =\int \left[{Z_{ BB'}} {  h_{A_1\ldots A_q A_1'\ldots A_p'}}- {\Gamma_{BB'  A_j}^{\phantom{A A 'A_j }D}} {  h_{A_1\ldots D\ldots A_q
  A_1'\ldots
  A_p'}}\right. \\&
  \left.\qquad\qquad\qquad\qquad\qquad\qquad
  - {\Gamma_{BB'A_j' }^{\phantom{A A 'A_j'}D'}} {  h_{A_1\ldots A_q A_1'\ldots D'\ldots A_p'}}\right]\overline{  f_{BA_1\ldots
  A_qB'A_1'\ldots A_p'}}dV_g\\&
   =\int{{  h}_{A_1\ldots A_qA_1'\ldots A_p'}}\left[ -\overline{Z^{BB'}  } + \overline{\Gamma_{\phantom{BB'}  D }^{DB'\phantom{ D}B  }}
   +   \overline{\Gamma_{\phantom{BB'} D'}^{BD'\phantom{ A_j'}B'} }  \right] \overline{{  f}_{BA_1\ldots A_qB'A_1'\ldots A_p'}}\\&
   \qquad  +\left[{ {  h}_{A_1\ldots D\ldots A_q A_1'\ldots A_p'}} \overline{\Gamma_{\phantom{BB'}  D }^{BB'\phantom{ D}A_j  } }
   +{  h}_{A_1\ldots A_q A_1'\ldots D'\ldots A_p'}  \overline{\Gamma_{\phantom{BB'} D'}^{BB'\phantom{ A_j'}A_j'}}\right] \overline{{
   f}_{BA_1\ldots
   A_qB'A_1'\ldots A_p'} }
   ,
   \end{split}
\end{equation*}
by using (\ref{eq:Z_AA'}) for the expression
of
the formal adjoint $Z_{BB'}^*$ and  (\ref{eq:unitary-gamma})  for the connection coefficients. Then we can relabel indices so that the
right hand side can be written as $(  h, \nabla^*  f)$ with
\begin{equation*}  \begin{split}
   (\nabla^*  f)_{A_1\ldots A_qA_1'\ldots A_p'}
   =&-{Z^{A_0A_0'}}  f_{A_0A_1\ldots A_qA_0'A_1'\ldots A_p'} +{\Gamma_{\phantom{ A_0A_0'}A_j }^{A_0A_0'\phantom{ A_j } E}}  f_{  \ldots
   E \ldots A_0'
   \ldots A_p'}+\Gamma_{\phantom{A_0A_0'}A_j'}^{A_0A_0'\phantom{ B_j'}  E'}  f_{A_0A_1\ldots A_q  \ldots  E' \ldots  }.
 \end{split}\end{equation*}

(2)  For any local sections  $h\in \Gamma(\Lambda^{q   } E^*\otimes \odot^{p }H^*)$ and $f\in \Gamma(\Lambda^{q +1  } E^*\otimes
\odot^{p-1}H^*)$, we have
\begin{equation*}
  \begin{split}(\mathscr D_{q,p  }   h,   f)&
  =\int\nabla_{ [A }^{A '}   h_{A_1\ldots A_q] A 'A_2'\ldots A_p'}\overline{{  f}_{A A_1\ldots A_q A_2'\ldots A_p'}}dV_g
  \\&=-\int\nabla_{ A B ' }
  h_{A_1\ldots A_qA 'A_2'\ldots A_p'}\varepsilon_{B'A'}\cdot \overline{{  f}_{A A_1\ldots A_q A_2'\ldots A_p'}}dV_g \\&
 = \int  h_{A_1\ldots A_qA 'A_2'\ldots A_p'} \overline{\nabla^{ A B ' }   {  f}_{A A_1\ldots A_q A_2'\ldots
 A_p'}\varepsilon_{B'A'}}dV_g\\&
 = \int  h_{A_1\ldots A_qA 'A_2'\ldots A_p'} \overline{  \nabla_{( A '}^{ A }     f_{|A A_1\ldots A_q| A_2'\ldots A_p')}} dV_g
     =(  h,\mathscr D_{q,p }^*  f)  \end{split}
\end{equation*}by using (1) and $\varepsilon^{B'A'}=-\varepsilon_{B'A'}$. Here we drop antisymmetrisation in the second identity by
\begin{equation}\label{eq:bracket-omit}
  \sum_{B_0, \ldots, B_q} \left(    g_{[B_0 \ldots B_q] \mathscr B'},\widetilde{  g}_{[B_0 \ldots B_q] \mathscr B'}\right)=\sum_{B_0,
  \ldots,
  B_q} \left(   g_{ B_0 \ldots B_q  \mathscr B'}, \widetilde{   g}_{ [B_0 \ldots B_q] \mathscr B'}\right)
\end{equation}
for any local sections $  g,\widetilde{ g}\in \Gamma(\otimes^{q+1 } E^* \otimes (\otimes^{p}H^*))$, and add  symmetrisation in the last
identity by
  \begin{equation}\label{eq:bracket-omit-sys}
  \sum_{B_1', \ldots, B_p'} \left(    h_{\mathscr B (  B_1 '\ldots B_p')},\widetilde{  h}_{\mathscr B B_1 '\ldots
  B_p'}\right)=\sum_{B_1', \ldots, B_p'} \left(   h_{ \mathscr B  ( B_1 '\ldots B_p') }, \widetilde{  h}_{ \mathscr B  (  B_1 '\ldots
  B_p') }\right)
\end{equation}
for any local sections $  h,\widetilde{ h}\in  \Gamma(\otimes^{q } E^* \otimes (\otimes^{p}H^*))$.
(\ref{eq:bracket-omit})-(\ref{eq:bracket-omit-sys}) follows from definition easily. See (3.4) in \cite{L2} for the proof of
(\ref{eq:bracket-omit-sys}).
 \end{proof}

\subsection{A   Weitzenb\"ock formula}  The  Weitzenb\"ock  formula for the  $k$-Cauchy-Fueter complex  is much more complicated than
that for the De Rham complex, since we have not only exterior forms, but also symmetric forms. Note that $ D_{j-1}^{(k)*}  f=0$, i.e.
$\mathscr D^*_{q-1 ,p +1}  f=0$ for $q=j$, $p=k-j  $,  implies that
\begin{equation}\label{eq:adjoint-condition}
  \nabla_{ (B_1'}^{ A }  {  f}_{|AB_1\ldots B_{q-1}  |B_2'\ldots B_{p+1}')}=0.
\end{equation}by Proposition \ref{prop:adjoint} (2).
The following is another version of lemma 2.1 of Horan \cite{Ho}  without proof.
\begin{lem}\label{lem:phi-partial}  For a tensor $f\in \Gamma(\mathfrak T_{q,r})$ such that
\begin{equation}\label{eq:condition-sys}
   f_{\mathscr A A_1'A_2'\ldots A_r'} =  f_{\mathscr A
A_1'(A_2'\ldots A_r')}\qquad {\rm and} \qquad    f_{\mathscr A(A_1'A_2'\ldots A_r')}=0 ,
\end{equation}
we have
\begin{equation}\label{eq:condition-sys0}
      f_{\mathscr A A_1'A_2'\ldots A_r'} =-\frac {r-1}r\varepsilon_{A_1'(A_2'}  f_{|\mathscr A |\phantom{ A'}A_3'\ldots
      A_r')C'}^{\phantom{|\mathscr A |
      }C'}.
\end{equation}
\end{lem}
\begin{proof} Suppose that $A_1'+\cdots +A_r'=l$, i.e. there are $l$'s $1'$ in $\{A_1',\ldots, A_r'\}$. Note that
(\ref{eq:condition-sys})
 implies that $ f_{\mathscr A A_1'(A_2'\ldots A_r')} + \cdots+ f_{\mathscr A A_j'(A_2'\ldots A_1' \ldots A_r')}+\cdots=0 $, and so
\begin{equation}\label{eq:1-0}
  (r-l)   f_{\mathscr A0'0'\ldots 0'1'1'\ldots}+l  f_{\mathscr A1'0'\ldots  0' 0'1'\ldots }=0.
\end{equation} Assume that $A_1'=0'$ in (\ref{eq:condition-sys0}).
Then
\begin{equation*}\begin{split}
    \varepsilon_{0'(A_2'}  f_{|\mathscr A |\phantom{ A'}A_3'\ldots A_r')C'}^{\phantom{|\mathscr A | }C'}&=\frac l{r-1}\varepsilon_{0'1'}
    f_{ \mathscr A
    \phantom{ A'}0' \ldots 0'1'\ldots 1'C'}^{\phantom{ \mathscr A  }C'}=\frac l{r-1}( - f_{ \mathscr A 0'0' \ldots 0'1'\ldots 1'1'} +
    f_{ \mathscr A
     1'0' \ldots 0'1'\ldots 1'0'})\\&=-\frac r{r-1}   f_{ \mathscr A  0'0' \ldots 0'1'\ldots 1'1'}
\end{split}\end{equation*}
by using  \begin{equation}\label{eq:up-down}
  f_{\ldots\phantom{ B'}\ldots B'\ldots }^{\phantom{\ldots } B'  }= - f_{\ldots 0' \ldots 1' \ldots}+ f_{\ldots 1' \ldots 0' \ldots} =-
  f_{\ldots B'\ldots\phantom{ B'} \ldots }^{\phantom{\ldots B'\ldots} B'  }
.
\end{equation} in the second identity and (\ref{eq:1-0}) in the last identity. It is similar for $A_1'=1'$.
\end{proof}

{\it Proof of Proposition \ref{prop:Weitzenbock}}. Since $ D_j^{(k)}  f=0$, i.e.
$\mathscr D_{q,p}  f=0$  for $q=j$, $p=k-j  $, we have
\begin{equation}\label{eq:=0}
 0 =  (\mathscr D_{q,p }   f,  \mathscr D_{q,p }  f)= ( \mathscr D_{q,p }^*\mathscr D_{q,p }   f,   f).
\end{equation}
 Then by choose local orthonormal quaternionic frame as before, locally we have
$
   (  \mathscr D_{q,p }^*\mathscr D_{q,p }   f )_{B_1\ldots B_q  B_1'\ldots
 B_p'}=
 \nabla^{ A}_{( B_1' }\nabla^{A'}_{|[A}  f_{B_1\ldots B_q]A'|B_2'\ldots B_p')}.$
To calculate\begin{equation*}\left(\nabla^{ A}_{( B_1' }\nabla^{A'}_{|[A}  f_{B_1\ldots B_q]A'|B_2'\ldots B_p')}, f_{B_1\ldots B_qB_1'
B_2'\ldots B_p' }\right ),\end{equation*} we only need to calculate the term without  symmetrisation by (\ref{eq:bracket-omit-sys}),
i.e.
\begin{equation}\label{eq:curvature-term-00}\begin{split}
   \nabla^{ A}_{  B_1' }\nabla^{A'}_{ [A}  f_{B_1\ldots B_q]A' B_2'\ldots B_p' }  &    = \varepsilon_{B'B_1'}
 \left(\nabla^{ A [B' }\nabla^{A']}_{[A} + \nabla^{ A (B' }\nabla^{A')}_{[A}\right)  f_{B_1\ldots B_q]A'B_2'\ldots B_p'}.
\end{split}\end{equation}
For the first term in the right hand side, note that \begin{equation*}\begin{split} \nabla^{ A [B' }\nabla^{A']}_{C} &=
   \varepsilon_{B'A'} \nabla^{ A [0' }\nabla^{1']}_{C} =\frac {\varepsilon_{B'A'}}2\left ( \nabla^{ A  0' }\nabla^{1' }_{ C}- \nabla^{ A
   1' }\nabla^{0'
   }_{C}\right) =- \frac {\varepsilon_{B'A'}}2 \left (\nabla^{ A0' }\nabla_{ C0'}+\nabla^{ A 1'
  }\nabla_{ C 1'}\right),
\end{split}\end{equation*}by $\varepsilon_{0'1'}= 1$,  and  $\sum_{B'=0',1'}\varepsilon_{B'A'} \varepsilon_{B'B_1'} =\delta_{A'B_1'}$.
 If
  we display the antisymmetrisation of the unprimed indices in the second term in (\ref{eq:curvature-term-00}), we see that
   \begin{equation}\label{eq:curvature-term-0}
  \begin{split}
 & \nabla^{ A}_{  B_1' }\nabla^{A'}_{ [A}  f_{B_1\ldots B_q]A' B_2'\ldots B_p' } = - \frac 12   \nabla^{ A
 C'}\nabla_{C'[A}f_{B_1\ldots B_q] B_1'\ldots B_p'}
  \\& \qquad \qquad \qquad   \qquad    + \frac{\varepsilon_{B'B_1'}}{q+1} \left  (\nabla^{ A (B' }\nabla^{A')}_{
 A} f_{B_1\ldots   B_q A'B_2'\ldots B_p'}+  \sum_{s=1}^q (-1)^s \nabla^{ A (B' }\nabla^{A')}_{B_s} f_{ \ldots A \ldots    A'B_2'\ldots
 B_p'} \right )  \\
 &= - \frac 12   \nabla^{ A    C'}\nabla_{C'[A}f_{B_1\ldots B_q] B_1'\ldots B_p'}+ \frac{\varepsilon_{B'B_1'}}{q+1} \nabla^{ A (B'
 }\nabla^{A')}_{
 A}   f_{B_1\ldots   B_q A'B_2'\ldots B_p'} \\&  \qquad  \qquad\qquad    \qquad    + \frac{\varepsilon_{B'B_1'}}{q+1}
 \left(\sum_{s=1}^q     {(-1)^s   }\left (\nabla^{ A (B' }\nabla^{ A')}_{B_s}-\nabla^{( A' }_{B_s}\nabla^{ B' )A } \right)
 \right.\\&\left.\qquad  \qquad \qquad  \qquad \qquad    \qquad   + \frac {1 }2\sum_{s=1}^q     {(-1)^s   }  \nabla^{B' }_{B_s}\nabla^{
 A' A }  +  \frac {1   }2 \sum_{s=1}^q     {(-1)^s   }\nabla^{
 A' }_{B_s}\nabla^{ B'  A }  \right )   f_{B_1\ldots A \ldots  B_q A'B_2'\ldots B_p'}
\\& :=(S_1f+ S_2f+S_3f+S_4f+S_5f)_{B_1\ldots B_q  B_1'\ldots
 B_p'}.\end{split}
\end{equation} Now we get $0=(S_1f,f) +\ldots +(S_5f,f)$ by   (\ref{eq:=0}). The reason we use the expansion   above is that $S_2f$ and
$S_3f$ are commutators of the form $\nabla_{[A}^{  (A'
 }\nabla^{B ')}_{ B]}$, which are curvature terms.
   Obviously by using (\ref{eq:bracket-omit}) and Proposition \ref{prop:adjoint}, we have
  \begin{equation}\label{eq:S1}
      (S_1f,f)= \frac 12\left ( {\nabla}^* \widehat{\nabla}{  f} ,{{  f}}   \right )  = \frac 12\left ( \widehat{\nabla}{  f} ,
      {\nabla}{{  f}} \right )= \frac 12\left ( \widehat{\nabla}{  f} , \widehat{\nabla}{{  f}}   \right )   \end{equation}
by antisymmetrisation  (\ref{eq:bracket-omit}) in the second identity.

Recall that   $\Phi=0$ for   quaternionic K\"ahler manifold, and so we have
\begin{equation}\label{eq:Kahler-vanishing}
   R_{  \phantom{ ( A'B' ) } A B C
   }^{ ( A'B')\phantom{  C A B_j } D  } =0,\qquad R_{  AB \phantom{ A'B'} C'  }^{\phantom{   AB } A'B'\phantom{C'}D '} = 2\Lambda_{
   AB}\delta_{
   C'}^{\phantom{
 B'}(A'}\varepsilon^{B')D'}
\end{equation}
by Corollary \ref{lem:R-upper} and Proposition \ref{prop:curvature}. Note that
\begin{equation} \label{eq:curvature-term-S1''}
  \begin{split}   {\varepsilon_{B'B_1'}}
      \nabla_{[C}^{  (B'
 }\nabla^{B_0')}_{ A]}  f_{B_1\ldots B_q B_0'B_2'\ldots B_p'}
   =& - \frac 12  \varepsilon_{B'B_1'}  \left( R_{ [C A]\phantom{ B'   B_0' }B_j '
   }^{\phantom{ [C A] }  B' B_0' \phantom{ D' } D' }  f_{B_1\ldots B_q B_0' \ldots D'\ldots B_p'}\right.\\&\left.\qquad\qquad\quad+  R_{
   \phantom{ ( B' B_0') } C A B_j
   }^{ ( B' B_0')\phantom{  C A B_j  } D  }  f_{B_1\ldots D\ldots B_qB_0' \ldots  B_p'}\right) \\
   =&   -   {\varepsilon_{B'B_1'}}  \Lambda_{ C A   } \delta_{B_j'}^{\phantom{B_j'} ( B' }  \varepsilon^{B_0')D'} f_{B_1\ldots B_q
   \ldots D'\ldots
   B_p'}
  \\
   =&    \frac {p+2}2   \epsilon_{ C A   }\Lambda   f_{B_1\ldots B_q B_1'  B_2'\ldots
   B_p'}
     \end{split}
\end{equation}by using  (\ref{eq:bracket-curvature}) for the  commutator, Proposition \ref{prop:Einstein} for $\Lambda_{AB}$ and
(\ref{eq:Kahler-vanishing}), and for $j= 2,\ldots, p$,
\begin{equation*}\begin{split}
   2\varepsilon_{B'B_1'} \delta_{B_j'}^{\phantom{B_j'} ( B' }  \varepsilon^{B_0')D'} f_{B_1\ldots B_q B_0' \ldots D'\ldots
   B_p'}=& \left(\varepsilon_{B_j'B_1'}   \varepsilon^{B_0' D'}  +\delta_{B_j'}^{\phantom{B_0'}   B_0' } \varepsilon_{B 'B_1'}
   \varepsilon^{B'  D'} \right) f_{B_1\ldots B_q B_0' \ldots D'\ldots
   B_p'}\\=&- f_{B_1\ldots B_q B_j' \ldots B_1'\ldots
   B_p'}
  \end{split}\end{equation*}
since
  $\varepsilon^{ B_0' D'} $  is antisymmetric and $f_{ \ldots}$ is symmetric in primed indices, while for $j=0$,
  \begin{equation*}
     \begin{split}
   2\varepsilon_{B'B_1'} \delta_{B_0'}^{\phantom{B_j'} ( B' }  \varepsilon^{B_0')D'} f_{B_1\ldots B_q  D'B_2'\ldots
   B_p'}=& \left(\varepsilon_{B_0'B_1'}   \varepsilon^{B_0' D'}  +\delta_{B_0'}^{\phantom{B_0'}   B_0' } \varepsilon_{B 'B_1'}
   \varepsilon^{B'  D'} \right) f_{B_1\ldots B_q D'B_2'\ldots
   B_p'}\\=&- 3f_{B_1\ldots B_q B_1'B_2'\ldots
   B_p'}.
  \end{split}
  \end{equation*}
   Then, lowering the superscript $A$ in $S_2f$  and applying (\ref{eq:curvature-term-S1''}) to $S_2f$, we  get
\begin{equation}\label{eq:curvature-term-S1'}
  \begin{split} (S_2f)_{B_1\ldots B_q  B_1'\ldots
 B_p'} =   - \frac { p+2} {q+1}  n  \Lambda{  f}_{B_1\ldots B_q   B_1'\ldots B_p'  },
    \end{split}
\end{equation}
 and similarly, applying (\ref{eq:curvature-term-S1''}) to $S_3f$, we  get
 \begin{equation*}
  \begin{split} (S_3f)_{B_1\ldots B_q  B_1'\ldots
 B_p'} = &2 \sum_{s=1}^q \frac {(-1)^s \epsilon^{CA}}{q+1} \varepsilon_{B'B_1'} \nabla_{[C}^{  (B'
 }\nabla^{B_0')}_{ B_s]}   f_{B_1\ldots A \ldots
 B_qB_0'B_2' \ldots
  B_p'} \\ =&  \sum_{s=1}^q \frac {p+2
   }{ q+1 }  (-1)^s \epsilon^{CA}\epsilon_{ C B_s   }   \Lambda{
   f}_{B_1\ldots   A \ldots  B_q  B_1'\ldots B_p'  }  =  \frac { 1-(-1)^q}{ 2  } \frac {p+2}{q+1}\Lambda  {
   f}_{   B_1\ldots  B_q  B_1'\ldots B_p'  }.
    \end{split}
\end{equation*}   Hence
\begin{equation}\label{eq:S2-3}
 (S_2f,f) + (S_3f,f)=-    \frac {  2n-1+(-1)^q }{ 2(q+1) }(p+2)\left ( \Lambda{  f} , {{  f}}   \right ).
\end{equation} Similarly we have
\begin{equation*}\begin{split}
(S_4f)_{B_1\ldots B_q  B_1'\ldots
 B_p'}&=     \sum_{s=1}^q\frac { -1 }{2(q+1)}   \nabla_{B_s B_1'}\nabla^{ A' A }    f_{A B_1\ldots\widehat  B_s \ldots  B_qA' B_2'\ldots
 B_p'}  \\
&= \frac {1}{2(q+1)} \sum_{s=1}^q\left(\nabla \nabla^{ * }    f\right)_{B_s B_1\ldots  \widehat  B_s \ldots  B_q B_1' B_2' \ldots B_p'}
    \end{split}\end{equation*} by $f $  antisymmetric in unprimed indices  and  the expression of $\nabla^{ * } $ in  Proposition
    \ref{prop:adjoint}, and
\begin{equation*}\begin{split}
 (S_5&f)_{B_1\ldots B_q  B_1'\ldots
 B_p'}= - \sum_{s=1}^q\frac {1 }{2(q+1)}    \nabla^{ A' }_{B_s}\nabla^{   A }_{B_1'}    f_{A B_1\ldots\widehat  B_s \ldots
 B_qA'B_2'\ldots B_p'}  \\&
=  \sum_{s=1}^q\frac {1}{2(q+1)} \frac {p }{p+1}  \nabla^{ A' }_{B_s}\varepsilon_{B_1'(A' }\nabla^{ A C' }    f_{|A B_1\ldots\widehat
B_s \ldots     B_q
|B_2'\ldots B_p')C'}
\\&
=  \sum_{s=1}^q\frac {1}{2(q+1)}\frac { 1}{p+1}\left[ - \nabla_{ B_s B_1' }  \nabla^{ A C' }    f_{A B_1\ldots \widehat  B_s \ldots
B_q B_2'\ldots B_p' C'}
+ \varepsilon_{B_1' B_j' }  \nabla^{ A' }_{B_s}  \nabla^{ A C' }    f_{A B_1\ldots \widehat  B_s \ldots      B_q B_2'\ldots A' \ldots
B_p' C'}\right]\\&
=   \frac { 1 }{2(q+1)(p+1)}  \sum_{s=1}^q  \left(\nabla \nabla^{ * }    f\right)_{B_s B_1\ldots \widehat  B_s \ldots   B_q B_1' B_2'
\ldots B_p'} +\widetilde{S}_5 f
    \end{split}\end{equation*}
where in the second identity we apply Lemma \ref{lem:phi-partial} to $ \nabla_{ B_1'}^{ A }    f_{ B_1\ldots A \ldots  B_q A'B_2'\ldots
B_p'} $ with
$r=p+1$,  since
the condition of this lemma is satisfied by (\ref{eq:adjoint-condition}). Note that  $(\widetilde{S}_5 f,f)=0$ by $f$  symmetric but
$\varepsilon_{B_1' B_j' }$
antisymmetric in $B_1', B_j'$. Therefore
\begin{equation}\label{eq:S4-5}
 (S_4f,f) + (S_5f,f)=     \frac {  (p+2)  }{2(q+1)(p+1)} \sum_{s=1}^q (-1)^{s-1}\left ( \nabla^{ * }    f ,\nabla^{ * }    f \right )=
 \frac
 {  (1-(-1)^q) (p+2)  }{4(q+1)(p+1)}  \left ( \nabla^{ * }    f ,\nabla^{ * }    f \right ).
\end{equation}
Substituting (\ref{eq:S1}),  (\ref{eq:S2-3}) and (\ref{eq:S4-5}) into $0=(S_1f,f) +\ldots +(S_5f,f)$, we get  the  identity
(\ref{eq:Weitzenbock}). \hskip 12mm$\Box$

\begin{appendix}
 \section{ Proof of some propositions}
  At first, the traces of curvatures vanish when we  antisymmetrise  primed (or unprimed) indices and symmetrise  unprimed (or primed) indices.
 \begin{prop} The curvatures of a unimodular   quaternionic   manifold satisfy
    \begin{equation}\label{eq:vanishing}\begin{split}
   R_{[A'B'] C(A B) }^{\phantom {[A'B'] C(A B) }   C }&
    =0,\,\,\, \qquad R_{(A'B') C[A B] }^{\phantom {(A'B') C[A B]  }   C }=0,\\ R_{[ A B ] C'(A' B') }^{\phantom {[ A B ] C'(A' B') }   C' }
    &=0, \qquad R_{( A B ) C'[A' B'] }^{\phantom {( A B ) C'[A' B']  }   C '}=0.
\end{split}\end{equation}
 \end{prop}
 \begin{proof} The first Bianchi identity \begin{equation}\label{eq: Bianchi1}
    R_{[abc]}^{\phantom{[abc]}d}=0,
 \end{equation} and the antisymmetry of $R_{abc}^{\phantom {abb}d}$ in $a$ and $b$ implies
  the cyclicity:
 \begin{equation}\label{eq:cyclicity}
    R_{ AA'B B' C C' }^{\phantom{AA'B B' C C'}     DD'}+
    R_{B B' C C' AA' }^{\phantom{B B' C C' AA' }     DD'}+
    R_{ C C' AA'B B'}^{\phantom{C C' AA'B B'}    DD'}=0
 \end{equation}
 for all $A,\ldots,A',\ldots$.
  Take trace over $C'$-$D'$ in (\ref{eq:cyclicity}) to get
 \begin{equation}\label{eq: Trace-C'-D'-2}
    2R_{A'B' AB C   }^{\phantom {A'B' AB C }     D }+R_{B'A'B CA  }^{\phantom {B'A'B CA}      D }+R_{B'A' CA B }^{\phantom {B'A' CA B}     D }
    +R_{BC B' C'A'   }^{\phantom {BC B' C'A'}      C' }\delta_{A }^{\phantom {A}D }+R_{C A C' A'B'   }^{\phantom {C A C' A'B' }      C' }\delta_{B
    }^{\phantom {A}D }=0.
 \end{equation} by the decomposition  (\ref{eq:decomposition}) of curvatures. Here and in the sequel we use
  vanishing of  traces in
  (\ref{eq:vanishing-trace}) repeatedly.

 Taking trace over  $C$-$D$  in (\ref{eq: Trace-C'-D'-2}), we get
 \begin{equation} \label{eq: Trace-C-D}
    R_{B'A'B CA  }^{\phantom {B'A'B CA}     C}+R_{B'A' CA B }^{\phantom {B'A' CA B}      C }
    +R_{BA B' C'A'   }^{\phantom {BA B' C'A' }      C' } +R_{B A C' A'B'   }^{\phantom {B A C' A'B'}    C' } =0,
 \end{equation}by using (\ref{eq:vanishing-trace}) again.
 Antisymmetrising $[A'B']$ and symmetrising $(A B)$ in (\ref{eq: Trace-C-D}), we get
 \begin{equation}\label{eq:vanishing-[B'A']-(A B)}
   R_{[B'A'] C(A B) }^{\phantom {[B'A'] C(A B)}      C }+R_{(A B) C'[ A' B'] }^{\phantom {(BA ) C'[ A' B']}      C '}
    =0,
 \end{equation}
 by  using
  \begin{equation}\label{eq:sym}
   R_{[B'A'] B C A}^{\phantom {[B'A']B CA  }      D}= R_{[B'A']C B A }^{\phantom {[B'A']CBA  }     D}\qquad {\rm  and} \qquad
   R_{(BA)B'C'A' }^{\phantom {(BA)  B' C' A'} D' }=-R_{(BA) C'B'A'   }^{\phantom {(BA)C' B'A'}     D' }  \end{equation}
   which follows from antisymmetry
   \begin{equation}\label{eq:antisym}
   R_{ A'B' A B C }^{\phantom { B'A' B CA  }      D}=-R_{B'A' B A C }^{\phantom { B'A' B CA  }      D},\qquad  R_{ A B  A' B' C' }^{\phantom { B'A' B CA'
   }      D'}=-R_{B A  B' A 'C '}^{\phantom { B'A' B CA'  }      D'},\qquad \end{equation} i.e., $\quad  R_{ ab C }^{\phantom {abA  }      D}=-R_{ba C
   }^{\phantom { baA  }      D}$ and $\quad  R_{ ab C' }^{\phantom {abA'  }      D'}=-R_{ba C' }^{\phantom { baA'  }      D'}$.
Here  (\ref{eq:sym}) means that $R_{[B'A']B C A }^{\phantom {[B'A']CBA  }     D}$ is symmetric in $B,C$, while $ R_{(BA)B'C'A' }^{\phantom {(BA)  B' C'
A'} D' }$ is antisymmetric in $B',C'$. These identities  will be used frequently to change the order of indices.

Taking trace over  $B$-$D$  in (\ref{eq: Trace-C'-D'-2}),
  we get
 \begin{equation} \label{eq: Trace-B-D}
     2R_{A'B' AB C   }^{\phantom {A'B' AB C }      B }+R_{B'A'B CA  }^{\phantom {B'A'B CA }      B }
    +R_{AC B' C'A'   }^{\phantom {AC B' C'A'}     C' } +2nR_{C A C' A'B'   }^{\phantom {C A C' A'B'}     C' } =0.
 \end{equation}
Antisymmetrising $[A'B']$ and symmetrising   $(AC)$ in (\ref{eq: Trace-B-D}), we see that
 \begin{equation}\label{eq:vanishing-[B'A']-(A B)2}
    R_{[A'B'] B (A C)   }^{\phantom {[A'B'] B (A C)}      B }+(2n+1)R_{( A C) C' [A'B']   }^{\phantom {(C A) C' [A'B']}      C' }=0,
 \end{equation} by (\ref{eq:sym}).
Now  (\ref{eq:vanishing-[B'A']-(A B)}) and (\ref{eq:vanishing-[B'A']-(A B)2}) imply that the first and the last identities in
  (\ref{eq:vanishing}).

Symmetrising $(A' B')$ and antisymmetrising $[A B ]$ in (\ref{eq: Trace-C-D}), we get
 \begin{equation}\label{eq:()-[]}
     R_{(B'A') C[A B] }^{\phantom {(B'A') C[A B]}      C }
  -   R_{[A B]C' (B'A')   }^{\phantom {[A B]C' (B'A')}     C' } =0.
 \end{equation}
On the other hand, we take trace over $C$-$D$ in (\ref{eq:cyclicity}) to get
  \begin{equation}\label{eq: Trace-C-D-1}
   R_{B'C'B CA   }^{\phantom {B'C'B CA }    C }\delta_{A' }^{\phantom { B'}D' }+R_{C'A' C A B }^{\phantom {C'A' C A B }     C }\delta_{B' }^{\phantom {
   B'}D' }
  +2nR_{ A B A' B' C'  }^{\phantom {A B A' B' C'  }     D' }  +R_{  B A B' C'A'   }^{\phantom { B A B' C'A'}      D' } +R_{BA C' A'B'   }^{\phantom {BA C'
  A'B'}     D' } =0.
 \end{equation}
Taking trace over $B'$-$D'$  in (\ref{eq: Trace-C-D-1}), we get
 \begin{equation}\label{eq: Trace-C-D-2}
     R_{ A'C' B CA   }^{\phantom { A'C' B CA}     C } +2R_{C'A' C A B }^{\phantom {C'A' C A B}      C }
  +2nR_{ A B A' B' C'  }^{\phantom {A B A' B' C' }      B' }  +R_{  BA B'  C'A'   }^{\phantom {BA B'  C'A'}      B' } =0.
 \end{equation}
Then symmetrising $(A' C')$ and antisymmetrising $[A B ]$    in (\ref{eq: Trace-C-D-2}), we get
 \begin{equation*}
      3R_{(A'C') C [A B] }^{\phantom {(C'A') C [A B]}     C }
  +(2n-1)R_{ [A B] B'(A' C')  }^{\phantom { [A B] B'(A' C')}     B' }  =0,
 \end{equation*}by using (\ref{eq:sym}) again.
This together with (\ref{eq:()-[]}) implies that the second and the third identity in
  (\ref{eq:vanishing}).
The proposition   is proved.
\end{proof}

{\it Proof of Proposition \ref{prop:curvature}.} (1)  The first identity in (\ref{eq:curvature}) follows directly from the definition of   $\Psi$ in
(\ref{eq:curvature-0}).
\vskip 3mm
(2)   Now   antisymmetrising $[A'B']$ and $[AC]$   in (\ref{eq: Trace-B-D}), we see that
 \begin{equation*}
    3 R_{[A'B'] B [A C]   }^{\phantom{ [A'B'] B [A C] }    B }-(2n+1)R_{[ A C] C' [A'B']   }^{\phantom{[ A C] C' [A'B'] }    C' }=0
 \end{equation*}
by using (\ref{eq:sym}) again. Namely
 \begin{equation}\label{eq:Lambda}
R_{[A'B'] C[A B] }^{\phantom {[A'B'] C[A B]}      C }=(2n+1) \Lambda_{AB} \varepsilon_{A'B'}
 \end{equation} by the definition of $\Lambda$ in (\ref{eq:curvature-0}). On the other hand  symmetrising $(A' B')$ and   $(A C)$ in (\ref{eq: Trace-B-D}), we get
 \begin{equation*}
  -   R_{(A'B') B (A C)   }^{\phantom{(A'B') B (A C) }     B }+(2n-1)R_{( A C) C' (A'B')  }^{\phantom{( A C) C' (A'B')}     C' }=0.
 \end{equation*}
Then  by the definition of   $\Phi$ in (\ref{eq:curvature-0}), we get
\begin{equation}\label{eq:Phi}R_{(A'B') C(A B) }^{\phantom{(A'B') C(A B)}    C }=
  (2n-1)  \Phi_{A B A'B' } .
  \end{equation}

\vskip 3mm

(3) Antisymmetrising $[A'B']$ and symmetrising  $(A B)$ in (\ref{eq: Trace-C'-D'-2}), we get
 \begin{equation}\label{eq: A-B-C-D}
    2R_{[A'B'] AB C   }^{\phantom {[A'B'] AB C }     D }-2R_{[A'B'] C (A B )  }^{\phantom {[A'B'] C (B A) }     D }
    +R_{C B  C'[ A'B']   }^{\phantom {C B  C'[ A'B']}      C' }\delta_{A }^{\phantom {C}D }+R_{ C A  C' [A'B']   }^{\phantom {C A  C' [A'B']}     C'
    }\delta_{B }^{\phantom {C}D }=0,
 \end{equation}
 by using the  antisymmetry
 (\ref{eq:antisym}). It follows from   the last identity in   (\ref{eq:vanishing}) and the definition of $\Lambda$ in (\ref{eq:curvature-0}) that
 \begin{equation*}
   R_{C B  C'[ A'B']   }^{\phantom {C B  C'[ A'B']}      C' }\delta_{A }^{\phantom {C}D }=R_{[C B]  C'[ A'B']   }^{\phantom {(C B)  C'[ A'B']  }      C'
   }\delta_{A }^{\phantom {C}D }=3\Lambda_{C B   }\delta_{A }^{\phantom {C}D }\varepsilon_{A'B'}.
    \end{equation*}
Thus the sum of the last two terms in  (\ref{eq: A-B-C-D}) is equal to $6\Lambda_{ C(B}  \delta_{A) }^{\phantom {(C}D } $. Then by    the trivial identity $2\Lambda_{ C(B}  \delta_{A) }^{\phantom {(C}D }  =\delta_{(C }^{\phantom {(C}D }\Lambda_{A) B   }+\delta_{(C }^{\phantom {(C}D
}\Lambda_{B) A   }$ and  the definition of $\Psi$ in (\ref{eq:curvature-0}),
 we see  that (\ref{eq: A-B-C-D}) is equivalent to
 \begin{equation*}
    \Psi_{  AB  C   }^{\phantom { AB  C }     D }=\Psi_{ C  (AB) }^{\phantom {C  (AB)}      D }. \end{equation*} Then  $\Psi_{  AB  C   }^{\phantom { AB C  } D }=\Psi_{ (AB C )  }^{\phantom {(AB C) } D } $ since $\Psi_{  AB  C   }^{\phantom { AB C  } D } $ is symmetric in $A,B$ by definition. $\Psi_{  AB  C   }^{\phantom { AB C  } A }=0$ by symmetrising $(BC)$ in the definition (\ref{eq:curvature-0}) of $\Psi$ and using the first identity in (\ref{eq:vanishing}).

\vskip 3mm

(4)
   Symmetrise $(A'B')$ and  antisymmetrise  $[A B]$ in (\ref{eq: Trace-C'-D'-2}) to get
\begin{equation*}\label{eq:R('')[]-0}
  2R_{(A'B') AB C   }^{\phantom{(A'B') AB C }     D }+2R_{(A'B')  C [AB] }^{\phantom{(A'B')  C [AB] }    D }
    -R_{C B  C'(A'B')    }^{\phantom{C B  C'(A'B')}    C' }\delta_{A }^{\phantom {C}D }+ R_{ C A  C' (A'B')    }^{\phantom{C A  C' (A'B')}    C'
    }\delta_{B }^{\phantom {C}D }=0.
\end{equation*}
Apply (\ref{eq:vanishing}) and  the definition of $\Phi$ in (\ref{eq:curvature-0}) to the last two terms above to get
\begin{equation}\label{eq:R('')[]}
   R_{(A'B') AB C   }^{\phantom{ (A'B') AB C }   D }+R_{(A'B')  C [AB] }^{\phantom{(A'B')  C [AB] }     D }
 -\delta_{ [A  }^{\phantom{[A}D }\Phi_{ B]C A'B' }=0.
\end{equation}
Now antisymmetrising $[A  B C ]$ above, we find that
\begin{equation}\label{eq:primed-antisys}
   R_{(A'B') [AB C]   }^{ \phantom{(A'B') [AB C]}     D } =0,
\end{equation}
which is equivalent to
\begin{equation*}\label{eq:anti[3]}
   R_{(A'B')  AB C    }^{ \phantom{(A'B')  AB C }     D }+2R_{(A'B')C [AB]   }^{ \phantom{(A'B')C [AB] }     D }=0 ,
\end{equation*}by  $R_{(A'B')  AB C   }^{ \phantom{(A'B')C  AB  }     D }$   antisymmetric in $A,B$.
Substitute this into (\ref{eq:R('')[]}) to get
the second identity in (\ref{eq:curvature}).
\vskip 3mm
 (5) To show
the last identity in (\ref{eq:curvature}),   antisymmetrise $[A B ]$ in (\ref{eq: Trace-C-D-1})  to get
\begin{equation*}\label{eq:CD-AB}
   - R_{ C'B'  C[B A ]   }^{\phantom{ C'B'  C[B A ] }    C }\delta_{A' }^{\phantom{C'}D' }+R_{C'A' C[ A B] }^{\phantom{C'A' C[ A B]}    C }\delta_{B'
   }^{\phantom{C'}D' }
  + 2n R_{ [A B] A'  B' C'  }^{\phantom{ [A B] A'  B' C'}     D' }  -2R_{  [ A B]C'(A'B')   }^{\phantom{ [ A B]C'(A'B')}      D' }  =0
\end{equation*}by using (\ref{eq:sym}) to  the first term.
Then we have
\begin{equation}\label{eq:[AB]}(2n+1)\Lambda_{AB}\left(\varepsilon_{ C'B'} \delta_{A' }^{\phantom {C'}D' }+\varepsilon_{ C'A'} \delta_{B' }^{\phantom
{C'}D' }  \right)+
    2n R_{ [A B] A'  B' C'  }^{\phantom{[A B] A'  B' C'}     D' }  -2R_{  [ A B]C'(A'B')   }^{\phantom{ [ A B]C'(A'B')}      D' } =0,
\end{equation}
by (\ref{eq:vanishing}) and (\ref{eq:Lambda}).
 Symmetrising $(A' B'C')$ in the above identity, we get
\begin{equation}\label{eq:Symmetrising}
    (2n-2) R_{ [A B]( A'  B' C')  }^{\phantom { [A B]( A'  B' C')}   D' } =0.
\end{equation}
If $n\neq 1$,   it is equivalent to
 \begin{equation}\label{eq:Symmetrising0}
    R_{ [A B] A'  B' C'  }^{\phantom{[A B] A'  B' C'}      D' }  +2R_{  [ A B]C'(A'B')   }^{\phantom{[ A B]C'(A'B')   }       D' }=0 ,\end{equation}by
    $R_{ [A B] A'  B' C'  }^{\phantom{[A B] A'  B' C'}      D' }$    symmetric in $A'$ and $B'$.
Substitute this into (\ref{eq:[AB]}) to get the last identity in (\ref{eq:curvature}). If $n=1$, (\ref{eq:curvature-4d}) follows from the definition of ${\Psi}_{   A'  B' C'  }^{'\phantom{  A'  B' C'}      D' }$.
We have ${\Psi}_{   A'  B' C'  }^{'\phantom{  A'  B' C'}      D' }=
 {\Psi}_{  ( A'  B' C' ) }^{'\phantom{()  A'  B' C'}      D' } $  as in part  (3) by exchanging the primed and unprimed indices.
\vskip 3mm

(6) At last, we symmetrise $(A B)$ in (\ref{eq: Trace-C-D-1}) to get
\begin{equation*}\label{eq:CD-AB}
  -  R_{ C'B'  C(A B)   }^{\phantom{C'B'  C(A B) }       C }\delta_{A' }^{\phantom {C'}D' }+R_{C'A' C(A B) }^{\phantom{C'A' C(A B) }      C }\delta_{B'
  }^{\phantom {C'}D' }
  + 2n R_{ (A B)A'  B' C'  }^{\phantom{ (A B)A'  B' C'}       D' }  -R_{(A B)C' B'A'   }^{\phantom{(A B)C' B'A'  }       D' } +R_{  (A B)C'A'B'
  }^{\phantom{(A B)C'A'B'}       D' }   =0,
\end{equation*}by using (\ref{eq:sym}) again. By (\ref{eq:vanishing}), (\ref{eq:Phi})  and the second identity in (\ref{eq:curvature}), we find that
\begin{equation}\label{eq:(AB)} (2n-1)\left(-\Phi_{AB  C'B'} \delta_{A' }^{\phantom {C'}D' }+\Phi_{AB C'A'} \delta_{B' }^{\phantom {C'}D' } \right
)+
    2nR_{ (A B) A'  B' C'  }^{\phantom{(A B) A'  B' C'}       D' }  +2R_{ (A B)C'[A'B']   }^{\phantom{(A B)C'[A'B']}       D' }=0.
\end{equation}
Antisymmetrise $[A' B'C']$ above to get
\begin{equation*}
    (2n+2) R_{(A B) [A'  B' C']  }^{\phantom{(A B) [A'  B' C'] }       D' } =0.
\end{equation*}
Namely
 \begin{equation*}
    R_{ (A B) A'  B' C'  }^{\phantom{ (A B) A'  B' C' }       D' }  +2R_{  ( A B)C'[A'B']   }^{\phantom{ ( A B)C'[A'B']  }      D' }=0 .\end{equation*}
Substitute this into (\ref{eq:(AB)}) to get
\begin{equation*}
       R_{(A B) A'  B' C'  }^{\phantom{(A B) A'  B' C' }       D' }  =\Phi_{AB  C'B'} \delta_{A' }^{\phantom {C'}D' }-\Phi_{AB C'A'} \delta_{B'
    }^{\phantom {C'}D' }  .
\end{equation*}The third identity in (\ref{eq:curvature}) is obtained.
The proposition is proved. \hskip 55mm$\Box$

\vskip 5mm
 {\it Proof of Proposition \ref{prop:Einstein}}.  (1) By the second Bianchi identity
    \begin{equation}\label{eq: Bianchi2}
       \nabla_{[a}R_{bc] d}^{\phantom{bc] d }e}=0.
    \end{equation}  for  $a=AA'$, $b=BA'$, $c=CA'$ and $d=DD'$, $e=DA'$, and taking summation over repeated indices, we find that
    \begin{equation*}
       \nabla_{A'[A}R_{|A'A'|BC]  D}^{\phantom{|A'A'|BC]D}D}\delta_{D'}^{\phantom {C'}A'}+ \nabla_{A'[A}R_{BC]A'A' D'}^{\phantom{BC]A'A'D'}A'}\delta_{D
       }^{\phantom {C }D}=0.
    \end{equation*}
The first term vanishes as a trace by (\ref{eq:vanishing-trace}), while the second term is $2\nabla_{A'[A}\Lambda_{  BC] }  \varepsilon_{A'D'}\delta_{D
       }^{\phantom {C }D}$ by the last identity of (\ref{eq:curvature}) and (\ref{eq:curvature-4d}) for the $4$-dimensional case.

(2)  Note that $\nabla_{BB'}\epsilon_{CD} = 0$ implies $R_{(A'B')ABC}^{\phantom{(A'B')ABC}E}\epsilon_{ED}
+R_{(A'B')ABD}^{\phantom{(A'B')ABD}E}\epsilon_{CE}=0$, or equivalently
 \begin{equation*}
 - \epsilon_{D[A} \Phi_{ B] C A'    B' }+\epsilon_{C[A} \Phi_{ B] D A'    B' }=0.
 \end{equation*} by Proposition \ref{prop:curvature}.
If we   antisymmetrise   $[DAB]$,  we get   $\epsilon_{[D A} \Phi_{ B] C A'    B' }=0$, which can hold only if $\Phi_{ AB A'    B' }=0$.

The Ricci curvature of the connection on  tangent  bundle  is given
by
   \begin{equation*}\label{eq:Ricci}  \begin{split}
 R_{a c}& = R_{ab cd   }g^{bd }= R_{AA'BB'CC'DD' }\epsilon^{BD}\varepsilon^{B'D'}= R_{AA'BB'CC' }^{\phantom {AA'BB'CC' }BB'} \\&= R_{ A'B'A B C }^{\phantom { A'B'A B C }    B  }\delta_{C'}^{\phantom { A'}B'}+ R_{ ABA' B'C ' }^{\phantom { A'C'A' C  B }   B'
 }\delta_{C }^{\phantom { A }B } =2( n+2)\Lambda_{   A C   }\varepsilon_{A'C'}  ,
 \end{split} \end{equation*}by the curvature decomposition in Proposition \ref{prop:curvature},  if we choose the local orthonormal  quaternionic  frame  $Z_{AA'}$   so that the metric is $g_{ab}$ in ¡¡(\ref{eq:gab}). Then
  \begin{equation*}\label{eq:Ricci}  \begin{split}
 s_g& = R_{ac   }g^{ac }=2( n+2)\Lambda_{   A C   }\varepsilon_{A'C'} \epsilon^{AC}\varepsilon^{A'C'} =8n(n+2)\Lambda.
 \end{split} \end{equation*}

 It remains to show $\Lambda_{AB}=\Lambda \epsilon_{AB}$.
Recall that
   for any   given section
$\phi_{AB}$ in $\Gamma(\Lambda^2 E^*)$  satisfying $\nabla_a \phi_{AB} = 0$, we must have $ \phi_{AB} = \alpha\Lambda_{  AB  }$ for some constant
$\alpha$ when the tensor $\Lambda_{AB}\neq 0$ (this is   lemma 7.7 of \cite{BaiE}). This is because
$\nabla_a \phi_{CD} = 0$ implies $R_{[A'B']ABC}^{\phantom{(A'B')ABC}E}\phi_{ED}
+R_{[A'B']ABD}^{\phantom{(A'B')ABD}E}\phi_{CE}=0$, i.e.
\begin{equation}\label{eq:e-decomp}
     \Psi_{ AB C }^{\phantom { AB E  }    E }\phi_{ED}+ \Psi_{ AB D }^{\phantom { ABC  }    E }\phi_{C E}-2 \Lambda_{
    C(A
    }  \phi_{B)D}+2 \Lambda_{
  D (A
    }  \phi_{B) C}=0
\end{equation}
by the curvatures decomposition in Proposition \ref{prop:curvature}. Symmetrising   $(ABC)$ followed by antisymmetrising   $[CD]$ kills the $\Lambda$ terms,  but simply
multiplies the $\Psi$ terms by $\frac 23$. Thus $-2 \Lambda_{
    C(A
    }  \phi_{B)D}+2 \Lambda_{
  D (A
    }  \phi_{B) C}=0$, which is essentially just the wedge product of two exterior forms. It is easy to see that it vanishes if and  only if $\phi_{AB}$ is a multiple of $\Lambda_{AB}$.
\hskip  105mm$\Box$

\vskip 5mm
 {\it Proof of  (\ref{eq:curvature-conformal})}.    Locally we choose a coordinate chart  $U_\alpha$ with trivialization $E^*|_{U_\alpha }
= U_\alpha\times \mathbb{C}^{2n}$, $H^*|_{U_\alpha }= U \times \mathbb{C}^{2 } $, and a two-component  local quaternionic  frame $\{Z_{AA'}\}$ such that $  \varepsilon$ and $\epsilon$ are  standard.
In particular, $\varepsilon_{0'1'}= \varepsilon^{1'0'}=1$.
By using (\ref{eq:connection-change}) repeatedly, we get
\begin{equation*}\begin{split}
   \widetilde{\nabla}_{AA'} \widetilde{\nabla}_{BB'}  f_{C' }= &\widetilde{{\nabla}}_{AA'} ({\nabla}_{BB'}  f_{C' }-   \Upsilon_{B C ' }   f_{B '})\\= &
  \nabla_{AA'}  {\nabla}_{BB'}  f_{C' }- \nabla_{AA'}   \Upsilon_{ B C' } \cdot  f_{B'} - \Upsilon_{B C' }\cdot\nabla_{AA'}   f_{B'}   \\&
 - \Upsilon_{ BA' }  {\nabla}_{ AB'}   f_{C' } - \Upsilon_{ AB' }  {\nabla}_{ B A'}   f_{C' }- \Upsilon_{A C' }{\nabla}_{ BB' }   f_{A'}\\&+ \Upsilon_{B
 A' }
  \Upsilon_{ A C' }   f_{B'} + \Upsilon_{ A C' }  \Upsilon_{B A' }   f_{B'} +\Upsilon_{ A B' }  \Upsilon_{B C' }   f_{A'}.
\end{split}\end{equation*}
Note that when $A$ and $B$, $A'$ and $B'$ are exchanged, the sum of the $3$rd and $6$th terms, the sum of $4$th and $5$th terms,  and
the sum
of $7$th and $9$th terms in the right hand side above are all invariant. The above identity minus the one with $A$ and $B$, $A'$ and $B'$   exchanged becomes
 \begin{equation}\label{eq:difference-curvature}\begin{split}-\widetilde{R}_{A B  A'B'C'  }^{\phantom{AA'BB'D' }D'}=-R_{A B  A'B'C'  }^{\phantom{AA'BB'D' }D'}
&-\nabla_{AA'}   \Upsilon_{ B  C'
    }\delta_{B'}^{\phantom{A'}D'}+\nabla_{BB'}   \Upsilon_{ A C' }\delta_{A'}^{\phantom{A}D'}\\&\,+\Upsilon_{ A C' }  \Upsilon_{B A' }
    \delta_{B'}^{\phantom{A'}D'}
    -\Upsilon_{ B C' }  \Upsilon_{A B' } \delta_{A'}^{\phantom{A'}D'}.
 \end{split}\end{equation}
Note that $\Lambda_{   A B    }  =\frac 13 R_{[ AB] C' [0'1']   }^{\phantom{[ A C] C' [0'1'] }  C' }, \Omega \widetilde{\Lambda}_{   A B    } =\frac 13\widetilde{R}_{[ AB] C' [0'1']   }^{\phantom{[ A C] C' [0'1'] }  C' }$ by definition of $\Lambda$ in (\ref{eq:curvature-0}) and \begin{equation}\label{eq:up-down}
    f_{\ldots B'\ldots\phantom{ B'} \ldots }^{\phantom{\ldots B'\ldots} B'  }=  f_{\ldots 0' \ldots 1' \ldots}  -  f_{\ldots 1' \ldots 0' \ldots} =-
    f_{\ldots\phantom{ B'}\ldots B'\ldots }^{\phantom{\ldots } B'  }.
\end{equation}
Antisymmetrise $[AB]$  and $[B'C']=[0'1']$ and
take trace over $A'$-$D'$ in (\ref{eq:difference-curvature}) to
  get the transformation formula  for $\Lambda$. We have used, for example, $\nabla_{0'[A}   \Upsilon_{ B]  1'
    }-\nabla_{1'[A}   \Upsilon_{ B]  0'
    }=     \nabla_{A'[A}   \Upsilon_{ B] }^{A'} $.

By definition of $\Phi$ in  (\ref{eq:curvature-0}),  symmetrise $(B'C')$
and $(AB)$  and
take trace over $A'$-$D'$  in (\ref{eq:difference-curvature})  to
  get the transformation formula for $\Phi$ in (\ref{eq:curvature-conformal}).

Similarly as (\ref{eq:difference-curvature}) by exchanging primed and unprimed indices,
we get
\begin{equation*}
  -  \widetilde{R}_{A'B' ABC  }^{\phantom{AA'BB'D }D} =-R_{A'B' ABC  }^{\phantom{AA'BB'D }D}-\nabla_{AA'}   \Upsilon_{ C B'
    }\delta_{B}^{\phantom{A}D}+\nabla_{BB'}   \Upsilon_{ C A' }\delta_{A}^{\phantom{A}D}+\Upsilon_{ C A' }  \Upsilon_{ AB' } \delta_{B}^{\phantom{A}D}
    -\Upsilon_{ C B' }  \Upsilon_{ BA' } \delta_{A}^{\phantom{A}D},
\end{equation*} from which we get the invariance of  $\Psi$ in (\ref{eq:curvature-conformal}) by antisymmetrising $[A 'B']=[0'1']$ and using  the transformation formula  for $\Lambda$ and the definition of $\Psi$ in  (\ref{eq:curvature-0}). \hskip  60mm$\Box$

\vskip 5mm
 {\it Proof of Proposition \ref{prop:D-conformal}}. The new connection  $
   \widetilde{\nabla}$ in (\ref{eq:connection-change})  induces the covariant derivatives for $E$ and $H$ by duality:
   \begin{equation}\label{eq:nabla-raise}
   \widetilde{\nabla}_{ A  A'}  f^{B'} =
    {\nabla}_{ A  A'}   f^{B'} +  \Theta_{A A'D'}^{\phantom{A A'D'}B'}  f^{D'}
                = {\nabla}_{ A   A'}  f^{B'}+  \delta_{ A' }^{\phantom{ A'
  }B' } \Upsilon_{ A D' }  f^{D'} .\end{equation}If we raise unprimed indices by $\widetilde{\varepsilon}^{A'B'}=\Omega^{-1}{\varepsilon}^{A'B'}$ (note that $\widetilde{\nabla}\widetilde{\varepsilon}^{A'B'}=0$), we get
   \begin{equation}\label{eq:nabla-raise'} \begin{split}
  \widetilde{\nabla}_{A  }^{A '}  f^{B'}:&=\widetilde{\nabla}_{ A  C '}  f^{B'}\widetilde{\varepsilon}^{C'A'}
   =\Omega^{-1}\left({\nabla}_{A  }^{A '}  f^{B'}-{\varepsilon}^{ A'B'} \Upsilon_{ AD ' }   f^{D'}\right),
\\
\widetilde{\nabla}_{A  }^{A '}  f_{B }:&=\widetilde{\nabla}_{ A  C '}  f_{B }\widetilde{\varepsilon}^{C'A'}
   =\Omega^{-1}\left({\nabla}_{A  }^{A '}  f_{B }- \Upsilon_{ B  }^{ A'}   f_{A }\right),\\\widetilde{\nabla}_{A  }^{A '}  f_{B '
   }:&=\widetilde{\nabla}_{ A  C '}  f_{B' }\widetilde{\varepsilon}^{C'A'}
   =\Omega^{-1}\left({\nabla}_{A  }^{A '}  f_{B' }- \Upsilon_{ AB'  }    f^{A '}\right).
   \end{split}
\end{equation}  Applying these formulas of covariant derivatives repeatedly,  we get
\begin{equation}\label{eq:tilde-D} \begin{split}
\left ( \widetilde{\mathscr D}_{q}^{p}(\Omega^{-q}  f)\right)_{A_1\ldots A_{q+1}}^{ A_1'\ldots A_{p+1 }'}=&- q\Omega^{-q-1}\widetilde{\nabla}_{[A_1
}^{(A_1'}
\Omega \cdot  f_{A_2\ldots A_{q+1}]}^{ A_2'\ldots A_{p+1 }')} + \Omega^{-q}\widetilde{\nabla}_{[A_1 }^{(A_1'}   f_{A_2\ldots   A_{q+1}]}^{ A_2'\ldots
A_{p+1}')}
  \\ =&\Omega^{-q-1 }\left\{-q\Upsilon_{[A_1 }^{(A_1'}  \cdot  f_{A_2\ldots A_{q+1}]}^{ A_2'\ldots A_{p+1 }')}+ {\nabla}_{[A_1 }^{(A_1'}   f_{A_2\ldots
  A_{q+1}]}^{ A_2'\ldots A_{p+1}')}\right.\\& \left. -  \varepsilon^{(A_1'A_j'} \Upsilon_{ A D' }   f_{A_1\ldots A_{q+1}]}^{ A_2'\ldots |D'|\ldots A_{p+1 }')}
  -   \Upsilon^{(A_1'}_{[ A_j }   f_{A_2\ldots A_1 \ldots A_{q+1}]}^{ A_2' \ldots A_{p+1 }')}\right\} . \end{split}\end{equation} Here $\widetilde{\nabla}_{ A_1
}^{ A_1'}
\Omega =\widetilde{\nabla}_{ A_1B'
}
\Omega \cdot\widetilde{\varepsilon}^{B'A_1'}=Z_{ A_1B'
}
\Omega \cdot\widetilde{\varepsilon}^{B'A_1'}=\Omega^{-1} Z_{ A_1
}^{ A_1'}\Omega=\Upsilon_{ A_1
}^{ A_1'}$, and the third term in the bracket  above vanishes automatically. While for the fourth term in the bracket, we have
\begin{equation}\label{eq:anti-nabla}-
  \Upsilon^{(A_1'}_{[ A_j }   f_{A_2\ldots A_1 \ldots A_{q+1}]}^{ A_2' \ldots A_{p+1 }')}= \Upsilon^{(A_1'}_{[ A_1 }   f_{A_2\ldots A_j
 \ldots A_{q+1}]}^{ A_2' \ldots A_{p+1 }')},\qquad j=2,\ldots,  q+1,
\end{equation}  and so the first and fourth terms in (\ref{eq:tilde-D}) cancel. The first identity in (\ref{eq:D-conformal}) follows.
Similarly, we have
\begin{equation*} \begin{split}
\left ( \widetilde{\mathscr D}_{q,p} (\Omega^{-q-1}  f)\right)_{ A_1\ldots A_{q+1 } A_2'\ldots A_{p  }'}=& -(q+1)\Omega^{-q-2} \Upsilon_{[A_1 }^{ A '}
f_{
A_2\ldots A_{q+1 } ] A 'A_2'\ldots A_{p}'} +\Omega^{-q-1}\widetilde{\nabla}_{[A_1 }^{ A '}    f_{ A_2\ldots A_{q +1} ] A ' A_2'\ldots }
\\
=&\Omega^{-q-2}\left\{-(q+1)\Upsilon_{[A_1 }^{ A '}    f_{ A_2\ldots A_{q+1 } ] A 'A_2'\ldots A_{p }'} +  {\nabla}_{[A_1 }^{ A '}    f_{A_2\ldots
A_{q+1
} ] A 'A_2'\ldots A_{p }'}\right. \\&\qquad\quad-  \Upsilon_{[A_j }^{ A'}    f_{ A_2\ldots A_1\ldots  A_{q +1} ] A'A_2'\ldots A_{p }'}
-  \Upsilon_{A '[A_1  }     f_{ A_2\ldots    A_{q+1 } ]  \phantom{A '}A_2'\ldots  \ldots A_{p }'}^{\phantom{A_1\ldots   A_{q+1} ]
 }A '}\\&\qquad\quad\left.
-  \Upsilon_{A_j'[A_1  }     f_{ A_2\ldots    A_{q+1 } ] A 'A_2'\ldots\phantom{A '}  \ldots A_{p }'}^{\phantom{A_1\ldots   A_{q+1} ]
A_1'A_2'\ldots}A '}\right\}.\end{split}\end{equation*}
  The last term in the bracket above vanishes by (\ref{eq:up-down}) and $f$  symmetric   in primed indices. We also use (\ref{eq:up-down}) to raise and lower $A'$ in the
  $4$th term. The result follows from an identity similar to (\ref{eq:anti-nabla}) that the sum of the third and $4$th terms in the bracket  cancels the first term.
  We get the second identity in (\ref{eq:D-conformal}). At last
\begin{equation*}\begin{split}  &
      \widetilde{\nabla}_{A'[A_1} \widetilde{\nabla}_{A_2}^{A'}(\Omega^{-q-1}  f)_{A_3\ldots A_{q+2}]}
    \\
  & =\widetilde{\nabla}_{A'[A_1}\left\{\Omega^{-q-2}  \left(- (q+1)  \Upsilon_{[A_2 }^{A'}f_{A_3\ldots A_{q+2}]]} +   {\nabla}_{[A_2}^{A'}f_{A_3\ldots A_{q+2}]]} -  \Upsilon_{[A_m}^{A'}
   f_{A_3\ldots A_2\ldots A_{q+2}]]} \right)\right\}\\
  & =\widetilde{\nabla}_{A'[A_1}\left\{\Omega^{-q-2}\left(-      \Upsilon_{A_2 }^{A'}    +   {\nabla}_{A_2}^{A'}\right)f_{A_3\ldots
   A_{q+2}]}\right\}\\
   &= \Omega^{-q-2}\left(   (q+2)\Upsilon_{A'[A_1}    \Upsilon_{A_2  }^{A'}   - (q+2)\Upsilon_{A'[A_1}  {\nabla}_{A_2}^{A'} - {\nabla}_{ A'[A_1}  \Upsilon_{ A_2}^{A'} -\Upsilon_{[A_2}^{A'}
   {\nabla}_{|A'|A_1} + {\nabla}_{ A' [A_1}\nabla_{ A_2}^{A'}\right)f_{   A_3 \ldots A_{q+2}]}\\&
  + \Omega^{-q-2}\left (\Upsilon_{A'[A_j}    \Upsilon_{A_2  }^{A'} -\Upsilon_{A'[A_j} {\nabla}_{A_2}^{A'}  \right)f_{ \ldots A_1 \ldots A_{q+2}]} + \delta_{A'}^{\phantom{A '}A'} \Omega^{-q-2}\left (-\Upsilon_{D'[A_1}    \Upsilon_{A_2  }^{D'} +\Upsilon_{D'[A_1} {\nabla}_{A_2}^{D'}  \right)f_{   A_3 \ldots A_{q+2}]}\end{split}\end{equation*}
 where $m=3,\ldots,  q+2$,  $j=2,\ldots,  q+2$.
 Here we use the identity as (\ref{eq:anti-nabla}) in the second    identity. Note that terms of the form $ \Upsilon_{A'[A_1} {\nabla}_{A_2}^{A'}   f_{     \ldots  ]} $  cancel each other by $ -\Upsilon_{[A_2}^{A'}
   {\nabla}_{|A'|A_1]}= -\Upsilon_{A'[A_1}
   {\nabla}^{ A' }_{A_2]}$ by using (\ref{eq:up-down}) to raise and lower primed indices,  and only one   of terms of the form $-\Upsilon_{A'[A_1}    \Upsilon_{A_2  }^{A'}   f_{     \ldots  ]} $ remains. So
  \begin{equation*}\begin{split} \widetilde{\nabla}_{A'[A_1} \widetilde{\nabla}_{A_2}^{A'}(\Omega^{-q-1}  f)_{A_3\ldots A_{q+2}]}  =&\Omega^{-q-2}\left(  -  \Upsilon_{A'[A_1}    \Upsilon_{A_2  }^{A'}    -{\nabla}_{ A'[A_1}  \Upsilon_{ A_2}^{A'} + {\nabla}_{ A' [A_1}\nabla_{
 A_2}^{A'}\right)f_{   A_3 \ldots A_{q+2}]}\\=&
 \Omega^{-q-2} \left(  -2\Omega \widetilde{\Lambda}_{ [A_1 A_2  }    +2\Lambda_{ [A_1 A_2  }   + {\nabla}_{ A' [A_1}\nabla_{
 A_2}^{A'}\right)f_{   A_3 \ldots A_{q+2}]}
,
\end{split}\end{equation*} by conformal transformation formula for $\Lambda$ in (\ref{eq:curvature-conformal}). We get the last identity in (\ref{eq:D-conformal}).\hskip  23mm$\Box$

\vskip 5mm

We  need the following  theorem to prove Theorem \ref{thm:BVP}.
 \begin{thm} \label{thm:Ell-reg} {\rm (Theorem 4.12 in \cite{We})}  Let $L$ be a self-adjoint and elliptic differential operator of order $m$ on a vector
 bundle $F$ over a compact manifold. Then there exist linear mappings $
    H_L, G_L: C^\infty (X, F)\longrightarrow C^\infty (X, F)
$
so that
\item[(1)] $ H_L( C^\infty (X, F))=\mathfrak H_L(E)$ and $\dim \mathfrak H_L(E)<\infty$.

\item[(2)] $L\circ G_L+H_L=G_L\circ  L +H_L=$ the identity on $C^\infty (X, F)$.

\item[(3)]  $H_L$ and $G_L$ extend to  bounded linear operators on $L^2(X,F)$.

\item[(4)] $C^\infty (X, F)= \mathfrak H_L(F)\oplus L\circ G_L(C^\infty (X, F)) = \mathfrak H_L(F)\oplus G_L\circ  L(C^\infty (X, F))$ is an orthogonal
    decomposition with respect to the inner product in $L^2(X,F)$. In particular, we have $\mathfrak H_L(F)=\ker L$.
\end{thm}
 {\it Proof of  Theorem \ref{thm:BVP}}.  We only give the proof of case $j=k$. Other cases are
similar (see \cite{Wa08} \cite{Wa10} for the  flat case).

We will omit the superscript $(k)$ for simplicity. Applying Theorem \ref{thm:Ell-reg} to the elliptic differential operators
  $L=\Box_k $  of
 $4$-th order, we see that that there exists a partial inverse operator
 $\mathbf{G}_k:\Gamma\left(
 \mathscr V_k \right ) \longrightarrow \Gamma\left(
 \mathscr V_k \right )$ such that
 \begin{equation}\label{eq:Green}
   \mathbf{G}_k\Box_k =\Box_k \mathbf{G}_k={\rm id}  \qquad {\rm on}\quad \left(\ker\Box_k \right)^{\perp}.
 \end{equation}

To see $\ker\Box_k=\mathscr H^k_{ (k) }(M)$, note that $(\Box_kf,f)=0$ implies that
 $(D_kf,D_kf)=0$  and $( D_{k-1} D_{k-1}^*  f,$ $D_{k-1} D_{k-1}^*f)=0$, i.e.    $ D_kf =0$ and $D_{k-1}D_{k-1}^*f =0$. Then we have $D_{k-1}^*f =0$ by
\begin{equation*}
   ( D_{k-1}D_{k-1}^*  f,f)= (D_{k-1}^*  f,D_{k-1}^*f)=0.
\end{equation*}

Now suppose that $f\in
 \Gamma\left(
 \mathscr V_k \right )$ such that  $f\perp  ({\rm Image}\, D_{k }^{ *}\oplus  \mathscr H^k_{ (k) }(M))$. Then obviously $f\in(\ker\Box_k  )^{\perp}$ since $\ker\Box_k  = \mathscr H^k_{ (k) }(M)$,  and $D_kf=0$ by $(D_kf,  h)=(f,D_k^*
 h)=0$ for any $  h\in \Gamma\left(
 \mathscr V_{k+1} \right )$.  The decomposition (\ref{eq:Hodge-decomposition}) holds if we can  show $f\in  {\rm Image}\, D_{k-1} $. If  set
\begin{equation*}
 u:= D_{k-1}^*D_{k-1}D_{k-1}^*\mathbf{G}_k f\in
\Gamma \left(
 \mathscr V_{k-1} \right ),\end{equation*}
we have
\begin{equation}D_{k-1}u= (D_{k-1}D_{k-1}^*)^2 \mathbf{G}_k f=
\left( (D_{k-1}D_{k-1}^*)^2 +D_k^*D_k\right)\mathbf{G}_k
f=\square_k\mathbf{G}_k f=f,\label{eq:D0u=f}
\end{equation}
by (\ref{eq:Green}) in the last identity and
\begin{equation*}  D_k ^* D_k \mathbf{G}_k f =0.
\label{eq:DDG}
\end{equation*}
This is because
\begin{equation*}
   (D_k ^* D_k \mathbf{G}_k f , D_k ^* D_k \mathbf{G}_k f )= ( D_k \mathbf{G}_k f ,D_k  D_k ^* D_k \mathbf{G}_k f )=( D_k \mathbf{G}_k f ,
   \square_{k+1} D_k \mathbf{G}_k f )=0
\end{equation*}
by\begin{equation}\begin{split} D_k(D_k^*D_k+( D_{{k-1}}D_{{k-1}}^*)^2 )&
=D_kD_k^*D_k = ( (D_{k+1}^*D_{k+1})^2+ D_{k }D_{k }^* )D_{k
}.\end{split}\label{eq:commutation}
\end{equation}and so
\begin{equation*} \square_{k+1} D_{k
}\mathbf{G}_kf= D_{k
} \square_{k }\mathbf{G}_kf= D_k   f =0.\label{eq:commutation-DG}
\end{equation*}
  The
commutating relation  (\ref{eq:commutation}) follows from   $D_k
D_{{k-1}} =  D_{k+1} D_{k }=0$. The result follows. \hskip  16mm$\Box$

\end{appendix}


\begin{thebibliography}{20}


\bibitem{adams2}  {\sc Adams, W., Loustaunau, P., Palamodov, V. and  Struppa, D.},
 {Hartogs' phenomenon for polyregular functions and projective dimension of releted modules over a polynomial ring}, {\it Ann. Inst.
 Fourier }{\bf 47} (1997) 623-640.





\bibitem{alesker2}{\sc  Alesker, S.}, {Pluripotential theory on quaternionic manifolds,} {\it J. Geom. Phys.} \textbf{62} (2012), no. 5,
    1189-1206.

\bibitem{alesker7}{\sc   Alesker, S. and Verbitsky, M. }, {Quaternionic Monge-Amp\`{e}re equation and Calabi problem for HKT-manifolds},
    {\it Israel J.
    Math.}
    \textbf{176} (2010), 109-138.



    \bibitem{alesker6} {\sc Alesker, S. and   Verbitsky, M. }, {Plurisubharmonic functions on hypercomplex manifolds and HKT-geometry},
        {\it J. Geom.
        Anal.}
        {\bf16} (2006)  375-399.


\bibitem  {BaiE}{\sc Bailey, T.  and   Eastwood, M.}, Complex paraconformal manifolds-their differential
geometry and twistor theory,
{\sl Forum Math.\/} {\bf 3} (1991), 61-103.

\bibitem  {Ba} {\sc  Baston,  R.}, Quaternionic complexes,
 {\sl  J. Geom. Phys.\/} {\bf 8} (1992)
29-52.

\bibitem  {BE}{\sc  Baston,  R. and Eastwood, M.},
{\it The Penrose transform. Its interaction with representation
theory}, Oxford Mathematical Monographs,   The Clarendon Press,
Oxford University Press, New York, 1989.

\bibitem  {BdHCGV}{\sc
Bergshoeff, E.,  de Wit, T.,  Halbersma, R.,  Cucu, S.,  Gheerardyn, J., Van Proeyen, A. and Vandoren, S.},
Superconformal $N=2$, $D=5$  matter with and without actions,  {\sl
J. High Energy Phys.\/} {\bf  2002},  no. 10, 045, 65 pp.



\bibitem{bS}{\sc Bure\v s, J.  and V. Sou\v cek,
V.},
  Complexes of invariant differential operators in
several quaternionic variables, {\it  Complex Var. Elliptic Equ.\/} {\bf
51} (2006), no. 5-6, 463-487.

\bibitem{bures}{\sc Bure\v s, J.,  Damiano, A. and
Sabadini, I.},  Explicit resolutions for several Fueter operators, {\it J. Geom.
Phys.\/} {\bf 57},   2007,   765-775.


\bibitem{CS09}{\sc\v
Cap, A.  and Slovak, J.},  {\it  Parabolic geometries  I. Background and general theory},  Mathematical Surveys and Monographs  {\bf
154}, American Mathematical Society, Providence, RI, 2009.

\bibitem{CS}{\sc\v
Cap, A., Slovak, J. and Sou\v cek, V.}, Invariant operators on manifolds with almost Hermitian symmetric structures. III. Standard
operators, {\it   Differential Geom. Appl.\/} {\bf 12} (2000), no. 1, 51-84.


\bibitem{CSS01}{\sc\v
Cap, A., Slovak, J. and Sou\v cek, V.},  Bernstein-Gelfand-Gelfand sequences, {\it   Ann. of Math. (2)\/} {\bf 154} (2001), 97-113.

\bibitem{CS2}{\sc
Cap, A.  and Sou\v cek, V.}, Subcomplexes in curved BGG-sequences, {\it  Math. Ann.\/} {\bf 354} (2012), no. 1, 111-136.

\bibitem{CMW}{\sc
Chang, D.-C.,
  Markina, I. and
  Wang, W.},
On the Hodge-type decomposition and cohomology groups of $k$-Cauchy-Fueter complexes over domains in the quaternionic space,   {\sl J.
Geom.
Phys.} {\bf 107} (2016), 15-34.

\bibitem  {CLW}{\sc Cherney, D., Latini, E. and Waldron, A.},
Quaternionic K\"ahler detour complexes and $N=2$  supersymmetric black holes,
{\sl
Comm. Math. Phys.\/} {\bf    302 } (2011),  no. 3, 843-873.


\bibitem  {CSS} {\sc Colombo, F., Sou\v cek, V. and
Struppa, D.},  Invariant resolutions for several Fueter operators,
{\sl   J. Geom. Phys.\/} {\bf  56} (2006), no. 7, 1175-1191.

\bibitem  {CSSS} {\sc Colombo, F., Sabadini, I.,
Sommen, F. and Struppa, D.},
{\it Analysis of Dirac systems and computational algebra},
Progress in Mathematical Physics {\bf  39},  Boston,
Birkh\"auser, 2004.






 \bibitem{EPW}{\sc   Eastwood, M., Penrose, R. and Wells, R.},
 Cohomology and massless fields,
{\sl
  Comm. Math. Phys..\/} {\bf 78} (1980), no. 3, 305-351.




 \bibitem{HKLR}{\sc  Hitchin, N.J., Karlhede, A., Lindstr\"om, U. and Ro$\rm \check{c}$ek, M.},
HyperK\"ahler metrics and supersymmetry,
{\sl Commun. Math. Phys.\/} {\bf 108}  (1987), 535-589.


 \bibitem{Hom}{\sc  Homma, Y.}, Estimating the eigenvalues on Quaternionic Kahler Manifolds, {\sl
International Jouranal of Mathematics\/} {\bf   17(6)} (2006), 665-691.


  \bibitem{Ho}{\sc  Horan, R.},
A rigidity theorem for quaternionic K\"ahler manifolds,
{\sl
Diff. Geom. and  Appl.\/} {\bf 6} (1996) 189-196.

 \bibitem{Ho2}{\sc  Horan, R.}, Cohomology of a
   quaternionic complex, in {\it Further advances in twistor theory III. curved twistor spaces}, edit by L. Mason et al., Chapman and
   Hall/CRC, London, New York, (2000), 66-71.

  \bibitem{IM}{\sc  Ivanov, S. and Minchev, I.},  Quaternionic K\"ahler and hyperK\"ahler manifolds with torsion and twistor spaces,
      {\it J. Reine Angew. Math.} {\bf  567}  (2004), 215-233.

\bibitem{KW}  {\sc Kang,  Q.  and Wang, W.},
On Penrose integral formula and series expansion of  $k$-regular
functions on the quaternionic space $ \mathbb{H}^n $,     {\it  J.
Geom. Phys.} {\bf 64} (2013),  192-208.

\bibitem{NN}  {\sc
Nagatomo, Y. and  Nitta, T.},
Vanishing theorems for quaternionic complexes,     {\it
Bull. London Math. Soc.} {\bf  29}  (1997),  no. 3, 359-366.

  \bibitem{NPV}{\sc  Neitzke, A., Pioline, B. and  Vandoren, S.},  Twistors and black holes,
{\sl JHEP\/} {\bf 0704}, 038 (2007).


 \bibitem{OP}{\sc
Ornea, L. and Piccini, P.},  Locally conformally K\"ahler structures in quaternionic geometry,     {\it Trans. Am. Math. Soc.} {\bf
349} (1997) 641-655.

\bibitem{PS}{\sc   Pand\v zi\' c, P.  and V. Sou\v cek,
V.},
BGG complexes in singular infinitesimal character for type $A$,  {\it
J. Math. Phys.} {\bf  58}  (2017), 111512.






  \bibitem{PR1}{\sc Penrose, R. and Rindler, W.}, {\it Spinors and Space-Time}, Vol. {\bf 1}, {\it Two-spinor calculus and relativistic
      fields}, Cambridge
      Monographs on Mathematical Physics,  Cambridge University Press, Cambridge,  1984.

    \bibitem{PR2}{\sc Penrose, R. and Rindler, W.}, {\it Spinors and Space-Time}, Vol. {\bf 2}, {\it Spinor and twistor methods in
        space-time geometry},
        Cambridge Monographs on Mathematical Physics, Cambridge University Press, Cambridge, 1986.



   \bibitem{Sa}{\sc  Salamon, S.}, Quaternionic K\"ahler manifolds,   {\it Invent. Math.}  {\bf 67}, (1982), no. 1,
143-171.

 \bibitem{Sa86}{\sc  Salamon, S.},
Differential geometry of quaternionic manifolds,
{\sl
Ann. Sci. \'Ecole Norm. Sup. (4)\/} {\bf   19 } (1986),  no. 1, 31-55.



 \bibitem{SW}{\sc
Semmelmann, U. and Weingart, G.},
Vanishing theorems for quaternionic K\"ahler manifolds,
{\sl
J. Reine Angew. Math.\/} {\bf  544} (2002), 111-132.

 \bibitem{Se}{\sc
Semmelmann, U. and  Weingart, G.},  The Weitzenb\"ock machine,
{\sl Compos. Math.\/} {\bf  146}  (2010),  no. 2, 507-540.


\bibitem{SS}{\sc
 Slovak, J. and Sou\v cek, V.}, Invariant operators of the first order on manifolds with a given parabolic structure,  in  {\it Global
 analysis and harmonic analysis} (Marseille-Luminy, 1999), 251-276, Semin. Congr. {\bf 4}, Soc. Math. France, Paris, 2000.


 \bibitem{Sp}{\sc
Spacil, O.},
Indices of quaternionic complexes,
{\sl
Diff. Geom. and   Appl.\/} {\bf 28  } (2010), 395-405.


  \bibitem{Sw}{\sc
Swann, A.},
Hyper-K\"ahler and quaternionic K\"ahler geometry,
{\sl
Math. Ann.\/} {\bf   289}  (1991),  no. 3, 421-450.


 \bibitem{Ver}{\sc
Verbitsky, M.},  Quaternionic Dolbeault complex and vanishing theorems on hyperK\"ahler manifolds,  {\it  Compos. Math.}  {\bf   143}
(2007),  no. 6, 1576-1592.


\bibitem{Wan}{\sc   Wan, D.  },
The continuity and range of the quaternionic
Monge-Amp\`ere operator on quaternionic space,   {\it Math. Zeit.} {\bf 285} (2017), 461-478.

\bibitem{WaK}{\sc   Wan, D.   and Kang,  Q.},  Potential theory for quaternionic plurisubharmonic functions,  {\it  Michigan Math. J.}
    {\bf   66}  (2017),  no. 1, 3-20.

\bibitem{WaZ}{\sc   Wan, D.   and Zhang, W. J.},  Quasicontinuity and maximality of quaternionic plurisubharmonic functions,  {\it J.
    Math. Anal. Appl.}  {\bf  424} (2015),  no. 1, 86-103.

\bibitem{WaWang}{\sc   Wan, D.   and Wang, W.}, On the quaternionic Monge-Amp\`ere operator, closed positive currents and Lelong-Jensen
    type formula on
    quaternionic space,    {\it Bull.  Sci. Math.} {\bf 141}  (2017), no. 4,  267-311.


\bibitem  {WR} {\sc
Wang, H. Y. and   Ren, G. B.}, {
Bochner-Martinelli formula for $k$-Cauchy-Fueter operator},  {\it J. Geom. Phys.}  {\bf 84}  (2014), 43-54.




\bibitem{Wa10}{\sc Wang, W.},
The $k$-Cauchy-Fueter complexes, Penrose transformation  and Hartogs' phenomenon  for
quaternionic $k$-regular functions,   {\it J. Geom. Phys.} {\bf 60}, (2010), 513-530.

\bibitem
{G2} {\sc Wang, W.},
On twistor transformations and invariant differential operator of simple Lie group $ G_{2(2)}$, {\it
J. Math. Phys.} {\bf 54} (2013), 013502.

\bibitem
{L2} {\sc Wang, W.}, On the weighted $L^2$   estimate for the $k$-Cauchy-Fueter operator and the weighted $k$-Bergman kernel,   {\it  J.
Math. Anal. Appl.} {\bf  452}  (2017),  no. 1, 685-707.



\bibitem{Wa08}{\sc Wang, W.}, The Neumann problem for the $k$-Cauchy-Fueter complexes over $k$-pseudoconvex domains in $\mathbb{R}^4$
    and the $L^2$ estimate,  to appear in {\it J. Geom. Anal.}, arXiv:1704.02856.

\bibitem{We}{\sc Wells, R.},  {\it  Differential analysis on complex
manifolds},  Graduate Texts in Mathematics  {\bf 65},
Springer-Verlag, New York-Berlin, 1980.


\end{thebibliography}
\end{document}